\documentclass{article}
\usepackage[utf8]{inputenc}
\usepackage{bbm}
\usepackage{multirow}
\usepackage{lineno}

\usepackage[overload]{empheq}
\usepackage{mathrsfs}
\usepackage{amsfonts}
\usepackage{amsmath}
\usepackage{amssymb}
\usepackage{dsfont}
\usepackage{diagbox}
\usepackage{graphicx}

\usepackage{tikz}
\usepackage{caption}
\usepackage{subcaption}
\usepackage{placeins, microtype}
\usepackage{latexsym,amsthm,xcolor,graphicx, longtable, tabu}
\usepackage{hyperref}
\usepackage{natbib}
\usepackage{mathtools}
\mathtoolsset{showonlyrefs}

\makeatletter
\def\thm@space@setup{
  \thm@preskip=0.4cm 
  \thm@postskip=\thm@preskip 
}
\makeatother

\newtheorem{thm}{Theorem}[section]
\newtheorem{definition}[thm]{Definition}

\newtheorem{assumption}[thm]{Assumption}
\newtheorem{lemma}[thm]{Lemma}

\newtheorem{remark}[thm]{Remark}
\newtheorem{proposition}[thm]{Proposition}
\newtheorem{corollary}[thm]{Corollary}
\newtheorem{example}[thm]{Example}

\newcommand\Z{\mathbb{Z}}
\newcommand\N{\mathbb{N}}
\newcommand\R{\mathbb{R}}
\newcommand\PP{\mathbb{P}}
\newcommand{\En}{\mathcal{E}_n}
\newcommand{\Enfty}{\mathcal{E}_\infty}
\newcommand{\Sn}{\mathcal{S}_n}

\newcommand{\toL}{\,{\buildrel {d} \over \longrightarrow}\,}

\newcommand{\towd}{\,{\buildrel {wd \;} \over \longrightarrow}\,}

\newcommand{\lp}[1]{\left(#1\right)}
\newcommand{\lb}[1]{\left[#1\right]}
\newcommand{\la}[1]{\left|#1\right|}
\newcommand{\lcb}[1]{\left\{#1\right\}}
\newcommand{\qv}[1]{\left\langle#1\right\rangle}
\newcommand{\norm}[1]{\left|\left|#1\right|\right| }

\long\def\metanote#1#2{{\color{#1}\
\ifmmode\hbox\fi{\sffamily\mdseries\upshape [#2]}\ }}

\numberwithin{equation}{section}
\newcounter{keepeqno}
\usepackage{fancyhdr}
\title{Quenched coalescent for diploid population models with selfing
and overlapping generations}

\author{
Louis Wai-Tong Fan
\footnote{School of Data Science and Society, University of North Carolina, Chapel Hill, NC, USA}
\footnote{Department of Organismic and Evolutionary Biology, Harvard University, Cambridge, MA, USA}
\and
Maximillian Newman
\footnote{Department of Genetic Medicine, University of Chicago, Chicago, USA}
\and
John Wakeley%
\footnotemark[2]
}

\date{\today}

\begin{document}

\captionsetup{font=small,width=0.85\textwidth}

\maketitle

\abstract{ 
We introduce a general diploid population model with self-fertilization (or `selfing') and possible overlapping generations, and study the genealogy of a sample of $n$ genes  as the population size $N$ tends to infinity. Unlike the traditional approach in coalescent theory which considers the unconditional (annealed) law of the gene genealogies averaged over the population pedigree, here we study the conditional (quenched) law of gene genealogies given the pedigree. We focus on the case of high selfing probability and obtain that this conditional law 
converges  to a random probability measure. This measure is the random law of a system of coalescing random walks on an exchangeable fragmentation-coalescence process of \cite{berestycki04} which in general allows multiple mergers of ancestral lineages.  As a special case, it contains a system of coalescing random walks on an ancestral graph with binary mergers only, which is identical in structure to the previously described ancestral recombination graph and ancestral selection graph.  We use simulations of the ancestral process to show how the site-frequency spectrum of genetic data can depend strongly on the pedigree. 
Our convergence result is established by means of 
a characterization of weak convergence in distribution for random probability measures on 
Skorokhod spaces.
}

\section{Introduction}\label{S: intro}

Coalescent processes have been widely used as models of gene genealogies that describe the ancestral structure of a sample of $n$ genes, when the total population size $N$ is sufficiently large.
The Kingman coalescent \cite{kingman1982}, for instance, has been enormously impactful in the study of natural genetic variation in populations \cite{wakeley2009coalescent}. Its power stems from its remarkable robustness; indeed, a large number of population models were shown to have the Kingman coalescent or its variant as their scaling limit as $N$ tends to infinity \cite{Mohle1998a}.
Other  models such as  the coalescent with asynchronous multiple mergers  \cite{donnelly1999particle, sagitov1999general, pitman1999} and the coalescent with simultaneous multiple mergers \cite{MohleSagitov2003, schweinsberg, birkner2018coalescent}
have also been discovered as scaling limits under exchangeable models, when the number of offspring per individuals has very high variance, and have been applied to a number of different species \cite{FreundEtAl2023}. 
The simultaneous multiple-mergers coalescent, also called the $\Xi$-coalescent, is most relevant to this paper.

Traditionally, coalescent models are obtained by taking average over the population pedigree,  the graph that represents the total history of reproductive relationships in the population. Namely, implicit in the approach of papers in classical coalescent theory \cite{kingman1982, sagitov1999general, pitman1999, schweinsberg, birkner2018coalescent} is an annealing over all realizations of the pedigree to describe the distribution  of gene genealogies across unlinked loci in the genome. However, 
this tradition of averaging over the pedigree is questionable because there is only one population pedigree, and all genetic information across loci is passed through this same pedigree.
For example, in order for unlinked loci to have independent genealogies (as they must by definition), under this averaging they would also need to have independent pedigrees.  This is problematic because even unlinked loci are subject to the same pedigree. 

This conceptual flaw was not unrecognized \cite{ball1990,wakeleyetal2012,WakeleyEtAl2016,WiltonEtAl2017,ralph2019}, but  perhaps luckily, it turned out that the Kingman coalescent can still be applied under the standard assumptions of neutral coalescent model because then the conditional limit is still equal to the unconditional limit \cite{TyukinThesis2015}. However, when there are simultaneous multiple mergers, the conditional genealogy can be \textit{different from} the unconditional genealogy \cite{DFBW24,abfw25}.

\medskip
    
    We believe that, at least at the conceptual level, \textit{coalescent theory should start by conditioning on the pedigree}.
    This has been the basis of the mathematically rigorous works \cite{TyukinThesis2015, DFBW24, NFW25, abfw25} that  collectively marked the emerging \textit{quenched-coalescent theory}. We now give  a brief  account of these works before describing the main contribution of the present paper.

    In \cite{abfw25}, the authors study the diploid Cannings model introduced in \cite{birkner2018coalescent}, where selfing is excluded. This is a model in which the offspring distribution, described by the matrix $(V_{ij})$ of the number of offspring between all pairs of parents, is invariant under permutations of the parents. 
Under the same condition on the offspring distribution that guaranteed  annealed convergence to the $\Xi$-coalescent holds, 
    the authors in \cite{abfw25} showed that 
    the genealogies conditioned on the pedigree converge to an inhomogeneous $(\Psi,c_{\rm pair})$-coalescent, where $\Psi$ is a Poisson point process on $[0,\infty)\times (\Delta\setminus \{\bf 0\})$ with intensity $dt \otimes\frac{1}{\langle x,x \rangle}\Xi(dx)$. This limiting process consists of the independent superposition of a Kingman coalescent with rate $c_{\rm pair}=1-\Xi(\Delta\setminus \{\bf 0\})$ and the multiple merger events whose the timing and intensity are specified by $\Psi$. 
    The result in \cite{abfw25} generalizes earlier work in \cite{TyukinThesis2015} and \cite{DFBW24}.

    While the  model in \cite{birkner2018coalescent, abfw25} is quite general and captures a wide range of reproductive variance, it has two limitations from a biological perspective. 
    Firstly, it does not allow for selfing, which is an important evolutionary force 
found in  taxa including eukaryotic microbes and marine invertebrates \citep{SassonAndRyan2017,YadavEtAl2023} and common in plants \citep{AbbottAndGomes1989,HartfieldEtAl2017,TeterinaEtAl2023}.
    Selfing probabilities have a bi-modal distribution among plant species, with fewer species having $s\in(0.2,0.8)$  and some species reaching as high as 0.99 \cite{SchemskeAndLande1985,AbbottAndGomes1989,VoglerAndKalisz2001,BarriereAndFelix2005,GoodwillieEtAl2005,SellingerEtAl2020}.
    Secondly, the model assumes non-overlapping generations, a simplification that excludes biologically realistic scenarios in which parents may come from different age cohorts. Overlapping generations are common in many natural populations and can introduce temporal correlations and ancestral dependencies not captured by discrete-generation models.

In previous work \cite{NFW25} we considered a diploid Moran model with a high selfing probability $\alpha_N$, and obtained quenched limits for the coalescence time of a sample of size $n=2$. 
Conditioning coalescence on the pedigree revealed three markedly different behaviors, depending on how quickly $\alpha_N\to 1$.  The critical case, called `limited outcrossing' in \cite{NFW25}, is when $1-\alpha_N$ is of order $1/N$.  This will also be our focus here.

\medskip

This paper aims to contribute to building the emerging \textit{quenched-coalescent theory}, by going beyond the existing mathematical work \cite{TyukinThesis2015, DFBW24, NFW25, abfw25}. We introduce a general diploid model with overlapping generations that extends the haploid model of \cite{sargsyan2008coalescent} and the diploid models of \cite{birkner2018coalescent} and \cite{NFW25}. Our model explicitly incorporates a selfing rate parameter and includes the diploid Moran-type models with selfing in \cite{NFW25, coron2022pedigree, linder2009} as special cases. Our main result significantly strengthens the analysis in \cite{NFW25}, in the `limited outcrossing' regime, by considering arbitrary sample size 
    $n$ and a substantially broader class of models. Moreover, it complements the results of \cite{abfw25, TyukinThesis2015, DFBW24} by identifying a distinct class of limiting conditional coalescent processes that arises in the high selfing regime.

    The new class of coalescent models that arise as $N\to\infty$ is a family of coalescing random walks on a directed random graph, which we call  a $Q$-$\lambda$ graph in this paper. This random graph is a subgraph of
    an exchangeable fragmentation-coalescence process (EFC) introduced in  \cite{berestycki04}, where $Q$ is the rate matrix for coalescence and $\lambda$ is the fragmentation rate at which each node spits into two. The parameter $\lambda$ is the relative rate of outcrossing to pairwise coalescence as $N\to\infty$. In the regime where the classical (annealed) limit is a time-rescaled Kingman, we find that these $Q$-$\lambda$ graphs correspond to the ancestral graphs with binary mergers that have been used to describe genetic ancestries in the presence of recombination \cite{hudson1983, GriffithsAndMarjoram1997} and selection \cite{kroneneuhauser1997}.
    As in \cite{NFW25}, time is rescaled so that any given pair of sample lineages coalesce with rate $2$ in the annealed limit in this paper. This is slightly different from that in \cite{abfw25} (where a given pair coalesce with rate 1) and that in \cite{DFBW24} (which involves the parameter $\psi$); see Remarks \ref{Rk:Timescale1} and \ref{Rk:Timescale2} for detail.
    The proof of the main result (Theorem~\ref{T: quenched_limited_outcrossing}) involves a novel characterization of  weak convergence in distribution  of random probability measures on the Skorokhod space $\mathcal{D}\lp{\R_+, E}$ for any locally compact Polish space $E$, and a general theorem about quenched convergence of Markov processes on a suitably enriched space of partitions.     
    While tightness and convergence of (deterministic) probability measures on  Skorohod spaces have been studied (see \cite{bk10,kouritzin2016}) by means of  convergence determining and separating classes of functions, convergence of \textit{ random } probability measures on  Skorohod spaces are less explored.

\medskip
From the application standpoint,
    the main motivation of coalescent theory (and our quenched-coalescent theory) is to describe the patterns of genetic variation expected under various biological scenarios to provide frameworks for statistical inference about past events and processes affecting populations. 
    In a given population,  gene genealogies are tree structures which emerge from tracing the ancestral lines of these sample backwards in time until the most recent common ancestor (MRCA). Mutations occurred  in the past, along the ancestral lines, result in genetic diversity among the sample. 
    For example, the site-frequency spectrum (SFS) is a commonly used measure of genetic variation upon which statistical  inferences are based \cite{BustamanteEtAl2001,Achaz2009,EldonEtAl2015,LiuAndFu2015,GaoAndKeinan2016,FerrettiEtAl2017,FreundEtAl2023}.  For a sample of $n$ genomes, the SFS records the number of polymorphic sites where a mutant base is found in $r \in \{1,\ldots,n-1\}$ copies \cite{Tajima1989,BravermanEtAl1995,Fu1995}.  Because per-site mutation rates are typically very small and the number of sites is large, the SFS is taken to reflect the total length of branches in the gene genealogy with $r$ descendants in the sample, or which are ancestral to $r$ of the samples. See Figure \ref{fig:tau} for an illustration of these branches when $r=3$ and $n=10$. Different biological phenomena, such as population growth and natural selection, lead to different coalescent predictions for the SFS.
Our main result sheds new light on the SFS of genetic data, by specifying how SFS depends on the pedigree. More precisely, under the conditional coalescent, the SFS should be viewed as a conditional SFS given the single pedigree as a latent variable. In Figure \ref{fig:sfs}, we provide simulations for the conditional SFS given 5 different pedigrees and thus illustrate how the pedigree can impact the SFS of the data.

\medskip
\noindent
{\bf Organization of this paper. }
    Section~\ref{S:GeneralModel} introduces our diploid exchangeable model with selfing and overlapping generations. An annealed version of our main result, akin to classical results of coalescent theory such as those of \cite{kingman1982, Mohle1998a, pitman1999, sagitov1999general, schweinsberg}, is presented in Section~\ref{S: coalescent_description_general}. Section~\ref{S: main_result} presents the main quenched convergence result for the coalescent of our exchangeable model model conditional on the pedigree. In Section~\ref{S: applications} we present applications of our main convergence result and illustrate how the SFS can depend strongly on the pedigree. The remainder of the paper focuses on the tools and calculations necessary to prove the main results: weak convergence of random measures in distribution in Section~\ref{S: weak_convergence_criteria}, and a suitable notion of convergence of the random pedigree and continuity of the coalescent law's dependence thereon in Section~\ref{S: main_result_proof}.

\section{An exchangeable diploid model with selfing and overlapping generation}\label{S:GeneralModel}

    \subsection{The model}\label{SS: model}
    
    In this paper, we introduce a diploid, monoecious, panmictic (well-mixed and randomly mating) population of constant size $N$, evolving in discrete time-steps with overlapping generations. This model generalizes the diploid  model in \cite{abfw25,birkner2018coalescent} to allow overlapping generations and selfing, and  extends the haploid model in \cite{sargsyan2008coalescent} to diploid populations while supporting general offspring distributions.
    
    Precisely, 
    the population model is specified by 
    a deterministic number $\alpha_N\in [0,1]$ and 
    the joint distribution of a random variable $K_N$ and a random symmetric matrix $\textbf{V}= (V_{i,j})_{1\leq i,j\leq N}$.    
The number $\alpha_N$ represents the \textit{selfing probability} for each offspring, and $K_N$ and $\textbf{V}$ represent, at an arbitrary timestep,  the \textit{total offspring number} and the \textit{pairwise offspring numbers} respectively. 
We assume the following:  
        \begin{itemize}
        \item $K_N$ and all entries of $\textbf{V}$ take values in $\{0,1,2,\ldots,N\}$.
        \item The total number of offspring satisfies $\sum_{i\leq j}^N V_{i,j} = K_N$.
        \item The full matrix $\textbf{V}$ is exchangeable, i.e.,
        \[
        \left(V_{i,j}\right)_{1\leq i,j\leq N} \overset{d}{=} \left(V_{\sigma(i),\sigma(j)}\right)_{1\leq i,j\leq N}
        \]
         for any permutation $\sigma$ of $[N] = \{1,2,\ldots, N\}$.
    \end{itemize}

We consider discrete time-steps indexed by $k \in \mathbb{Z}_+ = \{0,1,2,\ldots\}$, where $k=0$ is the present, $k=1$ the previous time-step, and so on \textit{backward into the past}. Let $\big \{(K^{(k)}_N, \textbf{V}^{(k)})\big \}_{k \in \Z_+}$  be a sequence of i.i.d.\  random variables that have the same distribution as $(K_N,\textbf{V})$. Reproduction events in different time-steps are taken to be independent, and occur as follows:
    
For each $k\in\Z_+$, $K^{(k)}_N$ individuals  are chosen uniformly without replacement, among the $N$ individuals from the $k$-th time-step, to be offspring.   
To describe parentage, we suppose that the
individuals at each time-step are labeled by $[N]$. Let $V^{(k)}_{i,j}$ be the number of offspring in the $k$-th time step in the past, produced by the pair of individuals $(i,j)$ in the $(k+1)$-th time-step in the past. 
The $K_N^{(k)}$ individuals defined by the entries of $\textbf{V}^{(k)}$ constitute the new individuals in time-step $k$. Each individual has two sets of chromosomes, as it is a diploid population. Genetic lineages at an autosomal locus are transmitted according to Mendel's law of random segregation, which means each gene copy in the offspring chooses independently from the two gene copies in the parent from which the gene copy is inherited. 
    
The remaining $N - K_N^{(k)}$ individuals are carried over unchanged (without random segregation) from time-step $k+1$, which  can induce \textit{overlapping generations}. 
    
We further assume that 
each of the $K^{(k)}_N$ offspring independently chooses to have either a single parent (selfing) or   two distinct parents (outcrossing), with probabilities $\alpha_N$ and $1-\alpha_N$ respectively. Therefore, conditional on $K_N^{(k)}$, the number of selfed offspring $\sum_{i=1}^N V_{i,i}^{(k)}$ is  an 
 independent binomial random variable: $\mathrm{Bin}(K_N^{(k)}, \alpha_N)$.   In other words, we assume that $(K_N,\textbf{V})$ satisfies
 \[
 \sum_{i=1}^N V_{i,i}  \sim \mathrm{Bin}(K_N, \alpha_N) \qquad\text{given }K_N.
 \]

         \FloatBarrier
        \begin{figure}
            \centering
            \includegraphics[width=0.5\linewidth]{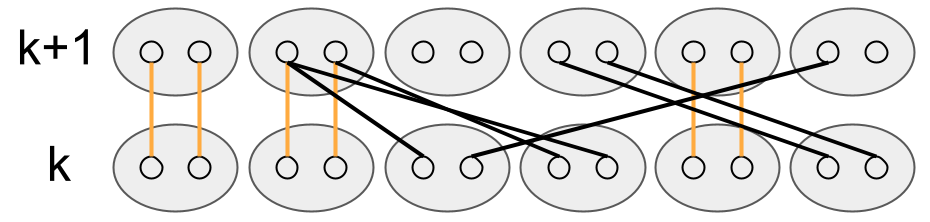}
            \caption{An illustration of our population model between time-steps $k+1$ and $k$ in the past, with size $N=6$. There are $K_N = 3$ offspring which are individuals 3,4,6 (from left to right) in timestep $k$. Bold edges represent reproductive relationships. Further, two of the offspring are reproduced by selfing. Individuals $1, 2, 5$, whose edges are marked in orange, simply persisted between consecutive time-steps and are not offspring. One can read  $(V_{i,j})$ from the pedigree. We see that $V_1 = 0$ as the first parent (the first individual in timestep $k+1$) has no offspring, $V_2 = 2$ as the second parent has two offspring (individuals 3 and 4 in timestep $k$). Also, $V_{2,2} =V_{4,4} = 1$.}
            \label{F: one_step}
        \end{figure}
        
See Figure~\ref{F: one_step} for  a realization of one time-step of this process when $N=6$.  

The population model described above includes two related but distinct classes of models in the literature as special cases.
    \begin{example}[The diploid Cannings model in \cite{birkner2018coalescent}]\rm
    When $K_N=N$ (non-overlapping generations) and $\alpha_N=0$ (no selfing),    
    our model  reduces to the  model introduced in \cite{birkner2018coalescent} and studied in \cite{abfw25}.
    \end{example}

    \begin{example}[A diploid Sargasyan-Wakeley model]\rm \label{Eg:SW}
        
    Our model also extends the model in \cite{sargsyan2008coalescent} by modeling diploid, not just haploid, individuals and by allowing a more general offspring distribution. 
    
    Suppose the random symmetric  matrices $\{V^{(k)}\}_{k\in\Z_+}$ satisfy the following extra assumption. Namely, their distribution is parametrized by a random number $P_N\in \{2,3\ldots,N\}$ representing the number of  \textit{potential} parents at a timestep.
        
    Let $\big \{(K^{(k)}_N,\,P^{(k)}_N)\big \}_{k \in \Z_+}$  be a sequence of i.i.d.\  random vectors that have the same distribution as
    $(K_N,\,P_N)$.
    For each $k\in\Z_+$, as before $K^{(k)}_N$ individuals  are chosen uniformly without replacement from the $k$-th time-step  to be offspring. Then, suppose that
\begin{enumerate}
    \item[(i)] $P^{(k)}_N$ individuals are chosen uniformly without replacement from the $(k+1)$-th time-step to be \textit{potential} parents of the $K^{(k)}_N$ offspring, and 
    \item[(ii)]  the actual parent(s) of the offspring is (are) chosen uniformly without replacement from the $P^{(k)}_N$ \textit{potential} parent(s).  
\end{enumerate}
             
            When $K_N=P_N=N$, this is the Wright-Fisher model with selfing considered in \cite{Mohle1998a, NordborgAndDonnelly1997} and
            \cite{kogan2025correlation} with free recombination ($r_N=1/2$) and $\alpha_N=s_N$.  
            The model in \cite{DFBW24} with $\psi=1$ is exactly the case when $K_N=N$ deterministically and $P_N$ is a random variable taking values $2$ and $N$ with probabilities $\lambda/N^{\theta}$ and $1-\lambda/N^{\theta}$ respectively.

    The model of \cite{NFW25} corresponds to the special case when $K_N=1$ deterministically. 
            When $\alpha_N = \alpha \in [0,1)$ for all $N$, this corresponds to the diploid Moran model with selfing considered in \cite{linder2009}. When $\alpha = 0$ this corresponds to the model as in \cite{coron2022pedigree}.
            The model in \cite{birkner2013} corresponds to our case when $\alpha_N=0$, and when $(K_N,P_N)$ is equal to $(1,2)$ and $([\psi N],2)$ with probabilities $1-\varepsilon_N$ and $\varepsilon_N$ respectively.
            There is a slight difference between our model and these previous models, which assume that the indices of the potential parents and those of the offspring are disjoint. 
    \end{example}
    
    \medskip

    To simplify notation, we shall omit the subscript $N$  when there is no confusion. For example, we will write  $(\alpha, K, K^{(k)})$ instead of $(\alpha_N,K_N, K^{(k)}_N)$. 

Let  $V_i:=\sum_{j=1}^N V_{i,j}$ be the total number of offspring for individual $i$, and $\tilde{V}_i:=V_{i,i} + V_i$. Then  $\sum_{i=1}^N \tilde{V}_i = 2K$ and the total genetic contribution of the $i$th individual to the next time step is
    \begin{equation}\label{E:Vi}
        \frac{\tilde{V}_i}{2N}:= \frac{V_{i,i} + V_i}{2N}
        = \frac{1}{2N}\Bigl(2V_{i,i} + \sum_{j \neq i} V_{i,j}\Bigr),
    \end{equation}
because the genetic contribution of a parent by selfing is twice that via outcrossed offspring.
    
    In the absence of selfing ($\alpha_N=0$), we have $\tilde V_i=V_i$, recovering the offspring-frequency weights of \cite{MohleSagitov2003} (see also \cite{birkner2018coalescent,abfw25}). These frequencies provide a convenient device to  express one-step coalescence probabilities (e.g.\ pairwise and triple mergers) as explicit polynomials in $\{\tilde{V}_i\}$ with Mendelian coefficients.
    
    \subsection{The pedigree as important latent variable of the population}

        The population dynamics described in Section~\ref{SS: model} give rise to a random directed graph $\mathcal{G}_N$ 
        that encodes the population pedigree, i.e. the reproductive relationships among all individuals. We give a formal definition below and offer the left panel of Figure \ref{F: several_realizations} as an illustration.   

        Each individual is diploid and carries two copies of each autosomal locus. Genetic lineages are transmitted through the population pedigree via \textbf{Mendelian segregation}: each gene copy in the offspring randomly inherits one allele from the corresponding gene copies in the designated parent (or one of the two gene copies in the case of selfing). The pedigree is described explicitly in Definition~\ref{D: pedigree}.

        \FloatBarrier

        \begin{figure}
            \centering
            \includegraphics[width=0.5\linewidth]{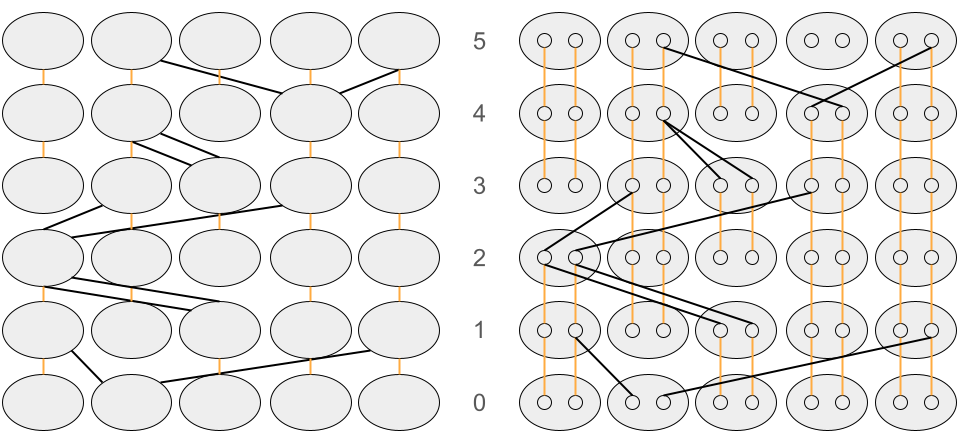}
            \caption{(Full single-locus population process on the right and the corresponding pedigree on the left). A realization of our population process with $N=5$ individuals from time-steps $k=0$ to $5$ in the past (right panel). The corresponding pedigree is shown on the left. Here we consider the Moran model in \cite{NFW25}, where $K_N=1$ deterministically. The black lines correspond to reproductive relationships while the yellow lines correspond to an individual persisting from one time-step to the next. The yellow lines give rise to overlapping generations.}
            \label{F: several_realizations}
        \end{figure}

        \begin{definition}[Pedigree] \label{D: pedigree}
        The \textit{population pedigree}, or simply the \textit{pedigree}, refers to  an undirected multi-graph 
        $\mathcal{G}_N$ with vertex set $[N] \times \Z_+$, where a vertex $(\ell,k)$ represents the individual with index $\ell$ at the $k$-th time-step.  The edges of $\mathcal{G}_N$ are between vertices in consecutive time-steps.
        For each $k\in\Z_+$, there are two edges  connecting each of the $K^{(k)}_N$ offspring with its parent(s). If the offspring has two distinct parents (outcrossing), then there is one edge to each parent; if the offspring has only one parent (selfing), then there are two edges  connecting  the offspring to its parent (hence $\mathcal{G}_N$ is a multigraph).  
        For each of the remaining $N-K^{(k)}_N$ individuals, there is a single edge connecting to itself in the consecutive time-steps. 
        \end{definition}

        The importance of the pedigree, as described in greater detail in the section ``Previous Work on Pedigrees" of \cite{DFBW24}, is that it encapsulates population dynamics that are common to every single locus on the genome. Even loci extremely far apart are coupled by its dynamics. In particular, gene genealogies far apart on the genome are described by conditionally independent realizations of the ancestral process with respect to the pedigree. Gene genealogies, even conditional on the pedigree, are stochastic due to Mendelian randomness. This is demonstrated in Figure~\ref{F: same_pedigree_different_genealogies}. Classical coalescent theory implicitly averages over realizations of the pedigree to determine the \textit{average} gene genealogy. However, simulations have demonstrated ways in which the pedigree can affect the structure of gene genealogies \cite{ball1990,wakeleyetal2012,WakeleyEtAl2016}.  Further, mathematically rigorous works have begun to identify new scaling limits of coalescent processes conditional on the pedigree \cite{TyukinThesis2015, DFBW24, NFW25, abfw25}.

        \begin{figure}
            \centering
            \includegraphics[width=0.5\linewidth]{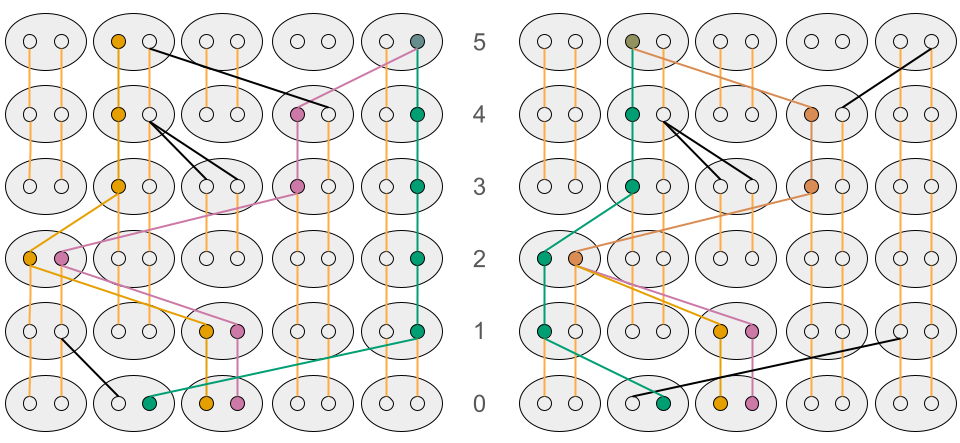}
            \caption{(Same pedigree but different genealogies). In the figure we see two different genealogical histories. We focus on a sample of $n=3$ lineages and trace their history backwards in time. Both histories are subject to the same pedigree, that displayed on the right of Figure~\ref{F: several_realizations}, and yet the history of the sample lineages backwards in time are distinct.}
            \label{F: same_pedigree_different_genealogies}
        \end{figure}

\section{Ancestral lines and the coalescent process for a sample}\label{S: coalescent_description_general}

Next, we let $n \in \{2,3,\ldots\}$ be the sample size and consider the lineages of $n$ sampled gene copies under our model. We further supposed that these $n$ distinct gene copies come from $n$ \textit{distinct individuals}, one gene copy from each, so that $n \leq N$ whereas in total there are $2N$ gene copies available in the population. Our analysis of the ancestral process of the sample relies on this assumption that ancestral lineages begin in different individuals. Remark~\ref{Rk:sameIndiv} shows how our results may be extended to other types of samples, given our focus on high probabilities of selfing ($\alpha_N \to 1$). In what follows, we use $X$ to denote labels of gene copies and $\hat{X}$ to denote labels of individuals containing them.

    \subsection{The ancestral partition process for a sample of size \texorpdfstring{$n$}{}}

For $k\in\Z_+$ we let $X_i^N(k)$ be the gene ancestral to $X_i(0)$  $k$ time-steps in the past. Under our model, we can write
        \begin{equation}\label{E:Alines}
            X_i^N(k) := \lp{M_i^N(k), \hat{X}_i^N(k)} \in \{0,1\} \times [N], \quad i \in [n], 
        \end{equation}
        where $\hat{X}_i^N(k)$ is the  individual in which the gene $ X_i^N(k)$ resides, and 
        \begin{equation*}
            M_i^N(k+1) =
            \begin{cases}
                \text{is an independent Bernoulli$\lp{\frac{1}{2}}$ random variable}&, \text{ if $\hat{X}_i^N(k)$ is an offspring}\\
                M_i^N(k) &, \text{ if $\hat{X}_i^N(k)$ is not an offspring}
            \end{cases}.
        \end{equation*}
Given the pedigree between time-steps $k+1$ and $k$, $\hat{X}_i^N(k+1)$ is chosen uniformly among the vertices in time-step $k+1$ to which $\hat{X}_i^N(k)$ is adjacent.

        For any given sample of genes at $t=0$ and for any given pedigree, there generally are many possible ways in which to trace these samples backwards in time. This is illustrated in Figure~\ref{F: same_pedigree_different_genealogies}, where the (labeled) tree structure of  $n=3$ ancestral lines (genetic lineages) backwards in time is evidently not fully determined by the pedigree.

        \begin{definition}[Ancestral line]
        For each $1\leq i\leq n$, the process   $X_i^N = (X_{i}^N(k))_{k\in\Z_+}$ is called the \textbf{ancestral line} of the $i$-th  sampled gene $X_{i}(0)$. 
        \end{definition}

        Each process $X_i^N$ is a discrete-time Markov chain taking values in  $\{0,1\}\times [N]$. Furthermore, these $n$ processes 
        $\left(X_{i}^N\right)_{1 \leq i \leq n}$
        are correlated Markov chains that form a family of coalescing random walks on $\{0,1\} \times [N]$. In particular, for all $i,j \in [n]$,
 $X_{i}^N(k+1) = X_{j}^N(k+1)$ whenever $X_{i}^N(k) = X_{j}^N(k)$.
        See Figure~\ref{F: same_pedigree_different_genealogies} for an illustration.
        If we condition on the pedigree, then these random walks are conditionally independent, and the transition probabilities conditional on the pedigree is given by the Mendelian randomness. 

        \medskip

The size of the state space of the family of coalescing random walks, namely  $(\{0,1\} \times [N])^n$, is of order $N^n$. It is customary to reduce the state space 
by ignoring the indices of the ancestral individuals and only keep track of the coalescence of sample genealogies as we go backward in time. This is done via partitions 
of $[n]$, or equivalently, equivalence relations on $[n]$.
For $1 \leqslant i, j \leqslant n$,
$k \in \Z_+$, we write $i \sim_k j$ if and only if samples $i$ and
$j$ descend from the same chromosome $k$ time-steps ago, i.e.,
\begin{equation}
  \label{eq:xinNm.alt}
  i \sim_k j \quad \Longleftrightarrow
  \quad X_i(k) = X_j(k). 
\end{equation}

This way, one obtains a stochastic process (called a coalescent process) with state space $\mathcal{E}_n$, the space of partitions of $[n]$.     
The initial state of this coalescent process is the partition into singletons $\xi_0^n:=\{\{i\}\}_{1 \leq i \leq n}$, since we sampled $n$ distinct gene copies, and will end up being the partition into a single block $\mathbf{1}_n := \{[n]\}$ which is the most recent common ancestor (MCRA) of the sample. 

For our diploid population, we need to account for which ancestral individuals contain two ancestral gene copies.
For this, we use notation from~\cite{MohleSagitov2003} and define the ancestral process on the state space
\begin{equation*}
  \mathcal{S}_n= \big \{ \left\{\left(C_1,C_2\right),\ldots,\left(C_{2x-1},C_{2x}\right),C_{2x+1},\ldots,C_b\right\} :
    b \in [n], 1 \leqslant x \in \lfloor {b}/{2} \rfloor, \{C_1, \ldots, C_b\} \in \mathcal{E}_{n} 
\big \},
\end{equation*}
where $\lfloor x\rfloor$ is the largest integer less than or equal
to $x$. We equip both spaces $\mathcal{E}_n$ and
$\mathcal{S}_n$ with the discrete topology.
Hence, each element $\xi\in \mathcal{S}_n$ is of the form
        \begin{equation}\label{E: xi_definition_sn}
            \{(C_1, C_2), (C_3, C_4), \ldots, (C_{2x-1}, C_{2x}), C_{2x+1}, \ldots, C_b\},
        \end{equation}
        where $\{C_1, C_2,\ldots, C_b\}\in\mathcal{E}_n$. We let  $x := \norm{\xi}$ be the number of individuals in a population that contain two sample lineages, and $b=|\xi|$ be the number of lineages (or blocks) remaining in a sample.  Note that $\mathcal{E}_n$ corresponds to a subset of $\mathcal{S}_n$. Clearly, $\xi\in \mathcal{E}_n$ if and only if $\norm{\xi}=0$.

        \begin{definition}[ancestral process]\label{D: ancestral}
            We define an $\mathcal{S}_n$-valued stochastic  process
            $\chi^{N,n}=(\chi^{N,n}(k))_{k\in\Z_+}$ as follows: 
            for $k \in \Z_+$, 
\begin{itemize}
    \item             $i$ and $j$ are in the same block in $\chi^{N,n}(k)$ if and only if  $X_i(k) = X_j(k)$, i.e. the ancestral lines of the $i$-th and the $j$-th samples coalesced  $k$ time-steps in the past, and
        \item             two blocks are in a set together in $\chi^{N,n}(k)$ if and only if  $\hat{X}_i(k) = \hat{X}_j(k)$, i.e.  the two ancestral lineages corresponding to these blocks are in the same individual in the population $k$ time-steps in the past.
\end{itemize}
We call this process $\chi^{N,n}$ the \textbf{ancestral process} of the sample.
        \end{definition}

We will investigate the convergence of a time-rescaling of $\chi^{N,n}$ with a random time change $S$ defined by
        \begin{equation}\label{E:TimeChange}
            S(k) := \sup \lcb{l \in \{0,1,2,\ldots,k\}:\, \chi^{N,n}(l) \in \mathcal{E}_n },
        \end{equation}
in other words the first time-step in the past that the ancestral lineages are not all in distinct individuals. Recall that the sample begins in this totally dispersed state by assumption: $\chi^{N,n}(0) \in \mathcal{E}_n$ in the present notation.  Thus $S(k)$ is well-defined, as the time for the sample to leave the subset $\mathcal{E}_n$ of $\mathcal{S}_n$.

We define 
        $\widetilde{\chi}^{N,n}$ to be the random time-change of $\chi^{N,n}$ as follows:
        \begin{equation*}
            \widetilde{\chi}^{N,n}(k) := \chi^{N,n}(S(k)).
        \end{equation*}    
The $\mathcal{E}_n$-valued process $\widetilde{\chi}^{N,n}$ is non-Markovian. However, in the limit $N\to\infty$, the random time-change \eqref{E:TimeChange} not only keeps the process to live in the state space $\mathcal{E}_n$ for which our convergence method works, but also prevents any accumulation of jumps of our rescaled process; see Remark \ref{Rk:randomtimechange}.

\subsection{An annealed scaling limit}
    
To connect to classical results in coalescent theory, including \cite{kingman1982, Mohle1998a, pitman1999, sagitov1999general, schweinsberg,birkner2018coalescent}, we shall consider the  unconditional (a.k.a. annealed) law, one that averages over the random pedigree, of the ancestral process. 

    Let $c_N$ 
    be the probability that two randomly chosen genes \textit{from two distinct
            individuals} coalesce in one time-step in the past. Then
    \begin{align}
    c_N = & \mathbb{E}\left[\frac{(N-K)K}{N^2\,(N-1)}\right]
    \;+\;
    \frac{1}{\binom{N}{2}}\,
    \mathbb{E}\!\left[
    \sum_{i=1}^N\!\left(
    \frac{1}{2}\binom{V_{i,i}}{2}
    +\frac{1}{4}\,V_{i,i}\!\sum_{j\neq i}V_{i,j}
    +\frac{1}{8}\binom{\sum_{j\neq i}V_{i,j}}{2}
    \right)
    \right] \notag\\
    =& \mathbb{E}\left[\frac{(N-K)K}{N^2\,(N-1)}\right]
    \;+\;
    \frac{1}{\binom{N}{2}}\,
    \mathbb{E}\!\left[
    \frac{1}{16}\sum_{i=1}^N \bigl(\tilde V_i^{\,2}-\tilde V_i\bigr)
    -\frac{\sum_{i=1}^N V_{i,i}}{8}
    \right], \label{E:cN_general}
    \end{align} 
    where the first term comes from one newborn and one carryover, and the second term comes from two newborns. This formula generalizes that in 
    \cite[Equation 1.4]{birkner2018coalescent}.

Convergence will be established for the  time-rescaled ancestral process  $\lp{\bar{\chi}^{N,n}(t)}_{t \in \R_+}$ defined by
        \begin{equation}\label{E: bar_definition}
            \bar{\chi}^{N,n}(t) = \widetilde{\chi}^{N,n}\lp{\lfloor t c_N^{-1} \rfloor}= \chi^{N,n}\lp{S\lfloor t c_N^{-1} \rfloor}.
        \end{equation}
We can and will consider this process $\bar{\chi}^{N,n} = \lp{\bar{\chi}^{N,n}(t)}_{t \in \R_+}$ as a random variable taking value in the Skorokhod space $\mathcal{D}\lp{\R_+, \En}$ equipped with the $J_1$ topology (see \cite{ethier2009markov}).

\begin{remark}\label{Rk:randomtimechange}\rm
    A main purpose of the random time-change $S$ is to avoid an accumulation of jumps (as $N\to\infty$) that would prevent convergence of the annealed coalescent in the $J_1$ topology. 
    Because selfing can coalesce distinct blocks in the same individual on a time-scale much shorter than the coalescence time-scale, 
    in the absence of this random time-change $S$ there will be an accumulation of jumps as $N\to\infty$. 
    Concretely, consider the state $\xi = \{(\{1\}, \{2\}), (\{3\}, \{4\})\}$. Because we allow for selfing, the process   $\left(\chi^{N,n}\lp{\lfloor t c_N^{-1} \rfloor}\right)_{t\in\R_+}$, beginning in state $\xi$, could transition to the state $\{(\{1\}, \{2\}), \{3,4\}\}$ in one time-step, and then $\{\{1\}, \{2\}, \{3,4\}\}$ in the next time-step. These two jumps, close together in the $c_N$ time-rescaling, would prevent convergence in $J_1$.
    
    The issue of accumulation of jumps was dealt with in \cite{birkner2018coalescent} by considering the complete dispersal map, which views the blocks in the same individual as distinct. This works in \cite{birkner2018coalescent} because there is no selfing there, and so with probability tending to $1$ the blocks will all disperse before there is any additional overlap of blocks in the same individual.

\end{remark}

We introduce three assumptions on the asymptotic behaviors of our model as $N\to\infty$ below. The first assumption ensures a continuous-time limiting model.
        
    \begin{assumption}\label{A:c_N}
    Suppose $\lim_{N\to\infty} c_N=0$, or equivalently, $\lim_{N\to\infty} \frac{\mathbb{E}[V_1^2]}{N}=0$ where  $V_i:=\sum_{j=1}^N V_{i,j}$ is the total number of offspring for individual $i$.
    \end{assumption}
    
Asymptotically, one unit of (continuous) time corresponds to $\lfloor c_N^{-1} \rfloor$ time-steps. 
To quantify the time-scale on which distinct sample lineages in the same individual disperse, we let $d_N$ denote the probability that, in a single time-step in the past,  two sample lineages in the same individual would disperse into two different individuals. Precisely,
        \begin{equation}\label{E: d_N_Def}
           d_N := \PP\lp{\chi^{N,2}(1) = \Big\{\{1\},\{2\}\Big\} \mid \chi^{N,2}(0) = \Big\{ (\{1\},\{2\} ) \Big\}} 
           = \frac{1-\alpha_N}{N}\mathbb{E}\lb{K_N}.
        \end{equation}
        We refer to $d_N^{-1}$ as the dispersal or outcrossing timescale.

The present work focuses on the asymptotic regime where the outcrossing timescale is comparable to, or slower than, the coalescence timescale. We refer this as the
\textit{limited outcrossing regime}, which generalizes the setting of the Moran model studied in \cite{NFW25}. This regime is formalized in Assumption \ref{A: timescale} below. 
         \begin{assumption}\label{A: timescale}
            We assume that $\alpha_N\to 1$ and $d_N c_N^{-1} \to \lambda \in \R_+$ as $N\to\infty$. 
        \end{assumption}
Assumption \ref{A: timescale} together with
\eqref{E: d_N_Def} implies that 
\begin{equation} \label{E:alpha to 1}
1-\alpha_N 
\;=\; \frac{N\,c_N}{\mathbb{E}[K_N]}\,\lambda\,(1+o(1))\to 0
\end{equation}
as $N\to\infty$. Hence,  $\alpha_N\to 1$ at a rate proportional to $ \frac{N\,c_N}{\mathbb{E}[K_N]}$ when $\lambda>0$,

\begin{remark}\rm
It is notable that the assumption  $d_N c_N^{-1} \to \lambda \in \R_+$  together with Assumption \ref{A:c_N} implies that the selfing probability $\alpha_N$ tends to 1 in many cases including the case when $K_N=N$ (e.g. Wright-Fisher model) and $K_N=O(1)$ (e.g. Moran model). Indeed, 
\eqref{E: d_N_Def} and the assumption $d_N c_N^{-1} \to \lambda$  imply the equality in \eqref{E:alpha to 1}.
Assumption \ref{A:c_N} then implies that $\alpha_N \to 1$ 
if $\mathbb{E}[K]$ is of order $N$. On other hand, note that \eqref{E:cN_general} and the fact $\sum_{i=1}^N \tilde{V}_i = 2K$ implies that $c_N\leq C\,\frac{\mathbb{E}[K+K^2]}{N^2}$ for some constant $C$ independent of $N$. Hence $\alpha_N \to 1$ if $\frac{\mathbb{E}[K^2]}{N\mathbb{E}[K]}\to 0$. 
\end{remark}

\medskip

Let $F_{\rm hap}: \Sn \to \En$ denote the \textbf{haploid map} defined as follows: for
        \begin{equation*}
            \xi = \{(C_1, C_2), \ldots, (C_{2x-1}, C_{2x}), C_{2x+1}, \ldots, C_{b}\}
        \end{equation*}
        we define
        \begin{equation}\label{E: hap_map}
            F_{\rm hap}\lp{\xi} := \{C_1 \cup C_2, \ldots, C_{2x-1}\cup C_{2x}, C_{2x+1}, \ldots, C_b \}.
        \end{equation}        
This map announces that lineages in the same individuals have coalesced, essentially reducing a diploid model to a haploid model. This map is different from the ``complete dispersal" map ${\rm cd}: \Sn \to \En$ used in \cite{birkner2018coalescent,abfw25}.

Our third and last assumption is  on the generator of the ancestral process. Let 
        \begin{equation}\label{Def:onestep}
    \mathfrak{p}_{\xi\eta}^{N,n} := \PP\lp{\chi^{N,n}(1) = \eta \mid \chi^{N,n}(0) = \xi}
        \end{equation}
        be the one-step transition probabilities  for the ancestral process $\chi^{N,n}$, and
define the matrix $\mathit{H}_{N,n} = \lp{h_{\xi\eta}^N}_{\xi, \eta \in \En}$ by
        \begin{equation*}
            h_{\xi\eta}^N :=  \PP\lp{  F_{\rm hap}\left(\chi^{N,n}(1)\right) = \eta \,\Big|\, \chi^{N,n}(0) = \xi} =\sum_{\zeta \in \Sn :\, F_{\rm hap}(\zeta) = \eta} \mathfrak{p}_{\xi\zeta}^N.
        \end{equation*}

        \begin{assumption}\label{A:Q_Nn}
            For any integer $n\geq 2$, the limit
            \begin{equation*}
              Q_n:= \lim_{N\to\infty}  \frac{1}{c_N}\lp{\mathit{H}_{N,n} - I}
            \end{equation*}
            exists as a real-valued matrix  $Q_n=(q^n(\xi,\eta))_{\xi,\eta\in \En}$. Furthermore, $\lp{Q_n}_{n \geq 2}$ is a consistent family, meaning that  $Q_n=Q_m\circ \varpi_{n,m}$ for any $2\leq m < n$,
            where $\varpi_{n,m}: \En \to \mathcal{E}_m$ is the projection obtained by restricting the partition to $[m]$.
        \end{assumption}

        To describe the annealed limit, we also describe an $n$-$\Xi$-coalescent (also called $n$-ksi-coalescent governed by $\Xi$ in this paper) following \cite[P.4]{birkner2018coalescent}.
        \begin{definition}\label{D: n-xi}
            Let $\Xi$ be a finite measure on the infinite ordered simplex 
            $$\Delta:=\left\{x=(x_1,x_2,\ldots) \in [0,1]^\infty: x_1 \geq x_2 \geq \ldots, \sum_{i=1}^{\infty} x_i \leq 1 \right\}.$$ 
             An $n$-$\Xi$-coalescent  is a continuous-time Markov process $\chi^n = \lp{\chi^n(t)}_{t \in \R_+}$ taking values in $\En$, whose transition rates are invariant under permutation and are described as follows: For any $\xi$ in $\En$  consisting of $b$ blocks, suppose that $\eta$ is an element of $\mathcal{E}_n$ obtained from $\xi$ by keeping $s$ of the blocks of $\xi$ the same, and coalescing the remaining $b-s$ blocks into $r$ blocks of sizes $k_1, k_2, \ldots, k_r$ respectively. The transition rate from $\xi$ to $\eta$ is given by
            \begin{equation*}
            \lambda_{b;k_1,\ldots, k_r;s}:=
                \mathbbm{1}_{(r,s) = (1,b-2)} \Xi(\mathbf{0})
                + 
                \int_{\Delta \setminus\{\mathbf{0}\}} \sum_{l=0}^s \sum_{i_1,\ldots, i_{r+l} \text{ distinct}}^{\infty} \binom{s}{l}\prod_{m = 1}^{r}x_{i_m}^{k_m} \,\cdot \prod_{w=r+1}^{r+l}x_{i_w} \cdot\, (1-\la{x})^{s-l}\frac{\Xi(dx)}{\left\langle x, x\right\rangle},
            \end{equation*}
            where $\left\langle x,x \right\rangle = \sum_i x_i^2$ and $\la{x} = \sum_i x_i$.
        \end{definition}
        Informally, an $n$-$\Xi$-coalescent can be described in terms of interval partitions. Take $x\in\Delta \setminus \{\mathbf{0}\}$ at rate $\frac{1}{\left\langle x, x \right\rangle} \Xi$ and perform an \textit{$x$-merger}. Additionally, each pair of blocks coalesces independently at rate $\Xi(\mathbf{0})$. In particular, any given fixed pair of lineages coalesce with rate
$$\lambda_{2;1;0} =\Xi(\mathbf{0})+ \Xi(\Delta \setminus\{\mathbf{0}\})=\Xi(\Delta).$$
Writing $x=(x_1,x_2,\ldots)$, an  $x$-merger can be roughly described as follows: as $\sum_i x_i \leq 1$, $x$ decides an interval partition. We throw each block of the coalescent uniformly at random on the unit interval. If they belong to the same element of the partition, and if they are at most $\sum_i x_i$, then they coalesce.

        \begin{remark}\rm\label{R: Qholds}\rm  
We shall show in Lemma \ref{L: ordered_offspring_distribution} that, under Assumption~\ref{A:c_N} and a usual condition on the ordered offspring distribution, we have that Assumption \ref{A:Q_Nn} holds and that for each $n\geq 2$,
 $Q_n$ is the generator of an $n$-$\Xi$-coalescent; i.e.
$\lambda_{b;k_1,\ldots,k_r;s} = q^n(\xi, \eta)$
            whenever $\xi, \eta\in \mathcal{E}_n$ and  $\lambda_{b;k_1, \ldots, k_r;s}$ are as in Definition \ref{D: n-xi}. 
    \end{remark}

        \begin{remark}\rm\label{R: xi_q_connection}\rm  
Assumption~\ref{A:Q_Nn}  implies that there is a unique finite measure $\Xi$ on $\Delta$ for which $Q_n$ is the infinitesimal generator of an $n$-$\Xi$-coalescent governed by $\Xi$ for all $n\geq 2$. This is the content of \cite[Proposition 3]{berestycki04}, which shows that the generator $C$ of the coalescence part of an exchangeable fragmentation-coalescence (EFC) process is completely determined by the consistent family $\lp{Q_n}$ of its finite-dimensional restrictions, each $Q_n$ encoding the jump rates on $\En$. Moreover, \cite[Proposition 3]{berestycki04} proves that any such consistent and exchangeable collection of generators $\lp{Q_n}_{n\geq 2}$ arises as the finite-dimensional projections of an EFC process taking values in $\Enfty$. In the purely coalescent case (i.e. no fragmentation), this implies that any consistent family of $n$-$\Xi$-coalescents can be realized as the projections of an infinite $\Enfty$-valued Markov process with infinitesimal generator $Q$ determined by $\Xi$.
    \end{remark}

        We can now describe the convergence of the unconditional law of the time-rescaled ancestral process $\lp{\bar{\chi}^{N,n}(t)}_{t \in \R_+}$ defined by \eqref{E: bar_definition}, 
        under the model presented in Section~\ref{S:GeneralModel}. 
        \begin{thm}[Annealed convergence]\label{T: annealed_limit}
            Suppose that Assumptions~\ref{A:c_N},  \ref{A: timescale} and \ref{A:Q_Nn} hold. Then there is a unique finite measure $\Xi$ on $\Delta$, made precise in Remark~\ref{R: xi_q_connection}, such that $\bar{\chi}^{N,n}$, as a $\mathcal{D}\lp{\R_+, \En}$-valued random variable, converges in distribution to an $n$-$\Xi$-coalescent  governed by $\Xi$ as $N\to\infty$.
        \end{thm}

    The proof of Theorem~\ref{T: annealed_limit} is postponed to Section~\ref{S: annealed_proof}, as it will follow from results in Section~\ref{S: main_result}. 
        
\begin{remark}\rm \label{Rk:Timescale1}
            Theorem \ref{T: annealed_limit} complements  \cite[Theorem 1.1]{birkner2018coalescent} by providing the high-selfing analogue. 
Note that in \cite{birkner2018coalescent, abfw25} where selfing is excluded, the $\Xi$ arising in the annealed limit is such that $\lambda_{2;1;0}=1$. In this paper, however, due to high selfing rate ($\alpha_N\to1$), the measure $\Xi$ in Theorem~\ref{T: annealed_limit} satisfies $\lambda_{2;1;0}=2$, i.e.
        \begin{equation}\label{E:TotalTimescale}
        2 =\Xi(\mathbf{0})+ \Xi(\Delta \setminus\{\mathbf{0}\}) = \Xi(\Delta).
        \end{equation}
This normalization was also adopted in \cite{NFW25}.
\end{remark}

\begin{remark}\rm \label{Rk:Timescale2}
That Theorem~\ref{T: annealed_limit} is not extended to the case $d_N c_N^{-1} \to \infty$ is not due to any particular difficulty with the annealed proof. Rather, the tools for the quenched proof for the the case $d_N c_N^{-1} \to \infty$ are sufficiently different from those required for the present work. Hence we leave this case to future work. 
For the simple case of $n=2$ samples from two different individuals, if $\alpha_N\to \alpha \in [0,1]$, one can check as in \cite[Theorem 3.1]{NFW25}  that the coalescence time $\tau^{N,2}$ satisfies that $\tau^{N,2}\,c_N$ converges in distribution to an exponential random variable with rate $\frac{2}{2-\alpha}$ (i.e. $\lambda_{2;1;0}=\frac{2}{2-\alpha}$). This is consistent with \eqref{E:TotalTimescale} when $\alpha=1$.
        \end{remark}

\section{Quenched convergence of the ancestral process}\label{S: main_result}

         Recall the pedigree $\mathcal{G}_N$ from Definition~\ref{D: pedigree} and define the $\sigma$-algebra $\mathcal{A}_N$  generated by the pedigree $\mathcal{G}_N$ and the labels of the $n$ distinct individuals from whom we have sampled, i.e.
        \begin{equation*}
            \mathcal{A}_N := \sigma(\mathcal{G}_N, \hat{X}_i^N(0): 1 \leq i \leq n).
        \end{equation*}
        We denote the quenched law of the time-rescaled ancestral process \eqref{E: bar_definition} by 
        \begin{equation}\label{Def:LawCondAn}
           \mathcal{L}^{N,n} := \PP\lp{\bar{\chi}^{N,n} \in \cdot \mid \mathcal{A}_N}.
        \end{equation}
        Our main result in this paper, Theorem \ref{T: quenched_limited_outcrossing}, asserts that this $\mathcal{M}_1\lp{\mathcal{D}\lp{\R_+, \En}}$-valued random variable converges  as $N\to\infty$. 
        Furthermore, the limit is the \textit{random} law of a family of coalescing random walk on a random graph described by an exchangeable fragmentation-coalescence processes (EFCs). The latter was introduced in \cite{berestycki04} which we describe in the next section, before giving the rigorous statement of our main result in Theorem \ref{T: quenched_limited_outcrossing}.
        
        Briefly, let $\lambda$ and $Q$ be given by Assumption~\ref{A: timescale} and Remark \ref{R: xi_q_connection} respectively. Then
        $\mathcal{L}^{N,n}$ converges to a \textit{random} element $\mathcal{L}_{Q,\lambda}^n\in\mathcal{M}_1\lp{\mathcal{D}\lp{\R_+, \En}}$ described as follows: 
        Consider a random graph $G_{Q,\lambda}^n=\big(G_{Q,\lambda}^n(t)\big)_{t\in\R_+}$, to be called a $Q$-$\lambda$ graph  starting with $n$ nodes at $t=0$,
        in which each node in the graph splits into two nodes independently at rate $\lambda$, and where any subset of nodes coagulate into one node according to the rate matrix $Q$ (to be made precise in Definition \ref{D: Q_lambda_graph}); see Figure \ref{F: ancestral_graph_walks} for an illustration.  
        Given $G_{Q,\lambda}^n$, we consider a system of coalescing random walks on it, where each random walker chooses one of the two directions to go when it reaches a point of fragmentation.
        Then $ \mathcal{L}_{Q,\lambda}^n$ is the conditional law of this system of coalescing random walks given $G_{Q,\lambda}^n$; we make this precise in Definition \ref{D: coalescent_on_Q_lambda}.

    \subsection{Random walks on exchangeable fragmentation-coalescence processes}\label{SS: EFCs}

The (random) $Q$-$\lambda$ graph $G_{Q,\lambda}^n$ mentioned above is a subgraph of
    an exchangeable fragmentation-coalescence process (EFC), where the EFC was introduced in
\cite{berestycki04}. Here we give a concise description of EFCs. 

        \begin{definition}[\cite{berestycki04}]
            An exchangeable fragmentation-coalescence (EFC) process $\Pi=(\Pi(t))_{t\in\R_+}$ is a $\Enfty$-valued exchangeable Markov process so that the restriction to $\En$, denoted by $\Pi^n$, is a c\'{a}dl\'{a}g finite state Markov chain which can only evolve by fragmentation of one block or by coalescence.
        \end{definition}

        By \cite[p. 781]{berestycki04} $\Pi$ is fully characterized by the initial state $\Pi(0)$, finite measures $\nu_{Disl}$ and $\nu_{Coag}$ on $\Delta$ where $\nu_{Coag}(\mathbf{0}) = 0$, and non-negative constants $c_k$ and $c_e$. In this paper we will always consider the case $\nu_{Disl} = 0$. The constant $c_e$ is  the rate at which each particle in the EFC fragments into two particles, while 
        $c_k$ and $\nu_{Coag}$ form the coalescing part of the EFC. Specifically, 
        for any finite subset $l$ of particles in the EFC, they coalesce as if they were an $l$-$\Xi$-coalescent where
        \begin{equation}\label{E:XiEFC}
        \Xi = c_k\delta_{\mathbf{0}} + \nu_{Coag}.            
        \end{equation}
        That is, the non-Kingman coalescences are governed by $\nu_{Coag}$ and $c_k$ is the rate at which any pair of particles coalesce independently.  

        \begin{example}[Ancestral Graphs]\label{X: ARG}\rm
        Suppose $\Pi = \lp{\Pi(t)}_{t \in \R_+}$ is an EFC with characteristics $c_k = 2$, $\nu_{Coag} = 0$, $c_e = \lambda,$ and $\nu_{Disl} = 0$. Then for any $n$ particles, each of the $\binom{n}{2}$ pairs of particles coalesce independently with rate $2$ and each of the $n$ particles fragment independently at rate $\lambda$, which is precisely the structure of an ancestral recombination graph or an ancestral selection graph \cite{Griffiths1991, GriffithsAndMarjoram1997, kroneneuhauser1997}, aside from the factor-of-$2$ difference in time scale noted in Remark~\ref{Rk:Timescale1}. This subgraph of $\Pi$ was conjectured in \cite[Section 4.1]{NFW25} as the scaling limit of a diploid Moran model with limited outcrossing. This is established in Section~\ref{SS: ARG_robustness} in this paper.  
        \end{example}

        \begin{example}\label{X: psi_model}
            Suppose that $\psi\in[0,1]$ and $\rho\in (0,\infty)$ are fixed constants. By a $(\psi, \lambda, \rho)$-EFC, denoted by $\Pi(\psi, \lambda, \rho)$, we mean an EFC where $c_k = 2\frac{1}{1+\rho \frac{\psi^2}{2}}$, $c_e = \frac{\lambda}{1+\rho \frac{\psi^2}{2}}$, and 
            \begin{equation*}
                \nu_{Coag}(dx) = 2\frac{\rho\frac{\psi^2}{2}}{1+\rho \frac{\psi^2}{2}}\delta_{\lp{\frac{\psi}{2}, \frac{\psi}{2}, 0, 0, \ldots}}(dx).
            \end{equation*}
            This EFC arises from the model in \cite{birkner2013} when $\psi\in [0,1]$ is a fixed constant, as we proved in Corollary \ref{Cor:BBE13}. 
        \end{example}

        The example~\ref{X: psi_model} generalizes when we allow $\psi$ to be random quantity. A natural prior is to take $\psi$ to be Beta-distributed. This yields an EFC whose coalescent part is a beta coalescent.

        \begin{example}\label{X: beta_coalescent_model}\rm
        Consider the EFC with 
            characteristics $c_k = 2\frac{1}{1+\frac{1}{12}\frac{\Gamma(4-r)}{\Gamma(2-r)}}$,  $c_e = \lambda \frac{1}{1+\frac{1}{12}\frac{\Gamma(4-r)}{\Gamma(2-r)}}$, $\nu_{Disl} = 0$ and
            \begin{equation*}
               \nu_{Coag}= 2\frac{\frac{\rho}{12}\frac{\Gamma(4-r)}{\Gamma(2-r)}}{1+ \frac{\rho}{12}\frac{\Gamma(4-r)}{\Gamma(2-r)}}\rho \int_{0}^1\frac{z^2}{2}\delta_{\lp{\frac{z}{2}, \frac{z}{2}, 0, 0, \ldots}}  \,{\rm Beta}(2-r,r)(d z),,
            \end{equation*}
where the ${\rm Beta}(2-r,r)(d z)$ measure has density
\[
\frac{1}{\Gamma(2-r)\Gamma(r)}z^{1-r}(1-z)^{r-1}, \quad z\in (0,1).
\]
This EFC arise when  $\psi$ in Example \ref{X: psi_model} is random and  Beta$(2-r, r)$ distributed for some $r \in [0,1]$. 
            
        \end{example}
        \begin{remark}\rm
            For $\Pi$ the EFC of Example~\ref{X: psi_model}, with an appropriate time-rescaling, the subgraph given by tracing the possible trajectories of coalescing random walks starting at the first $n$ blocks of $\Pi$ backwards in time is equal in distribution to the ancestral recombination graph of \cite{birkner2013}. In particular, the gene genealogies given by coalescing random walks on $\Pi$ are the same as for the gene genealogy of a single locus in \cite{birkner2013}.
        \end{remark}
        
        We define now a family of coalescing random walks $\lp{x_i}_{i = 1}^n$ on an EFC $\Pi$ with $\Pi(0) = \xi_0$. Two realizations of these coalescing random walks on a given $\Pi$ are demonstrated in Figure~\ref{F: ancestral_graph_walks}. We follow the Poissonian construction of \cite{berestycki04}. Let $P_C = \lp{\lp{t, \xi^{(C)}(t)}}_{t \geq 0}$ and $P_F = \lp{(t, \xi^{(F)}(t), k(t))}_{t \geq 0}$ be two independent Poisson point processes (PPPs) on the same filtration. The atoms of $P_C$ are points in $\R_+ \times \Enfty$ and $P_C$ has intensity measure $dt \otimes C$. The atoms of $P_F$ are points in $\R_+ \times \Enfty \times \N$ and $P_F$ has intensity measure $dt \otimes F \otimes \#$, where $\#$ is the counting measure on $\N$. $\Pi$ may be constructed in the obvious way from $P_C$ and $P_F$, as described in \cite[Section 3.2]{berestycki04}. We addend to this description a construction of coalescing random walks $\lp{x_i}$ on top of $\Pi$. The $\lp{x_i}$ are a family of $n$ coalescing particles $x_i$, modeled as $\N$-valued random variables, tracking the label of the block in $\Pi$ to which the $i$th particle belongs. We restrict ourselves to the case where the fragmentation $F = \lambda \mathbf{e}$, which is the case relevant to our results. In this case, the fragmentation of a block fragments it into precisely two parts.

        \begin{figure}
            \centering
            \includegraphics[width=0.5\linewidth]{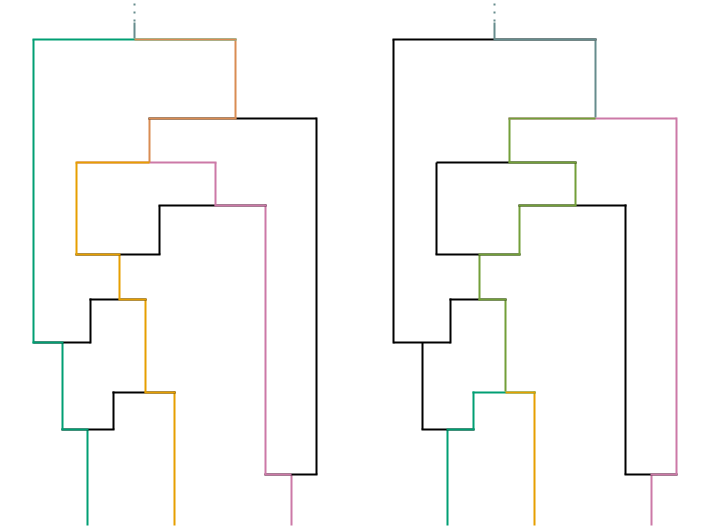}
            \caption{
            Two realizations of coalescing random walks on a single, fixed realization of an EFC $\Pi$ with $c_k = 1$. These random walks follow each coalescence in the EFC and choose between each of the possible edges ahead of them at a fragmentation with equal probability.}
            \label{F: ancestral_graph_walks}
        \end{figure}

        Let $m(t) := \max_i x_i(t)$ We begin with $x_i(0) = i$ for each $1 \leq i \leq n$ and construct $x_i$ as follows:
        \begin{itemize}
            \item if $t$ is not an atom time for either $P_C$ or $P_F$ then $x_i(t-) = x_i(t)$,
            \item if $t$ is an atom time for $P_C$ such that $\xi^{(C)}(t)_{\mid m(t-)} \neq \xi_0^{m(t-)}$ then $x_i(t)$ is equal to the label of the block to which the  $x_i(t-)$th block of $\Pi^{m(t-)}$ is mapped,
            \item if $t$ is an atom time for $P_F$ such that $k(t) = x_i(t-)$ then the label in $\Pi(t)$ of one of the two blocks into which the $x_i(t-)$th block fragments is chosen fairly (i.e. probability $\frac{1}{2}$ each) and $x_i(t)$ is set equal to this label for all $i$ with $x_i(t-)=k(t)$.
        \end{itemize}

        \begin{definition}\label{D: EFC_walks_coalescent}
            The coalescent process $\chi^n$ defined by
            \begin{equation*}
                i\sim_{\chi^n(t)} j \text{ if } x_i(t) = x_j(t)
            \end{equation*}
            is a $\En$-valued process. We denote by $\mathcal{L}_\Pi^n\in \mathcal{M}_1\lp{\mathcal{D}\lp{\R_+, \En}}$ the conditional law of $\chi^n$ given $\Pi$.
        \end{definition}

    We provide, in the main result of this paper contained in Section~\ref{SS: main_result}, criteria by which $\mathcal{L}_\Pi^n$ is the quenched scaling limit of the law of the coalescent conditional on the pedigree.

    \subsection{Main result}\label{SS: main_result}

        We are now ready to state our main result. 
        Recall that $\mathcal{L}^{N,n} := \PP\lp{\bar{\chi}^{N,n} \in \cdot \mid \mathcal{A}_N}\in \mathcal{M}_1\lp{\mathcal{D}\lp{\R_+, \En}}$ was defined in \eqref{Def:LawCondAn}, and $\mathcal{L}_\Pi^n$ was defined in Definition \ref{D: EFC_walks_coalescent}. 
        Throughout this paper, we equip the space 
        $\mathcal{M}_1\lp{\mathcal{D}\lp{\R_+, \En}}$ with the weak topology.
Under this topology, convergence in distribution (denoted by $\toL$ in this paper) of a sequence  $(\mu_N)_{N \in \N}$ of \textit{random }elements in $\mathcal{M}_1\lp{\mathcal{S}}$  is also called ``weak convergence in distribution" in \cite[Chapter 4]{Kallenberg2017}.

        \begin{thm}[Quenched convergence]\label{T: quenched_limited_outcrossing}
            Suppose Assumptions~\ref{A:c_N}, \ref{A:Q_Nn} and \ref{A: timescale} hold. Then there is a  positive finite measure $\Xi$ on $\Delta$ (as explained in Remark~\ref{R: xi_q_connection}) such that, for $\Pi$ an EFC with characteristics $c_k = \Xi(\mathbf{0})$, $\nu_{Coag} = \nu_{\Xi - c_k\delta_{\mathbf{0}}}$, $c_e = \lambda$, and $\nu_{Disl} = 0$, it holds  that    $\mathcal{L}^{N,n}$
            converges weakly in distribution to $\mathcal{L}_\Pi^n$ in $\mathcal{M}_1\lp{\mathcal{D}\lp{\R_+, \En}}$ as $N\to\infty$. That is,
        \[
        \mathcal{L}^{N,n} \toL \mathcal{L}_\Pi^n \quad \text{in} \quad \mathcal{M}_1\lp{\mathcal{D}\lp{\R_+, \En}}.
        \]
        \end{thm}

        The proof of Theorem~\ref{T: quenched_limited_outcrossing} is contained in Section~\ref{S: main_result_proof}. The proof proceeds by analyzing the possible trajectories of sample lineages in the pedigree backwards in time. On this graph of trajectories the sample lineages approximately perform a family of coalescing simple random walks. We show that the trace of these trajectories converge to a subgraph of a suitable EFC, and by a continuity argument show that $\mathcal{L}^{N,n}$ converges weakly in law to the random law governing the coalescing random walks on that EFC.

\begin{remark}\rm
    Compared to Theorem \ref{T: annealed_limit} which says that the (unconditional) law of $\bar{\chi}^{N,n}$ converges to a \textit{deterministic} element  (the law of a $\Xi$-coalescent) in $\mathcal{M}_1\lp{\mathcal{D}\lp{\R_+, \En}}$, Theorem \ref{T: quenched_limited_outcrossing} says that the conditional law $\mathcal{L}^{N,n} := \PP\lp{\bar{\chi}^{N,n} \in \cdot \mid \mathcal{A}_N}$ converges to a \textit{random} element in $\mathcal{M}_1\lp{\mathcal{D}\lp{\R_+, \En}}$ and the randomness in the limit is entirely captured by the EFC $\Pi$ for all sample size $n\geq 2$.
\end{remark}

        
        \begin{remark}\rm\label{R: strength_of_weak_convergence}\rm
            Note that the form of weak convergence, made precise in Section~\ref{S: weak_convergence_criteria}, is stronger than that described by the Meyer-Zheng topology of \cite{MZ84}. In particular, it follows from Theorem~\ref{T: quenched_limited_outcrossing} that the random occupation measure in $\mathcal{M}_{\rm loc}\lp{\R_+ \times \En}$ defined by $dt \otimes \delta_{\bar{\chi}^{N,n}(t)}$ converges vaguely in distribution with respect to the measure $\PP\lp{\cdot \,\mid\, \mathcal{A}_N}$ to the random measure $dt \otimes \delta_{\chi^n(t)}$ governed by $\PP\lp{\cdot \,\mid\, \Pi}$.
        \end{remark}

        \begin{remark}[Relaxing assumption on sampling]\label{Rk:sameIndiv}\rm
In our main result, we assumed that the $n$  sampled gene copies are from $n$ \textit{distinct} individuals, one gene copy from each individual. This assumption can be relaxed and the corresponding result can be obtained by our method, and we  briefly describe the corresponding limiting object here. 
Suppose the $n$  sampled gene copies consist of $2m$ gene copies sampled in pairs within individuals and $n-2m$ gene copies sampled each from distinct individuals. Then,
since $\alpha_N\to 1$, the two gene copies within each of the $m$ pairs will instantaneously coalesce with probability 1. Both the annealed limit and the quenched limit  (as $N\to\infty$) will be the same as before, but starts with $n-m$ gene copies rather than $n$.
The annealed limit in this case (for  fixed selfing probability) was obtained in \cite{NordborgAndDonnelly1997}. For the proof, one may specify what it means to take supremum in an empty set in \ref{E:TimeChange} and show that the transient behavior of the discrete-time processes in $\Sn$ converge to that of the limiting process.
        \end{remark}

\section{Applications}\label{S: applications}

    In this section, we discuss some applications of our main result that motivate this study. 

    In Section~\ref{SS: ARG_robustness} we establish that the system of coalescing random walks on the ancestral graph with binary mergers (like the ancestral recombination graph) arises as the limiting conditional coalescent for a class of models including both the Wright-Fisher model and the Moran model, showing robustness of such system. This result extends that of \cite{NFW25} from sample of size 2 to general size $n\geq 2$.

    In Section~\ref{SS: tree_stats}, we prove that our  main theorem implies the convergence of important tree statistics in population genetics as $N\to\infty$. We consider the internal and external branch lengths of the coalescent tree associated with the site-frequency spectrum (SFS).  We also illustrate, through simulation of the robust model in Section~\ref{SS: ARG_robustness}, how the SFS depends on the pedigree when the selfing probability is close to one.

    We then provide applications of Theorem~\ref{T: quenched_limited_outcrossing} to our diploid Sargasyan-Wakeley model in Section~\ref{SS: SW_model_apps}, and examples of how one can establish scaling limits for models like those in \cite{birkner2013, DFBW24} given by mixed demographies in Section~\ref{SS: mixed_demographies}.

    \subsection{Robustness of coalescing walks on the ancestral graph with binary mergers}\label{SS: ARG_robustness}

        In this section we establish, as a corollary of Theorem~\ref{T: quenched_limited_outcrossing} that the scaling limit described in \cite[Section 4.1]{NFW25} is robust to perturbations of the pedigree structure. That is, so long as non-binary mergers are vanishingly unlikely on the coalescent time-scale, the scaling limit is given by coalescing random walks on EFCs of the form described in Example~\ref{X: ARG}.
        
        Let $\mathfrak{c}_3 := \PP_{\xi_0^3}\lp{\chi^{N,n}(1) = \mathbf{1}_3}$ be the probability that three distinct sample lineages in three distinct individuals coalesce in a single time-step. This is the simplest type of non-binary merger. Then
        \begin{align}\label{E:triplecoa}
        \mathfrak{c}_3
        = &
        \frac{1}{\binom{N}{3}}\,
        \mathbb{E}\bigg[ 
        \frac{N-K}{N}\sum_{i=1}^N\!\left(
        \frac{1}{4}\binom{V_{i,i}}{2}
        +\frac{1}{8}\,V_{i,i} \sum_{j\neq i} V_{i,j}
        +\frac{1}{16}\binom{\sum_{j\neq i} V_{i,j}}{2}
        \right) \notag\\
        +&
        \sum_{i=1}^N\!\left(
        \frac{1}{8}\binom{V_{i,i}}{3}
        +\frac{1}{16}\binom{V_{i,i}}{2}\,\sum_{j\neq i} V_{i,j}
        +\frac{1}{32}\,V_{i,i}\binom{\sum_{j\neq i} V_{i,j}}{2}
        +\frac{1}{64}\binom{\sum_{j\neq i} V_{i,j}}{3}
        \right)
        \bigg], 
        \end{align}
        where, in the first equality, the first term comes from two newborns and one carried-over (parent $i$, the common parent for all 3 lineages) in which both newborn transmissions must pick the carried-over lineage’s copy in $i$; and the second term comes from three newborns in which all three transmissions through the same parent $i$ must choose the same copy.        
        The proof of Equation~\eqref{E:triplecoa} is given in Lemma \ref{lem:triple-one-step} in the Appendix.

         \begin{assumption}[Negligible triple coalescent]\label{A:triple}
            We assume that $\mathfrak{c}_3=o(c_N)$, i.e. $\mathfrak{c}_3 c_N^{-1} \to 0$ as $N\to\infty$. 
        \end{assumption}
        
        \begin{thm}\label{C: ARG_robustness}
            Suppose that Assumptions~\ref{A:c_N}, \ref{A: timescale} and \ref{A:triple} hold. Then Assumption \ref{A:Q_Nn} holds and the sequence of random measures $\PP_{\xi_0^n} \lp{\bar{\chi}^{N,n} \in \cdot \,\mid\, \mathcal{A}_N}$ converge weakly in law to $\mathcal{L}_\Pi^n$ where $\Pi$ is the ancestral recombination graph described in Example~\ref{X: ARG}.
        \end{thm}

        \begin{proof}
            By Lemma~\ref{L: one_step_combinatorics_no_triple}, $h_{\xi\xi}^{N,n} = 1 - 2\binom{|\xi|}{2}c_N + o(c_N)$. Therefore, the limiting exit rate of the process $\bar{\chi}^{N,n}$  from the state $\xi\in \En$ is $2\binom{\xi}{2}$ as $N\to\infty$, by Assumptions~\ref{A:c_N}.
            Further, Lemma~\ref{L: one_step_combinatorics_no_triple} demonstrates that $h_{\xi\eta}^{N,n} \in o(c_N)$ for any $\eta\in\En$ that cannot be obtained from $\xi$ via a single binary merger, i.e. for which $\xi \not \prec \eta$. For such $\xi, \eta$ we have that the limiting exit rate from $\xi$ to $\eta$ is $0$ by Assumptions~\ref{A:c_N}. Finally, there are $\binom{|\xi|}{2}$ partitions $\eta$ for which $\xi \prec \eta$. By exchangeability, these are equally likely and so the limiting transition rate from $\xi$ to $\eta$ is $1$ when $\xi \prec \eta$.
            
            It follows from the above that $Q_n$ from Assumption~\ref{A:Q_Nn} exists and is given by
            \begin{equation*}
                Q_n(\xi, \eta) =
                \begin{cases}
                    -2\binom{|\xi|}{2} &, \text{ if } \xi = \eta\\
                    2 &, \text{ if } \xi \prec \eta\\
                    0 &, \text{ otherwise}
                \end{cases}.
            \end{equation*}
            $Q_n$ is the transition-rate matrix of a Kingman $n$-coalescent with time rescaled by a factor of $2$. Hence Assumption~\ref{A:Q_Nn} is satisfied and $(Q_n)_{n \in \N}$ determines a measure $\Xi = 2\delta_{\mathbf{0}}$ on $\Delta$ in the sense of Remark~\ref{R: xi_q_connection}.  As Assumptions~\ref{A:c_N} and \ref{A: timescale} hold, and we have shown that $Q_{N,n}$ converges to the generator of a time-rescaled Kingman $n$-coalescent, the claim follows directly from Theorem~\ref{T: quenched_limited_outcrossing}.
        \end{proof}

        \begin{remark}\rm
            Theorem~\ref{C: ARG_robustness} establishes scaling limits for the models of \cite{NordborgAndDonnelly1997, Mohle1998a, NFW25}, which are alike insofar as they all have a time-rescaled Kingman coalescent for their annealed limits. The models of \cite{Mohle1998a} and \cite{NordborgAndDonnelly1997} are diploid Wright-Fisher models, i.e. $K_N = P_N = N$, while the model of \cite{NFW25} is diploid Moran model, i.e. $K_N = 1$ and $P_N = N$. This corollary establishes that the model of \cite[Section 4.1]{NFW25} is robust as long as the annealed limit of the models is Kingman-like with high selfing rate.
        \end{remark}

    \subsection{Convergence of tree statistics}\label{SS: tree_stats}

        The notion of convergence in what follows is weak convergence in distribution of random measures, which we make precise in Section~\ref{S: weak_convergence_criteria}. The standard reference is \cite[Chapter 4]{Kallenberg2017}. As described briefly in Remark~\ref{R: strength_of_weak_convergence}, the form of weak convergence established in Theorem~\ref{T: quenched_limited_outcrossing} implies that, for any compactly supported function $\varphi : \R_+ \times \En \to \R$, the associated integral functional $\int_0^\infty \varphi(s,\cdot(s))ds$ converges in distribution, i.e.
        \begin{equation*}
            \mathbb{E}\lb{\int_0^\infty \varphi\lp{s,\bar{\chi}^{N,n}(s)}ds \in \cdot \,\mid\, \mathcal{A}_N} \toL \mathbb{E}\lb{\int_0^{\infty} \varphi\lp{s,\chi^n(s)}ds \in \cdot \,\mid\, \Pi}.
        \end{equation*}

        An important class of these integral functionals, which in fact motivated our choice of the form of convergence of the random measures $\mathcal{L}^{N,n}$, are the internal and external branch lengths of the coalescent associated with the  site-frequency spectrum (SFS). Formally, the $r$th total branch length $\tau^{n,r}$ of a coalescent $\chi^n \in \mathcal{D}\lp{\R_+, \En}$ of a sample of size $n$ is defined by
        \begin{equation*}
            \tau^{n,r}(\chi^n) = \int_0^\infty \#\{C \in \chi^n(s) : |C| = r\} ds.
        \end{equation*}
        That is, $\tau^{n,r}$ is the \textit{sum of the lengths of the edges in the coalescent ancestral to $r$ samples}; see Figure \ref{fig:tau} for an illustration. 

\begin{figure}
    \centering
    \includegraphics[width=0.5\linewidth]{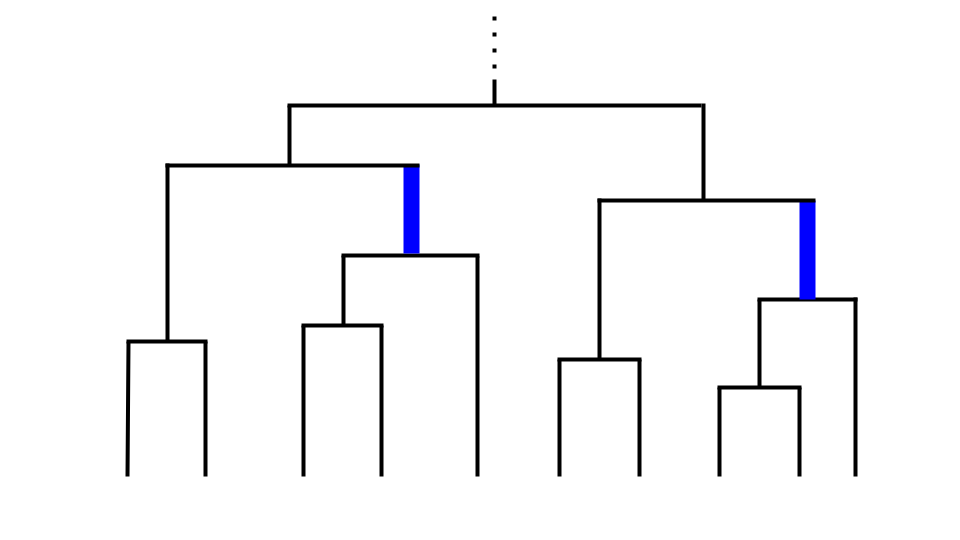}
    \caption{A gene genealogy for $n=10$ samples. Here, $\tau^{10,3}$ is the sum of the lengths of the two thick, blue edges. 
Although we do not consider mutation in this paper, note that any mutations on these branches will be inherited by, and thus present on, exactly $3$ gene copies in the sample.
}
    \label{fig:tau}
\end{figure}

        Note that $\tau^{n,r}$ is not of the form of $\varphi$ described above, so it requires an additional argument to show that these functionals do indeed converge. This is done via a truncation argument.

        \begin{proposition}\label{prop:tauconv}
            Suppose that, as $N \to \infty$, that Assumptions~\ref{A:c_N}, \ref{A: timescale}, and \ref{A:Q_Nn} hold. Then, for any finite $s \in \N$, $$\mathbb{E}\lb{\lp{\tau^{n,r}(\bar{\chi}^{N,n})}_{1 \leq r \leq s} \,\mid\, \mathcal{A}_N} \toL \mathbb{E}\lb{\tau^{n,r}(\bar{\chi}^{N,n})_{1 \leq r \leq s} \, \mid\, \Pi},$$ where $\Pi$ is as described in Theorem~\ref{T: quenched_limited_outcrossing}.
        \end{proposition}
        \begin{proof}
            For each $r\geq 1$ and $\xi\in\En$, set $g_r(\chi^n)=\#\{C\in \chi:|C|=r\}$, which is bounded above by $\lfloor n/r\rfloor$.  
            For $T>0$, define the truncated branch length
            \[
              \tau^{n,r,T}(\chi)=\int_0^T g_r(\chi(t))\,dt.
            \]
            By Remark~\ref{R: strength_of_weak_convergence}, the occupation measures
            $\mu_{\bar\chi^{N,n}}$ converge vaguely to $\mu_{\chi^n}$, so for any  cutoff function
            $\psi_{T,\varepsilon}\in C_c(\R_+)$ with $\int_0^\infty |\psi_{T,\varepsilon}(s)-\mathbf{1}_{[0,T]}(s)| ds \leq s$,
            \[
              \mathbb{E}\lb{\int_0^\infty \psi_{T,\varepsilon}(t)g_r(\bar\chi^{N,n}(t))\,dt \,\mid\, \mathcal{A}_N}
              \toL
              \mathbb{E}\lb{\int_0^\infty \psi_{T,\varepsilon}(t)g_r(\chi^n(t))\,dt \,\mid\, \Pi}.
            \]
            By letting $\varepsilon\to0$ we obtain
            \(
              \tau^{n,r,T}(\bar\chi^{N,n}) \toL \tau^{n,r,T}(\chi^n).
            \)
            Finally, monotone convergence as $T\to\infty$ yields joint convergence of $(\tau^{n,r})_{1\le r\le s}$ under the quenched laws.
        \end{proof}

\subsubsection{The site-frequency spectrum under nearly complete selfing}\label{sec:limitedoutcrossingSFS} 

In light of the importance of the SFS and other measures of genetic variation for inferring past events and processes affecting populations, it would be of some interest to go beyond the convergence result in Proposition~\ref{prop:tauconv} and to describe expected patterns of variation for a range of specific diploid population models conditional on the pedigree.  We leave this to future research, as it is outside the scope of the present work, but we note that a key question is the extent to which the SFS and other measures of genetic variation depend on the pedigree.  

Here we illustrate this using simulations of the (robust) conditional limit described in Theorem \ref{C: ARG_robustness}, namely, the coalescing random walk on the ancestral graph with binary mergers.
In this limit, the ancestral graph representing all the information of the pedigree relevant to the sample is generated by an EFC in which every distinct pair of ancestral lineages coalesce with rate $2$ and every single lineage splits at rate $\lambda$.  After an exponentially distributed waiting time with rate parameter $n(\lambda+n-1)$, one of these $\binom{n}{2} + n$ events is chosen in proportion to its rate.  If two lineages coalesce, the number of ancestral lineages decreases by one.  If a lineage splits, the number of ancestral lineages increases by one.  All events and their times are recorded, and the process stops the first time there is just a single ancestral lineage.  

Note that the marginal coalescent process for the sample is a Kingman coalescent process, which does not depend on $\lambda$.  However, the time it takes to reach a single lineage in the EFC process when $\lambda$ is large may far exceed any reasonable time for any coalescing random walk on the resulting ancestral graph to reach its most recent common ancestor.  Thus, there should be little error in stopping the EFC process before it reaches a single lineage.  The way time is measured here, there would be only a $e^{-20}\approx2\times10^{-9}$ chance of a pairwise coalescence time greater than $10$ in the marginal process.  In our simulations, we set the splitting rate to zero at $10$ units of time, so the remainder of the graph is a single Kingman coalescent tree, which we note preserves the marginal distribution of the process.  

The gene genealogy at a locus is generated by tracing ancestral lineages backward in time through the graph, as in the presentation of Figure~\ref{F: ancestral_graph_walks}.  When a lineage encounters a split, it follows one or the other ancestral line with equal probability, $1/2$.  When two lineages meet at a coalescent event in the graph, they necessarily coalesce. 
Unlinked loci have independent gene genealogies, conditional on the graph.  Figure~\ref{fig:sfs} displays simulation results for the SFS for six different values of $\lambda$ ranging from $1000$ to $0$.  Each panel shows the SFS for five independently generated ancestral graphs assuming a sample of size $n=20$.  Given each graph, gene genealogies of $10^5$ unlinked loci were simulated and their values of $\tau^{n,r}$ were recorded.  Figure~\ref{fig:sfs} plots the averages $\overline{\tau^{n,r}}$ of these branch lengths over the $10^5$ gene genealogies, normalized to sum to one, that is divided by $\sum_r \overline{\tau^{n,r}}$.

In Figure~\ref{fig:sfs}A, which has $\lambda=1000$, the SFS for every graph is not noticeably different than under the Kingman coalescent (not shown). We note that this is also the expectation for a fixed $\alpha_N = \alpha$ \cite{NordborgAndDonnelly1997,NFW25}.  A small effect of the graph (or pedigree) can be seen even when $\lambda=100$ in Figure~\ref{fig:sfs}B.  At the other extreme, when $\lambda=0$ as in Figure~\ref{fig:sfs}F, the entire pedigree of the population is reduced to a single Kingman coalescent tree, and the gene genealogies of every locus have this exact same tree.  Here, the SFS reflects the random outcomes, of branch lengths and numbers of descendants of each branch, of an effectively single-locus coalescent process.  Something quite similar is seen with $\lambda=0.1$ in Figure~\ref{fig:sfs}E.  The middle plots, Figure~\ref{fig:sfs}C and Figure~\ref{fig:sfs}D, with $\lambda=10$ and $\lambda=1$, occupy a transition zone between the deterministic case ($\lambda\to\infty$) where all graphs have the same SFS and the highly stochastic case ($\lambda\to 0$) where each random graph has a distinctly different SFS.  

\begin{figure}[h]
\centering

\includegraphics[scale=1]{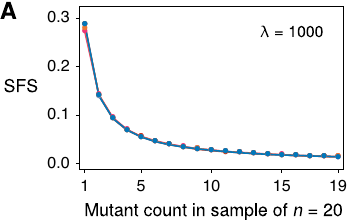}$\qquad$\includegraphics[scale=1]{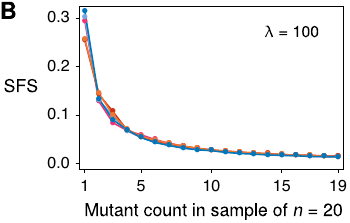}$\quad$

\vspace{1.5em}
 
\includegraphics[scale=1]{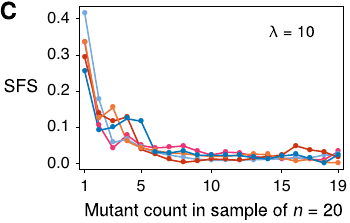}$\qquad$\includegraphics[scale=1]{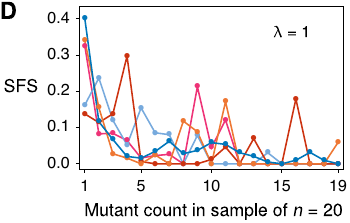}$\quad$

\vspace{1.5em}
 
\includegraphics[scale=1]{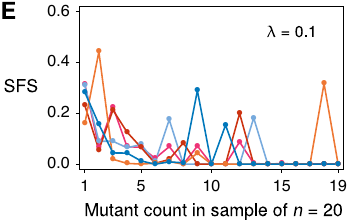}$\qquad$\includegraphics[scale=1]{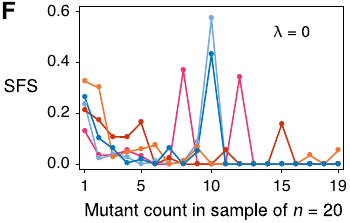}$\quad$
\caption{\small Expected site-frequency spectra under the limited-outcrossing model for $5$ independently generated ancestral graphs, corresponding to $5$ pedigrees, for each of six different values of $\lambda\in\{1000,100,10,1,0.1,0\}$ in decreasing order from panel A to F.  SFS on the vertical axis means the relative expected number of polymorphic sites where the mutant allele is found in each count $r\in\{1,2,3,\ldots,19\}$ in a sample of size $n=20$.  These were estimated by simulating gene genealogies of $10^5$ unlinked loci on each graph/pedigree.}
\label{fig:sfs}
\end{figure}

    \subsection{Diploid Sargasyan-Wakeley model}\label{SS: SW_model_apps}

        Consider our diploid Sargasyan-Wakeley model in Example~\ref{Eg:SW}. In this case \eqref{E:cN_general} reduces to
            \begin{equation}\label{E:cN}
            c_N= \mathbb{E}\lb{
            \frac{K}{N}\frac{N-K}{N-1}\frac{1}{N}
            \,+ \,        \frac{K(K-1)}{N(N-1)} \,\frac{1}{2P}},
            \end{equation} 
        where we noted that $2\frac{K}{N}\frac{N-K}{N-1}$ and $\frac{K(K-1)}{N(N-1)}$  are respectively the probability that exactly one (of the two distinct individuals) is an offspring and the probability that both individuals are offspring. 

        To apply the Theorem~\ref{T: quenched_limited_outcrossing}, we need to find assumptions on the joint distribution $\pi_N$ of $(K_N,P_N)$ under which Assumption~\ref{A:Q_Nn} holds, and for which we can explicitly describe the governing measure $\Xi$. To this end, we write $\pi_N$ as a mixture on the unit square $[0,1]^2$. We assume that there exists a measure $\nu_N$ on $[0,1]^2$ so that
        \begin{equation*}
            \pi_N := \int_{[0,1]^2} \delta_{\lfloor \bar{K} N \rfloor} \otimes \delta_{\lfloor \bar{P} N \rfloor } d\nu_N(\bar{K}, \bar{P}).
        \end{equation*}
        Define a family of measure $\mu_N$ on $[0,1] \times \N$ by
        \begin{equation*}
            \mu_N(dx, m) := \nu_N\lp{dx \times \left[\frac{m}{N}, \frac{m+1}{N} \right)}.
        \end{equation*}

        \begin{thm}\label{T:Mixture}
            Suppose, as $N \to \infty$, that Assumptions~\ref{A:c_N} and \ref{A: timescale} hold, and that $\frac{1}{2c_N}\mu_N$ converges vaguely on $[0,1] \times \N$ to a finite measure $\mu$ for which $\int \frac{x^2}{m} d\mu(x,m) \leq 1$. Define the sub-probability measure $\Theta(\mu)$ on $\Delta\setminus \{\bf{0}\}$ by
            \begin{equation*}
              \Theta(\mu):=  \int_{(0,1]\times\N} \frac{x^2}{m}\,
                \delta_{\big(\frac{x}{m},\ldots,\frac{x}{m},0,0,\ldots\big)}\, d\mu(x,m) \qquad \text{on }\Delta\setminus \{\bf{0}\},
            \end{equation*}
            where the dirac mass has $m$ copies of $\frac{x}{m}$ in a row and then all zeroes. Then, where $\Pi$ is an EFC with $c_k = 2 \big(1- \Theta(\mu)(\Delta\setminus \{\bf{0}\})\big)$ and $\mu_{\nu_{\rm Coag}} = 2\Theta(\mu)$, we have that $\mathcal{L}^{N,n}$ converges weakly in law to $\mathcal{L}_\Pi^n$.
        \end{thm}
        \begin{proof}
            To prove the claim we proceed by characterizing the ordered offspring distribution $\Phi_N$ on $\Delta$ from Lemma~\ref{L: ordered_offspring_distribution}. That is, let $\tilde{V}_{(1)}$, $\tilde{V}_{(2)}$, $\ldots$, $\tilde{V}_{(N)}$ denote the ordered genetic contribution of individuals between time-steps as in Lemma~\ref{L: ordered_offspring_distribution}. $\Phi_N$ is defined to be the law of the infinite tuple
            \begin{equation*}
                \mathfrak{V}_N = \lp{\frac{\tilde{V}_{(1)}}{2N}, \frac{\tilde{V}_{(2)}}{2N}, \ldots, \frac{\tilde{V}_{(N)}}{2N}, 0, 0, \ldots} \in \Delta.
            \end{equation*}

            We claim that as $N\to\infty$,
            \begin{equation*}
                \frac{1}{2c_N}\Phi_N(dx) \to \frac{1}{\left \langle x,x\right\rangle} \Theta(\mu)(dx)
            \end{equation*}
            vaguely on $\Delta \setminus\{\mathbf{0}\}$. Lemma~\ref{L: ordered_offspring_distribution} would imply that Assumption~\ref{A:Q_Nn} holds, and that $Q_n$ is the infinitesimal generator of an $n$-$\Xi$-coalescent with $$\Xi = 2\Theta(\mu) + 2 \Big(1- \Theta(\mu)(\Delta\setminus \{\bf{0}\})\Big)\,\delta_{\mathbf{0}},$$
            giving the desired formula for $c_k$ and $\mu_{\nu_{\rm Coag}}$.
            
            An application of Theorem~\ref{T: quenched_limited_outcrossing} then gives the result. It suffices, therefore, to show that the conditions of Lemma~\ref{L: ordered_offspring_distribution} hold.
            
            By definition of vague convergence it suffices to show that, for any $f = g(x_1,\ldots, x_M)$ a function of the first $M$ entries of an element $x$ in $\Delta$ with $f(\mathbf{0}) = 0$ that
            \begin{equation*}
                \frac{1}{2c_N}\Phi_N(f) \to \int \frac{f(z)}{\left \langle z,z \right\rangle} \Theta(\mu)(dz).
            \end{equation*}
            Let $\iota: (0,1] \times \N \to \Delta$ denote the embedding
            \begin{equation*}
                (x,m) \mapsto \lp{\frac{x}{m}, \ldots, \frac{x}{m}, 0,\ldots},
            \end{equation*}
            where $(x,m)$ is mapped to the element of $\Delta$ consisting of $m$ copies of $\frac{x}{m}$ and then infinitely many zeroes. With this notation we have that
            \begin{equation*}
                \int \frac{f(z)}{\left \langle z,z \right\rangle} \Theta(\mu)(dz) = \int \frac{f\circ \iota(x,m)}{\frac{x^2}{m}}d\Theta(\mu)(x,m)
                = \int f\circ \iota(x,m) d\mu(x,m) = \mu(f\circ \iota).
            \end{equation*}

            The convergence will follow from a concentration of measure result, which says that
            \begin{equation}\label{E: concentration_phenomenon}
                \PP\lp{\norm{\mathfrak{V}_N - \iota(dx,m)} \geq \epsilon \,\mid\, \pi_N \in [x,x+dx] \times \{m\})} \leq \frac{1}{\epsilon}O\lp{\frac{1}{N}}.
            \end{equation}
            That is $\mathfrak{V}_N$, conditional on the number of parents and offspring $\pi_N$ being in the rectangle $dx \times \{m\}$, concentrates around $\iota(dx,m)$. This implies the main result as it implies, by continuity of $f$ and compactness of $\Delta$, that
            \begin{equation*}
                \frac{1}{2c_N} \Phi_N(f) = \frac{1}{2c_N}\mu_N(f \circ \iota) + o(1) \to \mu(f\circ \iota),
            \end{equation*}
            where the last arrow follows from the assumption that $\frac{1}{2c_N}\mu_N$ converges vaguely to $\mu$ on $(0,1]\times \N$.

            We now justify \eqref{E: concentration_phenomenon}. Conditional on $\pi_N \in [x,x+dx] \times \{m\}$,
            exactly $\lfloor dx \, N \rfloor$ diploid offspring are produced by the $m$ potential parents. Ignoring outcrossing events for the moment, the vector of parental contributions $(W_1, \ldots, W_m)$ is multinomially distributed with parameters $\lfloor dx\, N \rfloor$ and cell probabilities $(\frac{1}{m}, \ldots, \frac{1}{m})$. In particular,
            \begin{equation*}
            \mathbb{E}\lb{W_i} = \frac{dx \, N}{m}
            \end{equation*}
            and
            \begin{equation*}
                Var\lp{W_i} = \frac{dx\, N}{m}\lp{1-\frac{1}{m}}.
            \end{equation*}
            When we renormalize by $2N$, we obtain
            \begin{equation*}
                Var\lp{\frac{W_i}{2N}} = O\lp{\frac{x}{mN}}.
            \end{equation*}
            Chebyshev's inequality for each fixed $i \leq m$ and a union bound therefore yield \eqref{E: concentration_phenomenon}. This union bound works because $f$ depends only on finitely many components of each element of $\Delta$. 
            
            Finally, we need to explain why this argument applies when we allow for outcrossing. Because of Assumption~\ref{A: timescale}, the proportion of offspring arising from outcrossing events is $o(1)$. These contribute a vanishingly small proportion of the genetic contributions of each of the parents.

        \end{proof}

    \subsection{Mixed demographies}\label{SS: mixed_demographies}

        Many models of interest involve mixed demographic behavior, such as Kingman-like behavior with rare extreme events, as in \cite{birkner2013, DFBW24}. Suppose that we have two different parent-offspring distributions $\pi_N^{(1)}$ and $\pi_N^{(2)}$ on $[N] \times [N]$ for the Sargasyan-Wakeley model of Example~\ref{Eg:SW}. 
        
        Let $c_N^{(1)}$ and $c_N^{(2)}$ denote the one-step transition probabilities for two sample lineages in distinct individuals coalescing in a single time-step. Similarly, we let $d_N^{(1)}$ and $d_N^{(2)}$ denote the outcrossing probabilities for a single time-step for each of the two demographies. For simplicity, we assume that the selfing probabilities $\alpha_N^{(1)}$ and $\alpha_N^{(2)}$ for the two demographic histories are identical, i.e. $\alpha_N^{(1)} = \alpha_N^{(2)} = \alpha_N$. For any fixed $\rho > 0$ we define the mixed parent-offspring distribution $\pi_N$ by
        \begin{equation}\label{mixedpiN}
            \pi_N = (1-\rho c_N^{(1)}) \pi_N^{(1)} + \rho c_N^{(1)}\pi_N^{(2)}.
        \end{equation}
        Finally, let $\mathcal{A}_N^{\rm mixed}$ denote the random pedigree governed by the mixed demography $\pi_N$ and $\bar{\chi}^{N,n}$ denote the coalescent with time rescaled by $c_N^{(1)}$, as in Equation~\eqref{E: bar_definition}. The quenched limit of gene genealogies for the mixed demography is characterized by the following corollary.

        \begin{corollary}\label{C: mixed_demographies}
            Suppose that $c_N^{(1)} \to 0$, $\rho c_N^{(2)} \to \rho'< \infty$,
            $\lp{c_N^{(1)}}^{-1}d_N^{(1)} \to \lambda < \infty$, and  $d_N^{(2)} \to 0$ as  $N \to \infty$. Suppose also that, for each $i \in \{1,2\}$, 
            the Sargasyan-Wakeley model (Example~\ref{Eg:SW}) under
            $\pi_N^{(i)}$ satisfies Assumption~\ref{A:Q_Nn} and the measure on $\Delta$ governing the limiting annealed coalescence 
            is $\Xi^{(i)}$. Then Assumption~\ref{A:Q_Nn} holds for the Sargasyan-Wakeley model under $\pi_N$ defined by \eqref{mixedpiN}, and
\[
\PP\lp{\tilde{\chi}\lp{\lfloor t c_N^{(1)}(1 + \rho') \rfloor} \in \cdot \,\mid\, \mathcal{A}_N^{\rm mixed}} \towd \mathcal{L}_{\Pi}^n \quad \text{in} \quad \mathcal{M}_1(\mathcal{D}\lp{\R_+, \En}),
\]
where $\Pi$ is an EFC with characteristics 
\begin{align}
 c_k =& \frac{1}{1+\rho'}\Xi^{(1)}(\mathbf{0}) + \frac{\rho'}{1 + \rho'} \Xi^{(2)}(\mathbf{0}), \label{E:ck_mixture1}\\
\nu_{Coag} =& \frac{1}{1+\rho'}\Xi^{(1)} + \frac{\rho'}{1+\rho'} \Xi^{(2)} - c_k\delta_{\mathbf{0}}, \label{E:ck_mixture2}\\   
c_e =& \frac{\lambda}{1+\rho'} \quad \text{and}\quad \nu_{Disl} = 0. \label{E:ck_mixture3}
\end{align}

        \end{corollary}
        \begin{proof}
            As the two demographies satisfy Assumption~\ref{A:Q_Nn}, so  does $\pi_N$. If $Q_n^{(i)}$ is the generator of the demography $\pi_N^{(i)}$ then we have that the generator of the mixed demography $Q_n^{\rm mixed}$ is 
            \begin{equation*}
                Q_n^{\rm mixed} = \frac{1}{1+\rho'}Q_n^{(1)} + \frac{\rho'}{1+\rho'} Q_n^{(2)} + o(1),
            \end{equation*}
            as 
            \begin{equation*}
                c_N= (1-\rho c_N^{(1)})c_N^{(1)} + \rho c_N^{(1)}c_N^{(2)} = c_N^{(1)}(1+\rho' + o(1)).
            \end{equation*}
            Observe that consistency of generators is closed under linearity, and so $\lp{Q_n^{\rm mixed}}_{n \in \N}$ is a consistent family of generators. Further,
            the relationship connecting infinitesimal generators on $\Enfty$ and measures on $\Delta$ described in Remark~\ref{R: xi_q_connection} is linear. Hence there is a unique measure $\Xi^{\rm mixed}$ associated to the family $\lp{Q_n^{\rm mixed}}$ given by
            \begin{equation*}
                \Xi^{\rm mixed} = \frac{1}{1+\rho'}\Xi^{(1)} + \frac{\rho'}{1+\rho'} \Xi^{(2)}.
            \end{equation*}
            As $d_N^{(1)}(c_N^{(1)})^{-1} \to \lambda < \infty$ and $d_N^{(2)} \to 0$ the outcrossing probability of the mixed demography $d_N^{\rm mixed} = (1-\rho c_N^{(1)})d_N^{(1)} + \rho c_N^{(1)}d_N^{(2)}$ satisfies $d_N^{\rm mixed} \lp{c_N^{(1)}(1+\rho')}^{-1} \to \frac{\lambda}{1+\rho'}$.
            The result then follows directly from Theorem~\ref{T: quenched_limited_outcrossing}.
        \end{proof}

        As a particular application of Corollary~\ref{C: mixed_demographies}, we adapt the model of \cite{birkner2013}. In the model of \cite{birkner2013} we have that 
        \begin{equation}\label{E: BBE_demography}
            \pi_N^{\rm BBE} = \lp{1-\frac{\rho}{N(N-1)}} \delta_{1} \otimes \delta_N + \frac{\rho}{N(N-1)}\delta_{\lfloor \psi N \rfloor} \otimes \delta_2,
        \end{equation}
        where $\psi$ is a fixed number in $[0,1]$. That is, with probability $1-\frac{\rho}{N^2}$ a single time-step consists of a single reproductive event, a single offspring born to any one of the $N$ potential parents with probability $\alpha_N$ or to any of the $\binom{N}{2}$ pairs of parents with probability $1-\alpha_N$; that is, it behaves as a Moran model as in \cite{NFW25}. With probability $\frac{\rho}{N^2}$, two parents reproduce to form $\lfloor \psi N \rfloor$ of the individuals in the nest time-step (going forward in time.) Let $\mathcal{A}_N^{\rm BBE}$ denote the pedigree associated to $\pi_N^{\rm BBE}$.

        \begin{corollary}\label{Cor:BBE13}
            Suppose that, as $N \to \infty$, we have that $(1-\alpha_N)N^{-1} \to \lambda < \infty$. Then the quenched law of the time-rescaled coalescent associated to $\pi_N^{\rm BBE}$
            \begin{equation*}
                \PP\lp{ \tilde{\chi}^{N,n} \lp{\left\lfloor t \lp{c_N^{(1)}\lp{1+\rho\frac{\psi^2}{2}}}^{-1} \right\rfloor} \in \cdot \,\mid\, \mathcal{A}_N^{\rm BBE}}
            \end{equation*}
            converges weakly in law to $\mathcal{L}_\Pi^n$, where $\Pi$ is the EFC of Example~\ref{X: psi_model}.
        \end{corollary}
        \begin{proof}
            We follow the notation introduced for Corollary~\ref{C: mixed_demographies}, where $\pi_N^{(1)} = \delta_1 \otimes \delta_2$ and $\pi_N^{(2)} = \delta_{\lfloor \psi N \rfloor} \otimes \delta_2$. With this notation we have that
            \begin{equation*}
                c_N^{(1)} = \frac{1}{N(N-1)} \quad \text{ and } c_N^{(2)} = \frac{\psi^2}{2} + o(1). 
            \end{equation*}
            Therefore we may write Equation~\eqref{E: BBE_demography} as 
            \begin{equation*}
                \pi_N^{\rm BBE} = (1-\rho c_N^{(1)})\pi_N^{(1)} + \rho c_N^{(1)} \pi_N^{(2)},
            \end{equation*}
            and we see that $\rho c_N^{(2)} \to \rho \frac{\psi^2}{2}$ (i.e. $\rho'=\rho \frac{\psi^2}{2}$ in the assumption of Corollary~\ref{C: mixed_demographies}).

            Observe now that the probability of three sample lineages in three distinct individuals coalescing in a single time-step for $\pi_N^{(1)}$ is zero. One readily checks that the assumptions of Theorem~\ref{C: ARG_robustness} are satisfied and hence $\Xi^{(1)} = 2\delta_{\mathbf{0}}$. 

            We now calculate the transition-rate matrix $Q_l^{(2)} = \lp{q^{l,2}_{\xi\eta}}_{\xi,\eta\in\mathcal{E}_l}$ for $\pi_N^{(2)}$. It suffices to do so for $q_{\xi_0^l\eta}^l$ for a fixed $l \in \Z_+$ by exchangeability. Observe firstly that if $\eta$ consists of at least three non-singleton blocks, then we immediately have that $q_{\xi_0^l\eta}^l = 0$. This is simply because their are precisely two parents under $\pi_N^{(2)}$. 
            
            Suppose now that $\eta$ consists of $b$ blocks $C_1, C_2, ..., C_b$ of sizes $k_1, k_2, \ldots, k_b$, at most two of which are non-singletons. Where $s$ is the number of singletons, we write that $\eta$ is of type $b;k_1,\ldots,k_r;s$, as in \cite{schweinsberg}. We only need to consider the cases now where $s \geq b-2$. Observe that the probability that any of the $l$ individuals containing a sample lineage is one of the two parents is $o(1)$ as $N\to\infty$, which will be negligible for $\pi_N^{(2)}$. We proceed now assuming that $k_1 \geq 2$ and that $k_1 \geq k_2 \geq \ldots \geq k_b$ without loss of generality.

            For any $j$ of the $l$ sample lineages, the probability that they are the offspring of one of the two parents is $\psi^j(1-\psi)^{l-j} + o(1)$. If $k_2 \geq 2$, then there is only one way for $\eta$ to appear, by $k_1$ sample lineages falling into one parent, $k_2$ in the other, and the remaining $l - k_1-k_2$ remaining sample lineages not being offspring. Therefore
            \begin{equation}\label{E: Q_calc_1}
                q^{l,2}_{\xi_0^l \eta} = 2^{1-k_1-k_2}\psi^{k_1+k_2}(1-\psi)^{l-k_1-k_2}+o(1).
            \end{equation}
            The factor of $2$ comes from the $k_1+k_2$ sample lineages sorting themselves into the right parents, where the order of the parents does not matter.

            If $k_2 = 1$, then $k_1$ sample lineages fall into one of the two parents. It is possible that exactly one or 0 of the remaining $l-k_1$ sample lineages fall into the other parent. Therefore
            \begin{equation}\label{E: Q_calc_2}
                q^{l,2}_{\xi_0^l \eta} = 2\lp{\frac{\psi}{2}}^{k_1}\lp{\frac{\psi}{2}(1-\psi)^{l-k_1-1} + (1-\psi)^{l-k_1}}+o(1) = 2^{1-k_1}(1-\psi)^{l-k_1} \frac{2-\psi}{2-2\psi} + o(1).
            \end{equation}
            Note that the first term in the summand comes from one of the $l-k_1$ remaining sample lineages falling into one of the two parents, while the second term comes from calculating that none of the remaining sample lineages do.

            Finally, we calculate $q^{l,2}_{\xi_0^l \xi_0^l}$. There can be zero, one, or two of the $l$ sample lineages that enter the two parents. Zero lineages entering the two parents occurs with probability $(1-\psi)^l + o(1)$ as $N\to\infty$. There are $\binom{l}{1}$ ways for $1$ sample lineage to be in one of the two parents, and so this occurs with probability $\binom{l}{1}\psi(1-\psi)^{l-1}+o(1)$. There are $\binom{l}{2}$ ways for $2$ sample lineages to be in the two parents, and in this case they do not coalesce together with probability $\frac{1}{2}$. Therefore
            \begin{equation}\label{E: Q_calc_3}
                q^{l,2}_{\xi_0^l \xi_0^l} = (1-\psi)^l + l(1-\psi)^{l-1}\psi + \frac{l(l-1)}{4}(1-\psi)^{l-2}\psi^2 + o(1).
            \end{equation}

            Combining Equations~\eqref{E: Q_calc_1}, \eqref{E: Q_calc_2}, and \eqref{E: Q_calc_3} shows that $Q_n^{(2)}$ corresponds exactly to the measure $\Xi^{(2)} = 2\delta_{\lp{\frac{\psi}{2}, \frac{\psi}{2}, 0, \ldots}}$ on $\Delta\setminus \{\bf{0}\}$.
        \end{proof}

    \medskip

    Having stated our main result and its applications, we organize the rest of the paper as follows: In Section~\ref{S: weak_convergence_criteria} we provide a simple criterion by which weak convergence of random measures over the Skorokhod space $\mathcal{D}\lp{\R_+, E}$ converges when $E$ is taken to be a locally compact Polish space. In Section~\ref{S: main_result_proof} we show that that the subgraph of the pedigree consisting of possible lineage trajectories converges to a $Q$-$\lambda$ graph, which is then realized via a coupling argument as a subgraph of an EFC. With the characterization of weak convergence of random measures in hand, the quenched convergence of coalescing random walks on the pedigree is shown by a continuity argument.

\section{Characterization of weak convergence in distribution of random measures on \texorpdfstring{$\mathcal{D}\lp{\R_+, E}$}{}}
\label{S: weak_convergence_criteria}

To prove Theorem~\ref{T: quenched_limited_outcrossing} we shall develop a characterization of weak convergence in distribution of random probability measures in Theorem \ref{T: random_measure_integral_functionals}. We are not aware of any previous recognition of this characterization in the literature, so Theorem \ref{T: random_measure_integral_functionals} may have general interest.

More precisely, we provide in this section sufficient criteria to characterize weak convergence in distribution when the Polish space $S$ in question is precisely $\mathcal{D}\lp{\R_+, E}$ with the $J_1$ topology, where $E$ is a locally compact Polish space (such as $\mathcal{E}_n$ or $\mathcal{E}$). The basic strategy is to construct a suitable dense family in $C_b\lp{\mathcal{D}\lp{\R_+, E}}$ and then to devise criteria by which evaluations our random measures against this suitable class of test functions converge.

Let $\mathcal{S}$ be a Polish space, such as $\mathcal{D}\lp{\R_+, E}$.
Under the weak topology on $\mathcal{M}_1\lp{\mathcal{S}}$, convergence in distribution (denoted by $\toL$ in this paper) of a sequence  $(\mu_N)_{N \in \N}$ of \textit{random }elements in $\mathcal{M}_1\lp{\mathcal{S}}$  is also called ``weak convergence in distribution" and is denoted by $\towd$ in \cite[Chapter 4]{Kallenberg2017}. That is,            
            \begin{definition}\label{D: weak_convergence}
                Let $\mu$ and $(\mu_N)_{N \in \N}$ be random probability measures on  a Polish space $\mathcal{S}$. We say that $\mu_N$ \textbf{converges weakly in law} to $\mu$ and write
                \[
                \mu_N \towd \mu \quad \text{in} \quad \mathcal{M}_1(\mathcal{S})
                \]
                if $\mu_N \toL \mu$ when $\mathcal{M}_1(S)$ is equipped with the weak topology.
            \end{definition}
            By \cite[Theorem 4.19]{Kallenberg2017}
            \begin{equation}\label{iff}
                \mu_N \towd \mu \text{\quad if and only if\quad } \mu_N(f) \toL \mu(f)
            \end{equation}
            for any continuous and bounded function of bounded support $f$.
            We use the latter equivalent characterization throughout for the Polish space $\mathcal{S} = \mathcal{D}\lp{\R_+, \mathcal{E}_n}$.

    \subsection{\texorpdfstring{A dense family in $C_b\lp{\mathcal{D}\lp{\R_+, E}}$}{}}

        Suppose that $E$ is a locally compact Polish space. Then $\mathcal{D}\lp{\R_+, E}$ is a Polish space with the $J_1$ topology \cite[Theorem 5.6, p.121]{ethier2009markov}, under a bounded metric; see \cite[Chapter 3, eqn. (5.2)]{ethier2009markov}.
        
        We shall describe a dense subset of $C_b\lp{\mathcal{D}\lp{\R_+, E}}$ in Lemma~\ref{L: density} below when $C_b\lp{\mathcal{D}\lp{\R_+, E}}$ is given the topology of uniform convergence on compacta. To this end, we first describe an explicit continuous function on $\mathcal{D}\lp{\R_+, E}$. 
        
        Fix $\varphi\in C_c\lp{\R_+ \times E}$ and $a \in \R$, and define the function $I(\varphi, a):\,\mathcal{D}\lp{\R_+, E}\to\R$ by
        \begin{equation*}
            I(\varphi, a)(x) := a +  \int_0^\infty \varphi(s, x(s)) ds.
        \end{equation*}

        Here, for any topological space $X$ we take $C_{c}(X)$ to denote the continuous functions from $X$ into $\R$ with compact support. 
        \begin{lemma}\label{L:cts}
            For any $\varphi\in C_{c}\lp{\R_+ \times E}$ and $a\in \R$, $I(\varphi, a)$ is a bounded and continuous function on $\mathcal{D}\lp{\R_+, E}$. 
        \end{lemma}
        \begin{proof}
            We need to show that for any sequence $x_N$ in $\mathcal{D}\lp{\R_+, E}$ converging to a path $x$ therein that $I(\varphi, a)(x_N)$ converges to $I(\varphi, a)(x)$.
    
            As $\varphi$ has compact support, there exists a finite $T$ in $\R_+$ such that $supp(\varphi) \subset [0,T] \times E$. Therefore
            \begin{equation*}
                I(\varphi, a)(x_N) = a + \int_0^T \varphi(s, x_N(s)) ds.
            \end{equation*}
            As $x$ has at most countably many points of discontinuity \cite[Lemma 5.1]{ethier2009markov} and $\lim_{N \to \infty} x_N(s) = x(s)$ for all continuity points of $s$ of $x$, we have that $x_N$ converges to $x$ almost surely on $[0,T]$. Therefore, by continuity $\varphi(s,x_N(s))$ converges to $\varphi(s,x(s))$ almost surely on $[0,T]$. As $|\varphi(s, x_N(s))| \leq \norm{\varphi}_{\infty} \mathbbm{1}_{[0,T]}(s)$, the sequence of functions $\varphi(s,x_N(s))$ is dominated by an integrable function. The dominated convergence theorem therefore gives the convergence.
        \end{proof}
        
        We now define the collection $\mathcal{I}:=\{I(\varphi, a):\,\varphi\in C_{c}\lp{\R_+ \times E},\,a\in \R\}$. Note that $I(\varphi, a) \subset C_b(\mathcal{D}\lp{\R_+, E})$ by Lemma \ref{L:cts}.

        To apply a Stone-Weierstrass, we begin by showing that $\mathcal{I}$ separates points in $\mathcal{D}\lp{\R_+, E}$. That is, for any $x \neq y$ there exists $\varphi, a$ such that $I(\varphi,a)(x) \neq I(\varphi, a)(y)$.
        
        \begin{lemma}\label{L: separating_points}
            For any locally compact Polish space $E$, $\mathcal{I}$ separates points in $\mathcal{D}\lp{\R_+, E}$.
        \end{lemma}
        \begin{proof}
            If $x \neq y$, then there is a continuity point $t_0 \in \R_+$ of both $x$ and $y$ such that either $x(t_0) \neq y(t_0)$. Because $E$ is locally compact and Hausdorff, there exist disjoint open neighborhoods with compact closure $U$ and $V$ of $x(t_0)$ and $y(t_0)$, respectively. Choose $\theta$ in $C_c(E, [0,1])$ with $\theta_{|U} \equiv 1$ and $\theta_{|V} \equiv 0$, which exists by Urysohn's lemma.
    
            As $t_0$ is a continuity point, there is a $\delta > 0$ such that $x(s) \in U$ and $y(s) \in V$ for all $|s-t_0| < \delta$. Let $\psi \in C_c(\R_+, [0,1])$ satisfy $\psi_{|[t_0-\frac{1}{2}\delta, t_0 + \frac{1}{2}\delta]} \equiv 1$ and $supp(\psi) \subset [t_0-\delta, t_0+\delta]$. 
            
            Define $\varphi$ by
            \begin{equation*}
                \varphi(s,e) := \psi(s) \theta(e).
            \end{equation*}
            $\varphi$ clearly belongs to $C_c(\R_+ \times E)$. Further
            \begin{align*}
                I(\varphi, 0)(x) > 0 = I(\varphi, 0)(y),
            \end{align*}
            so $I(\varphi, 0)$ separates $x$ and $y$. Thus $\mathcal{I}$ separates points of $\mathcal{D}\lp{\R_+, E}$.
        \end{proof}

        We now define
        $\bar{\mathcal{I}}$ to be the collection of all finite sums of finite products of elements of $\mathcal{I}$. That is, each element $I$ in $\bar{\mathcal{I}}$ may be written as
        \begin{equation*}
            \sum_{i = 1}^n \prod_{j=1}^{k_i} I_{ij},
        \end{equation*}
        where each $I_{ij}$ belongs to $\mathcal{I}$ and $k_i, n \in \Z_+$. An extension of this class of continuous, bounded integral functionals will, by an application of Stone-Weierstrass \cite[Chapter 7, p.~11]{Ok_AppliedPointSetTopology}, be dense in $C_b(\mathcal{D}\lp{\R_+, E})$.
    
        \begin{lemma}\label{L: density}
             Suppose that $E$ is locally compact and Polish. Then $\bar{\mathcal{I}}$ is dense in $C_b(\mathcal{D}\lp{\R_+, E})$ in the topology of uniform convergence on compacta.
        \end{lemma}
        \begin{proof}
            To show that $\bar{\mathcal{I}}$ is dense in $C_b(\mathcal{D}\lp{\R_+, E})$ in the topology of uniform convergence on compacta, we need to show that for any function $f$ in $C_b(\mathcal{D}\lp{\R_+, E})$, any $\varepsilon > 0$, and any compact set $K$ there exists a function $I$ in $\bar{\mathcal{I}}$ such that
            \begin{equation*}
                \sup_{e \in K} |f(e) - I(e)| < \varepsilon.
            \end{equation*}
        
            We begin by observing that $\bar{\mathcal{I}}$ contains constants; take $\varphi \equiv 0$ and $a$ to be free. $\bar{\mathcal{I}}$ also is closed under addition, multiplication, and scalar multiplication, which is immediate. By Lemma~\ref{L: separating_points} $\mathcal{I}$, and so $\bar{\mathcal{I}}$, separates points.
    
            Fix any compact set $K \subset \mathcal{D}\lp{\R_+, E}$ and any $f\in C_b(\mathcal{D}\lp{\R_+, E})$. The restriction of elements of $\bar{\mathcal{I}}$ to $K$ yields an algebra $\bar{\mathcal{I}}(K)$ of continuous functions from $K$ to $\mathbb{R}$ that contains the constants and separates points. By the Stone-Weierstrass theorem for compact Hausdorff spaces \cite[Chapter 7, p.~11]{Ok_AppliedPointSetTopology}, the uniform closure of $\bar{\mathcal{I}}(K)$ is all of $C(K)$. This gives the claim.
            
        \end{proof}

    \subsection{A characterization of weak convergence in distribution}
        We show by means of a density argument that it suffices to check, for weak convergence in distribution of random measures, joint convergence in distribution when testing against integral functionals, as described by the following theorem. 

        \begin{thm}\label{T: random_measure_integral_functionals}
            Let $\lp{\mu_N}_{N \in \N}$ and $\mu$ be random measures  taking value in $\mathcal{M}_1\lp{\mathcal{D}\lp{\R_+, E}}$, where $E$ is a locally compact Polish space. Then 
            $\mu_N \towd \mu$
            if and only if for any finite collection $\{\varphi_i\}_{i=1}^k$ of elements in $C_c\lp{\R_+ \times E}$ that
            \begin{equation}\label{E: integral_functional_convergence}
                \lp{\mu_N\lp{\int_{\R_+} \varphi_i(s, x(s))ds}}_{1 \leq i \leq k}\toL \lp{\mu\lp{\int_{\R_+} \varphi_i(s, x(s))ds}}_{1 \leq i \leq k}.
            \end{equation}
        \end{thm}

        \begin{proof}
        We first prove sufficiency of \eqref{E: integral_functional_convergence}.
            From the "if part" of \eqref{iff}, it suffices to show that
            \begin{equation}\label{E: weak_convergence}
                \mu_N(f) \toL \mu(f) \quad \text{for } f\in C_b\lp{\mathcal{D}\lp{\R_+, E}}.
            \end{equation}
            By Lemma~\ref{L: density} it suffices, without loss of generality, to demonstrate \eqref{E: weak_convergence} for $f$ in $\bar{\mathcal{I}}$. For any element $f$ in $\bar{\mathcal{I}}$ there is an array of elements $\lp{I_{ij}}_{i\in [n], j \in [k_i]}$ of $\mathcal{I}$ such that
            \begin{equation*}
                f = \sum_{i=1}^n \prod_{j=1}^{k_i}I_{ij}.
            \end{equation*}
            Linearity of expectation yields
            \begin{equation}\label{E: decomposition_I}
                \mu_N(f) = \sum_{i=1}^n \mu_N\lp{\prod_{j=1}^{k_i} I_{ij}}.
            \end{equation}
            By assumption, $\mu_N(\lp{I_{ij}}_{i \in [n], j \in [k_i]})$ converges in distribution to $\mu(\lp{I_{ij}}_{i \in [n], j \in [k_i]})$. As polynomials are continuous and continuity preserves convergence in distribution, \eqref{E: decomposition_I} converges in distribution to
            \begin{equation*}
                \sum_{i=1}^n \mu\lp{\prod_{j=1}^{k_i} I_{ij}} = \mu(f),
            \end{equation*}
            which gives \eqref{E: weak_convergence}.

            We now prove the necessity of Equation~\eqref{E: integral_functional_convergence}. By Lemma~\ref{L:cts} the integral functionals are all continuous. The continuous mapping theorem then gives the claim.
            
        \end{proof}

        The characterization of weak convergence in distributions of random measures on $\mathcal{D}\lp{\R_+, E}$ for locally compact and Polish $E$ given by Theorem~\ref{T: random_measure_integral_functionals} allows one to give classical-like characterizations of such convergence, as typified by Theorem \ref{T: weak_convergence_of_random_measures} below.

        \begin{definition}[Weak convergence in finite-dimensional distribution]\label{D: weak_FDD}
            Let $\lp{\mu_N}_{N \in \N}$ and $\mu$ be random measures taking value in $\mathcal{M}_1\lp{\mathcal{D}\lp{\R_+, E}}$, where $E$ is a locally compact Polish space. We say that the sequence $\mu_N$ converges weakly in finite-dimensional distribution to $\mu$ if 
            \begin{equation*}
                \lp{\pi_{\vec{t}}}_*\mu_N \towd \lp{\pi_{\vec{t}}}_* \mu \quad \text{in} \quad \mathcal{M}_1(\R^k)
            \end{equation*}            
            for any finite collection of times $\lp{t_i}_{1 \leq i \leq k}$, for any $k\geq 1$, where $\pi_{\vec{t}}$ is the evaluation map
            \begin{align*}
                \pi_{\vec{t}} : \mathcal{D}\lp{\R_+,E} &\to E^k\\
                x &\mapsto \lp{x(t_i)}_{1 \leq i \leq k}
            \end{align*}
           and     $\lp{\pi_{\vec{t}}}_* \mu$ is the pushforward of $\mu$ by $\pi_{\vec{t}}$. 

        \end{definition}

        We now define an analogue of the compact containment condition for random measures on Skorokhod space \cite[p. 129]{ethier2009markov}.
    
        \begin{definition}[Compact containment condition for random measures]\label{D: CCC_random_measures}
            Let $\{\mu_N\}$ denote a sequence of random measures on $\mathcal{D}\lp{\R_+, E}$. Then $\{\mu_N\}$ satisfies the compact containment condition if, for every $\varepsilon > 0$ and $T > 0$ that there is a compact set $K(\varepsilon, T) \subset E$ such that
            \begin{equation*}
                \PP\lp{\liminf_{N \to \infty} \mu_N\lp{x([0,T]) \subset K(\varepsilon, T) } \geq 1-\varepsilon}  = 1.
            \end{equation*}
        \end{definition}

        We obtain the following characterization of weak convergence in law
        of random measures in $\mathcal{D}\lp{\R_+, E}$.
        
        \begin{thm}\label{T: weak_convergence_of_random_measures}
            Let $E$ be a locally compact Polish space. Let $\mu$ and $\{\mu_N\}_{N \in \N}$ be random measures taking value in $\mathcal{M}_1\lp{\mathcal{D}\lp{\R_+, E}}$. Suppose the followings hold:
            \begin{itemize}
                \item[(i)] $\{\mu_N\}$ converges weakly in finite-dimensional distribution to $\mu$ as in Definition~\ref{D: weak_FDD}, and
                \item[(ii)] $\{\mu_N\}$ satisfies the compact containment condition as in Definition~\ref{D: CCC_random_measures}.
            \end{itemize}
            Then $\mu_N \towd \mu$.
        \end{thm}

        \begin{proof}
            By Theorem~\ref{T: random_measure_integral_functionals} it suffices to show that, for any finite collection $\lp{\varphi_i}_{1 \leq i \leq k}$ of elements of $C_c\lp{\R_+, E}$ that
            \begin{equation*}
                \lp{\mu_N\lp{\int_0^{\infty} \varphi_i(s,x(s))ds}}_{1 \leq i \leq k} \toL \lp{\mu\lp{\int_0^{\infty} \varphi_i(s,x(s))ds}}_{1 \leq i \leq k}.
            \end{equation*}
            To this end, let $F_i$ be defined by
            \begin{equation*}
                F_i := \int_0^{\infty} \varphi_i(s,x(s))ds.
            \end{equation*}
            Fix $G$ in $C_b(\R^k)$. We need to show that
            \begin{equation*}
                \mathbb{E}\lb{G\lp{\mu_N(F_1),\ldots \mu_N(F_k)}} \to \mathbb{E}\lb{G\lp{\mu(F_1),\ldots \mu(F_k)}}.
            \end{equation*}
            We will proceed by approximating $F_i$ by Riemann sums.
            
            As each $\varphi_i$ is compactly supported, there exists $T > 0$ such that $\cup_i supp\lp{\varphi_i} \subset [0,T] \times E$. By the compact containment condition, for any $\eta, \rho > 0$ there exists $K(\eta, T)$ a compact set in $E$ such that
            \begin{equation}\label{E: CCC_convergence_proof}
                \PP\lp{ \mu_N\lp{x([0,T]) \subset K(\eta, T)} \geq 1-\eta} \geq 1 - \rho
            \end{equation}
            for all sufficiently large $N$.
            
            Since $\varphi_i$ is continuous on $[0,T] \times K(\eta, T)$ for any compact $K \subset E$, it is uniformly continuous there. Fix $\varepsilon > 0$. By uniform continuity, for any $\varepsilon > 0$ we can choose $\delta>0$ so that
            \begin{equation*}
                |\varphi_i(s, e) - \varphi_i(t, e)| < \varepsilon
            \end{equation*}
            for every $i = 1, \ldots, k$ and any $|s-t| < \delta$. Choose a fine partition $0 = t_0 < t_1 < \ldots < t_r = T$ so that $\min_i\{t_{i+1} - t_i\} < \delta$ and define the discrete Riemann sum $R_i^{\varepsilon}$ by
            \begin{equation*}
                R_i^{\varepsilon} := \sum_{j=0}^{r-1} \lp{t_{j+1} - t_j} \varphi_i(t_j, x(t_j)).
            \end{equation*}

            For ease of notation we write $\mu_N(F)$ for the tuple
            \begin{equation*}
                \lp{\mu_N(F_1), \ldots, \mu_N(F_k)}
            \end{equation*}
            in $\R^k$, and similarly for $\mu_N(R^{\varepsilon})$, $\mu(F)$, and $\mu(R^{\varepsilon})$.
            By the triangle inequality we have that
            \begin{align}\label{E: three_epsilon}
               & \la{\mathbb{E}\lb{G(\mu_N(F)) - G(\mu(F))}} \\
                \leq &\mathbb{E}\lb{\la{G(\mu_N(F)) - G(\mu_N(R^{\varepsilon}))}}
                + 
                \la{\mathbb{E}\lb{G(\mu_N(R^{\varepsilon})) - G(\mu(R^{\varepsilon}))}}
                + 
                \mathbb{E}\lb{\la{G(\mu(R^{\varepsilon})) - G(\mu(F))}}.
            \end{align}

            We first show that the second summand in \eqref{E: three_epsilon} can be made arbitrarily small. Recall that $\{t_i\}_{i=0}^r$ are the points of our mesh for the Riemann sums $R^{\varepsilon}$ with $t_0 = 0$ and $t_r = T$. Define the map $H_i: E^r \to \R$ by
            \begin{equation*}
                H_i(e_0, \ldots, e_{r}) := \sum_{j=0}^{r-1} (t_{j+1}-t_j)\varphi_i(t_j, e_j).
            \end{equation*}
            Then $R_i^{\varepsilon} = H_i \circ \pi_{\vec{t}}$. As $\mu_N$ converges weakly in finite-dimensional distribution by assumption (i), $\mu_N(R^\varepsilon)$ converges in distribution to $\mu(R^{\varepsilon})$. Therefore the second summand in \eqref{E: three_epsilon} can be made arbitrarily small..

            We now show the first summand can be made arbitrarily small. As the $F_i$ and $R_i^{\varepsilon}$ are uniformly bounded by $M :=T\norm{\varphi_i}_{\infty}$ we have $\mu_N(F), \mu_N(R^\varepsilon) \in [-M,M]^k$ almost surely. Since $G$ is uniformly continuous on $[-M,M]^k$ we have for every $\delta > 0$ that there exists $\gamma > 0$ such that
            \begin{equation*}
                |G(u) - G(v)| < \delta \text{\, when \,} \max_{1 \leq i \leq k} |u_i - v_i| < \gamma.
            \end{equation*}
            By uniform continuity of $G$ on $[-M,M]^k$ we therefore have that
            \begin{equation}\label{E: G_closeness}
                \la{G(\mu_N(F)) - F(\mu_N(R^{\varepsilon}))}
                \leq \delta + 2\norm{G}_{\infty} \mathbf{1}_{\{\max_i|\mu_N(F_i) - \mu_N(R^{\varepsilon})| \geq \gamma \}}.
            \end{equation}
            Taking expectations over \eqref{E: G_closeness} yields
            \begin{equation}\label{E: first_summand_bound}
                \mathbb{E}\la{G(\mu_N(F)) - F(\mu_N(R^{\varepsilon}))}
                \leq \delta + 2\norm{G}_{\infty}\PP\lp{\max_i \la{\mu_N(F_i) - \mu_N(R_i^\varepsilon)} \geq \gamma}.
            \end{equation}
            As we can take $\delta$ arbitrarily small, it will suffice to show that the probability term can be made arbitrarily small.

            To show the probability term can be made arbitrarily small, we utilize the compact containment condition. Fix $\eta, \rho > 0$. By assumption (ii), there exists a compact set $K(\eta, T)$ such that, for the event $A_N := \{\mu_N(x([0,T]) \subset K(\eta,T)) \geq 1-\eta\}$
            \begin{equation*}
                \PP\lp{A_N} \geq 1-\rho
            \end{equation*}
            for all sufficiently large $N$. On $A_N$ we therefore have
            \begin{equation*}
                \la{\mu_N(F_i) - \mu_N(R_i^\varepsilon)} \leq T\varepsilon + 2T\norm{\varphi_i}_{\infty} \eta.
            \end{equation*}
            As $\eta$ and $\varepsilon$ we can take arbitrarily small, we can take them so that
            \begin{equation*}
                T\varepsilon + 2T\norm{\varphi_i}_{\infty} \eta < \gamma
            \end{equation*}
            for all $i = 1,\ldots, k.$ Therefore
            \begin{equation}\label{E: prob_bound}
                \PP\lp{\max_i \la{\mu_N(F_i) - \mu_N(R_i^{\varepsilon})} \geq \gamma} \leq \PP\lp{A_N^c} \leq \rho
            \end{equation}
            for all sufficiently large $N$. Combining \eqref{E: prob_bound} with \eqref{E: first_summand_bound} gives
            \begin{equation*}
                \mathbb{E}\la{G(\mu_N(F)) - F(\mu_N(R^{\varepsilon}))}
                \leq \delta + 2 \norm{G}_\infty \rho.
            \end{equation*}
            Taking $\delta$ and $\rho$ arbitrarily small then makes the first summand of \eqref{E: three_epsilon} arbitrarily small.

            The third summand can be made arbitrarily small by the same argument as for the first summand, giving the result.

        \end{proof}

        When $E$ is compact, we immediately have the following corollary.

        \begin{corollary}\label{C: fdd_compact_case}
            Suppose that $E$ is a compact metric space and $\mu$ is a random measure on $\mathcal{D}\lp{\R_+, E}$. Let $\mu_N$ denote a sequence of random measures on $\mathcal{D}\lp{R_+, E}$ such that, for any finite collection of times $\lp{t_i}_{1 \leq i \leq k}$ such that $\lp{\pi_{\vec{t}}}_* \mu_N \towd \lp{\pi_{\vec{t}}}_* \mu$. Then $\mu_N \towd \mu$.
        \end{corollary}

        \begin{remark}\rm
            Corollary~\ref{C: fdd_compact_case} allows us to strengthen some already existing results, including \cite[Theorem 3.8]{abfw25} which demonstrates weak convergence in finite-dimensional distribution. By Corollary~\ref{C: fdd_compact_case}, weak convergence in law therefore follows from compactness of $\mathcal{E}_n$.            
        \end{remark}

\section{Proof of Theorem~\ref{T: quenched_limited_outcrossing}}\label{S: main_result_proof}

    We will relate subgraphs of an EFC process $\Pi$ to what we call a $Q$-$\lambda$ graph. Here $Q$ corresponds to the generator of a $\Xi$-coalescent, $\lambda$ refers to the fragmentation rate of the particles. In Section~\ref{S: Q_lambda_construction} we describe $Q$-$\lambda$ graphs as particle systems where an initial sample of size $n$, each particle fragments into two particles independently at rate $\lambda$, and where particles coalesce as a $\Xi$-coalescent. In Section~\ref{SS: discrete_time_graph} we show that the subgraph of the pedigree given by following all possible trajectories of a sample of $n$ lineages converges, with the $c_N^{-1}$ time-rescaling to a $Q$-$\lambda$ graph. In Section~\ref{S: quenched_critical_proof} we show that the coalescent of coalescing random walks on the discrete-time ancestral graph converges to that of a $Q$-$\lambda$ graph. By realizing the $Q$-$\lambda$ graph as a subgraph of an EFC process via a coupling argument, we finish the proof of Theorem~\ref{T: quenched_limited_outcrossing}.

    \subsection{Construction of the \texorpdfstring{$Q$-$\lambda$}{TEXT} graph}\label{S: Q_lambda_construction}

        Fix $\lambda \in \R_+$  and a sequence of consistent matrices $\lp{Q_n}_{n \in \N}$, and write $Q_n = \lp{q_{\xi\eta}^n}_{\xi,\eta \in \mathcal
        E_n}$. We assume they are
        obtained from Assumptions \ref{A: timescale} and \ref{A:Q_Nn} respectively. We define a stochastic process, called a a $Q$-$\lambda$ ancestral graph, with state space 
        \begin{equation}\label{Def:StateSpaceOmega}
            \Omega:=\Z_+ \times \Z_+ \times \mathcal{E},
        \end{equation}
        where $\mathcal{E} = \cup_{n \geq 2} \En$.
        For each $(l,m,\xi) \in \Omega$, the integer $l$ will track the current number of lineages, $m$ the label of the particle that most recently fragmented, and $\xi$ the structure of the most recent coalescence event. 
        
        This process is a random graph in which each node in the graph fragments independently at rate $\lambda$, and where any subset of $r$ nodes, represented by the singleton $\xi_0^r=\{\{i\}\}_{i=1}^r$, coagulate according to $\eta \in\mathcal{E}_r$ with rate $q_{\xi_0^r \eta}^r$. We provide the formal definition below.
        \begin{definition}\label{D: Q_lambda_graph}
            A \textbf{$Q$-$\lambda$ ancestral graph} is a continuous-time Markov process with sample paths in $\mathcal{D}\lp{\R_+, \Omega}$ and infinitesimal generator
            $\mathfrak{L}$ acting on $f\in C_b(\Omega)$ by
            \begin{equation}\label{E: graph_generator}
                \mathfrak{L}f(l,m,\xi)
                =
                \sum_{j=1}^l \lambda \lb{f(l+1, j, \xi)-f(l,m,\xi)} +  \sum_{\eta \in \mathcal{E}_l \setminus \{\xi_0^l\}} q_{\xi_0^l \eta}^l \lb{f(|\eta|, m, \eta) - f(l,m,\xi)}
            \end{equation}
        For $n\geq 1$, we denote by $G_{Q,\lambda}^n=\big(G_{Q,\lambda}^n(t)\big)_{t\in\R_+}$ a $Q$-$\lambda$ ancestral graph starting with $n$ nodes, i.e. when $G_{Q,\lambda}^n(0)=(n,m, \xi)$ for some $(m,\xi)\in \Z_+ \times \cup_{n \geq 2} \En$.
        \end{definition}

Existence of  $G_{Q,\lambda}^n$ holds because the fragmentation rate is linear in the first coordinate and does not depend on the the other two coordinates. This can be verified by, for instance, \cite[Proposition 2.9 in Chap. 4]{ethier2009markov}.

Next we describe $n$ independent random walks $\{x_i\}_{1 \leq i \leq n}$ on the ancestral graph $G_{Q,\lambda}^n$, which we view as $\Z_+$-valued processes. The joint process $\lp{G_{Q,\lambda}^n, x_1, \ldots, x_n}$ may be described by a generator $\tilde{\mathfrak{L}}$ defined on $C_b\lp{\Omega \times \Z_+^n}$. Define a vector $\sigma: \Z_+^n \times \Z_+ \to \{0,1\}^n$ by 
        \begin{equation*}
            \sigma(x,j)_i := \mathbbm{1}_{\{x_i = j\}}.
        \end{equation*}
        The vector $\sigma(x,j)$ in the $i$th position is $1$ if $x_i = j$ and $0$ otherwise. This gives a compact representation of all the walks that currently are in position $j$ on the graph. 
        
        Define another vector $\rho: \Z_+^n \times \mathcal{E} \to \Z_+^n$  where $\rho(x,\xi)_i$ is equal to the index of the block to which the $i$th particle belongs. To make this canonical, we may suppose that $\xi = \{C_1, C_2, \ldots, C_b\}$ is ordered by its least element. That is, for each $r < s$ we have that
        \begin{equation*}
            \inf C_r < \inf C_s.
        \end{equation*}
        $\rho(x,\xi)_i$ is therefore equal to the index of the unique block $C_{\rho(x,\xi)_i} = [x_i]_{\xi}$ of $\xi$, where $[x_i]_{\xi}$ is the equivalence class of $x_i$ with respect to $\xi$.

        \begin{definition}[Coalescing random walks on the ancestral graph]\label{D: random_walks_on_graph}
            The $n$ $\Z_+$-valued random walks $\{x_i\}_{1 \leq i \leq n}$ on the $Q$-$\lambda$ ancestral graph $G_{Q,\lambda}^n$ have their joint distribution with $G_{Q,\lambda}^n$ described by the generator $\tilde{\mathfrak{L}}$. They satisfy that
            \begin{enumerate}
                \item $x_i(0) = i$ for all $1 \leq i \leq n$,
                \item $x_i(t) \leq l(t)$ for all $1 \leq i \leq n$ and all $t \in \R_+$,
                \item at any fragmentation event, each of the $x_i$ will become $l+1$ or remain unchanged with equal (i.e. $\frac{1}{2}$) probability,
                \item once any two of the $n$ processes agree, they remain identical for all time thereafter.
            \end{enumerate}
        \end{definition}

        This system of coalescing random walks on the $Q$-$\lambda$ graph gives rise to a unique (in law) process with sample paths in $\mathcal{D}\lp{\R_+, \En}$, which we denote by $\chi^n_G$, when we keep track of the indices of these random walks. 
        \begin{definition}\label{D: coalescent_on_Q_lambda}
            We denote the law of this process $\chi^n_G$, conditional on $G_{Q, \lambda}^n$, by
            \begin{equation*}
                \mathcal{L}_{Q,\lambda}^n := \PP\lp{\chi^n_G \in \cdot \,\mid\, G_{Q,\lambda}^n}.
            \end{equation*}
        \end{definition}

    \subsection{Convergence of the discrete-time ancestral graph}\label{SS: discrete_time_graph}

        We begin by giving a construction of the discrete-time ancestral graph $G^{N,n} = \lp{l^N, m^N, \xi^N}$ as a discrete-time process taking values in $\Omega$.
        
        We will define $G^{N,n}$ in terms of another process, $\tilde{G}^{N,n}$ which is the subgraph of the population pedigree $\mathcal{G}_N$ given by all possible trajectories of the sample lineages $\lp{\hat{X}_i^N}_{1 \leq i \leq n}$, tracing all possible ancestral individuals of the sample as vertices of this subgraph.
        Let $\tilde{V}_k^N\subset[N]$ be the set of all possible ancestral individuals of the sample at time-step $k$ in the past.
        
        That is, $\tilde{G}^{N,n}$ consists of a sequence of  edges $\lp{\tilde{E}_k^N}$ where the edges from $\tilde{V}_k^N$ to $\tilde{V}_{k+1}^N$ consist of those edges from individual to itself for persisting in that time-step (i.e. not dying) and also edges from parents to offspring for reproductive events in this time-step. The edge set $\tilde{E}_k^N$ is merely the edges of the pedigree connecting those elements of $\tilde{V}_k^N$.
        
        We recall that $n$ is the sample size, and assume the initial conditions is $G^{N,n}(0) = \lp{n, 1, \xi_0^n}$. For $k\geq 1$, we define $G^{N,n}(k) = \lp{l^N(k), m^N(k), \xi^N(k)}$ as follows:
\begin{itemize}
    \item $l^N(k) := |\tilde{V}_k^N|$ is the number of distinct nodes in the subgraph $\tilde G^{N,n}$ at the $k$th time-step,
    \item $m^N(k)$ is the index of the node 
    in $\tilde G$ the subject of most recent outcrossing event, i.e. 
        \[
        m^N(k) = 
        \begin{dcases}
            m^N(k-1), \qquad \qquad  \qquad \qquad  \text{ if there are no outcrossing events in the $k$th time-step}\\
            \inf\left\{v \in \{1,2,\ldots, l^N(k)\} :\, \text{there exist }  (v,\pi), (v, \pi') \in \tilde{E}_k^N \text{ with } \pi \neq \pi' \right\}, \quad \text{ otherwise}
        \end{dcases}.
        \]
        \begin{remark}\rm
            Note that while it is possible in our model that more than one outcrossing event occurs in each time-step, this will occur with negligible probability over the lifespan of the graph. This is shown in Lemma~\ref{L: ancestral_graph_transition_rates}.
        \end{remark}
    \item         Label the vertices $\tilde{V}_k^N = \{\tilde{v}_i\}_{1 \leq i \leq l^N(k)}$ in order.
        We then define $\xi^N(k)$ as the unique partition in $\mathcal{E}_{l^N(k)}$ defined by the transitive closure of the binary relation $R^N(k)$
        \begin{equation*}
            i \sim_{R^N(K)} j \text{ if } \tilde{v}_i^N \text{  is the offspring of } \tilde{v}_j^N \text{  or vice-versa}.
        \end{equation*}
\end{itemize}
 Figure~\ref{F: ancestral_graph_fig} shows a realization of the ancestral graph $G^{N,n}$ for a sample of size $n = 2$.
 
        \begin{figure}
            \centering
            \includegraphics[width=0.5\linewidth]{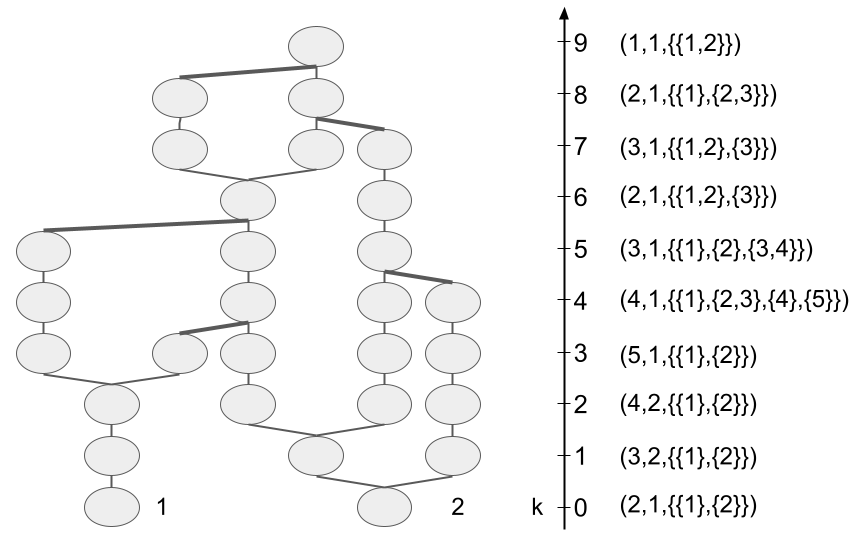}
            \caption{Here we see the discrete-time ancestral graph $G^{N,2}$ for a sample of size $n = 2$. Between the $0$-th and the $1$-st time-step $m$ changes because the second node, read left to right, fragments. Between the third and fourth time-step the second and third nodes coalesce, while the remaining three nodes are uninvolved in the coalescence.}
            \label{F: ancestral_graph_fig}
        \end{figure}
    
        \begin{remark}\rm\label{R: trivialization}
            Observe that if all of the reproductive events are via selfing then, then $R^N(k) = \xi^N(k)$. Notice that this is not the case with probability $O(1-\alpha_N)$. Further, we can compare the probability of this configuration of occurring with the transition matrix $Q_l^N.$ Indeed, for each $\eta \in \En$ the rate at which a coalescence with the structure $\eta$ occurs when there are $l$ extant sample lineages is precisely the sum of all the rates at which $l$ sample lineages in $l$ distinct individuals overlap into individuals with the structure $\eta$. That is, the rate at which $l$ sample lineages coagulate like $\eta$ is precisely the sum of all the rates at which the $l$ sample lineages in $l$ distinct individuals enter a state whose image under the haploid map is $\eta$, i.e. $q^{N,l}_{\xi_0^l\eta}$.
        \end{remark}

        Assume that the demography of the population in the general model of Section~\ref{A:Q_Nn}. We let $Q^{N} = \lp{Q^N_n}_{n \in \Z_+}$ denote the finite $N$ generator described therein, i.e.
        \begin{equation}\label{Def:QNn}
            Q^N_n = \lp{q_{\xi\eta}^{N,n}}_{\xi, \eta \in \En} = c_N^{-1}\lp{h_{\xi\eta}^{N,n} - \delta_{\xi\eta}}_{\xi,\eta \in \En},
        \end{equation}
        where $\delta_{\xi\eta}$ is the Kronecker delta function.
        We calculate the one-step transition probabilities for $G^{N,n}$ explicitly in terms of $Q^N_n$ in the following lemma.

        \begin{lemma}\label{L: ancestral_graph_transition_rates}
            Suppose that Assumption \ref{A: timescale} holds.
            Then the one-step transition probabilities of the Markov chain $\{G^{N,n}(k)\}_{k\in\Z_+}$ satisfy the following asymptotic property: given that $G^{N,n}(k) = (l,m,\xi)\in \Omega$, it holds that
            \begin{equation}\label{E: graph_transitions}
            G^{N,n}(k+1) =
                \begin{cases}
                    (l,m,\xi) &, \text{ w/prob. } 1-l^2O(c_N)\\
                    (l+1, m', \xi) &, \text{ w/prob. } d_N + l^2 o(c_N)\\
                    (l', m', \xi) &, \text{ w/prob. } l^2 o(c_N)\\
                    (l'', m', \eta) &, \text{ w/prob. }  l^3 O(d_N c_N)\\
                    (|\eta|, m, \eta) &, \text{ w/prob. } c_Nq^{N,l}_{\xi_0^l \eta} \alpha_N^l
                \end{cases},
            \end{equation}           
            for each $(m',l',l'', \eta)$ such that $m' \neq m$ and $l' \neq  l+1$, $l'' \neq |\eta|$, and $\eta \in \mathcal{E}_l$. In the above, the error terms $O(c_N),o(c_N),O(d_N c_N)$ are all uniform for $(l,m,\xi)\in \Omega$.
        \end{lemma}
        \begin{remark}\rm
            Note that a sequence $\{a_N\}$ belongs to $O(b_N)$ if there exists a constant $C$ such that $a_N \leq C b_N$ for all $N\geq 1$. A sequence $\{a_N\}$ belongs to $o(b_N)$ if $\lim_{N \to \infty} \frac{a_N}{b_N} = 0$. Note that $O(b_N)$ and $o(b_N)$ consist of collections of real-valued sequences and so set-valued inclusions between $O(a_N)$ and $o(b_N)$ make sense. We write $O(b_N) \subset o(c_N)$ if $\lim_{N \to \infty} \frac{a_N}{c_N} = 0$ for any sequence $\{a_N\}\in O(b_N)$.
        \end{remark}
            
        \begin{proof}
            Observe that the probability of either an outcrossing event or a coalescent event among the $l$ nodes is bounded above by $2c_N\binom{l}{2} + ld_N$. By Assumption~\ref{A: timescale}, the probability of neither an outcrossing event nor a coalescent event is $1-l^2O(c_N)$. This gives the first line.

            Now we show that the probability that two outcrossing events occur in the same time-step occurs with negligible probability, implying line 3. Indeed, the probability that two outcrossing events occur is bounded by the probability that any of the $\binom{l}{2}$ possible pairs of particles are both offspring and both are the result of outcrossing. The probability that any such pair consists of two offspring is $\mathbb{E}\lb{\frac{K_N(K_N-1)}{N(N-1)}}$. The probability that they both are the result of outcrossing is $(1-\alpha_N)^2$. Therefore the probability of at least two outcrossing events is at most
            \begin{equation*}
                \binom{l}{2}(1-\alpha_N)^2 \mathbb{E}\lb{\frac{K_N(K_N-1)}{N(N-1)}}
                \leq \binom{l}{2}d_N (1-\alpha_N).
            \end{equation*}
This, together with the fact that $\alpha_N\to 1$ (by Assumption~\ref{A: timescale}), implies that the probability of at least two outcrossing events is bounded by $\binom{l}{2} o(d_N)$. This gives the third line of \eqref{E: graph_transitions}.

            For any $\eta \in \En$ not equal to $\xi_0^l$, the probability in one time-step that there is any coalescence of the $l$ nodes in the graph with structure $\eta$ and without outcrossing is 
            $$p_{\xi_0^l \eta}^{N,l} \alpha_N^l = c_Nq_{\xi_0^{l} \eta}^{N,l} \alpha_N^l,$$
            which gives the fifth line of \eqref{E: graph_transitions}.

            We now show that an outcrossing event occurs during a coalescence with negligible probability, implying the fourth line. 
            As $\xi_0^l$ consists of particles in $l$ distinct individuals, the coalescence rate is independent to the selfing rate. In particular, the probability that we have a coalescence with structure $\eta$ and an outcrossing event is therefore $O((1-\alpha_N)q^{N,l}_{\xi\eta}) \subset o(c_N)$. This gives the fourth line  of \eqref{E: graph_transitions}.
        \end{proof}

        Equipped with Lemma~\ref{L: ancestral_graph_transition_rates} giving the transition rates of the discrete-time ancestral graph $G^{N,n}$, we are prepared to demonstrated convergence of the sped-up ancestral graph to a $Q$-$\lambda$ graph.

        \begin{lemma}\label{L: graph_convergence}
            Suppose that Assumptions~\ref{A:c_N}, \ref{A: timescale}, and \ref{A:Q_Nn} hold.  Then the sped-up discrete-time ancestral graph $\bar{G}^N = \lp{G^N\lp{\lfloor t c_N^{-1}\rfloor}}_{t \in \R_+}$ converges weakly in the $J_1$ topology in $\mathcal{D}\lp{\R_+, \Omega}$ to a $Q$-$\lambda$ ancestral graph $G^n_{Q,\lambda}$ with respect to the measures $\PP_{\xi_0^n}^N$, where $Q = \lp{Q_n}_{n \in \N}$ is as in Assumption~\ref{A:Q_Nn}.
        \end{lemma}

        \begin{proof}

            We proceed as in the proof of \cite[Lemma 7.6]{NFW25}.
            Let $T^{(N)}$ denote the linear operator on the space $C_b\lp{\Omega}$ of bounded continuous functions on $\Omega$ defined by $T^{(N)}f(l,m,\xi) = \mathbb{E}\lb{f\lp{G^N(1)} \mid G^{N,n}(0) = (l,m,\xi)}$. The generator $\mathfrak{L}^N$ of the discrete-time process $G^{N,n}$ is given by
            \begin{equation*}
                \mathfrak{L}^N f(l,m,\xi) = \lp{T^{(N)}-I}f(l,m,\xi),
            \end{equation*}
            which can be calculated to sufficient accuracy as in Equation~\ref{E: graph_transitions}. Let $\mathfrak{L}$ be the infinitesimal generator of the continuous-time process $G_{Q,\lambda}^n := \lp{G_{Q,\lambda}^n(t)}_{t \in \R_+}$. That is, by Equation \eqref{E: graph_generator},
            \begin{equation}\label{E: generator_difference}
                \mathfrak{L}f(l,m,\xi) =
                \sum_{j=1}^l \lambda \lb{f(l+1, j, \xi)-f(l,m,\xi)} +  \sum_{\eta \in \mathcal{E}_l \setminus \{\xi_0^l\}} \tilde{q}_{\xi_0^l \eta}^l \lb{f(|\eta|, m, \eta) - f(l,m,\xi)}
            \end{equation}
            It follows from Equation \eqref{E: graph_transitions} that for all $f \in C_b(\Omega)$ with finite support that,
            \begin{equation*}
                \sup_{(l,m,\xi) \in \Omega} \la{c_N^{-1} \mathfrak{L}^Nf(l,m,\xi) - \mathfrak{L}f(l,m,\xi)} \to 0 \text{ as $N \to \infty$.}
            \end{equation*}
            Let $\{T(t)\}_{t \in \R_+}$ be the semigroup on $C_b(\Omega)$ of $G_{Q,\lambda}^n$. By \cite[Theorem 6.6, Chapter 1]{ethier2009markov} and \eqref{E: generator_difference}, it holds that
            \begin{equation}\label{E: semigroup_difference}
                \lim_{N \to \infty} \sup_{0 \leq t \leq T} \sup_{(l,m,\xi) \in \Omega}
                \la{\lp{T^{(N)}}^{\lfloor t c_N^{-1} \rfloor}f(l,m,\xi) - T(t) f(l,m,\xi)} = 0
            \end{equation}
            for all $f$ in the domain of $\mathfrak{L}$
            
            We now proceed to demonstrate tightness. Observe that, for any $n$ that the set $K_n := [n] \times [n] \times \cup_{2 \leq k \leq n} \mathcal{E}_k$ is compact. To establish the compact containment condition \cite[(7.9), p.129]{ethier2009markov} holds if for any $\varepsilon \in (0,1)$ and $T \in (0,\infty)$, there exists $M(\varepsilon, T) \in \Z_+$ such that
            \begin{equation}\label{E: CCC}
                \limsup_{N \to \infty} \PP_{\xi_0^n}^N\lp{G^N(t) \in K_{M(\varepsilon, T)} \, \, \forall 0 \leq t \leq T} \leq \varepsilon.
            \end{equation}
            Notice that for all time $t\in\R_+$,
            \begin{equation*}
              \max\left\{  |\xi^N(t)|, \,|m^N(t)| \right\}\leq \sup_{0 \leq s \leq t} l^N(t).
            \end{equation*}
            Therefore, to demonstrate \eqref{E: CCC}, it suffices to show that
            \begin{equation}\label{E: CCC_bis}
                \limsup_{N \to \infty} \PP_{\xi_0^n}^N\lp{\sup_{0 \leq t \leq T} l^N(t) \leq M(\varepsilon, T)} \leq \varepsilon.
            \end{equation}

            Let $\mathcal{Q}$ denote the generator of the cross-section size $L = \lp{L_t}_{t \in \R_+}$ for the limiting graph. Then for any $f$ in $C_b(\Z_+)$ we have that
            \begin{equation*}
                \mathcal{Q}f(n) = \lambda n \lp{f(n+1)-f(n)}
                +
                \sum_{i=1}^{n-1}\sum_{|\eta| = i} \tilde{q}_{\xi_0^l \eta}^n \lp{f(|\eta|) - f(n)}.
            \end{equation*}
            It is a general fact about Markov processes that for any $f$ in $C_b(\Z_+)$ that 
            \begin{equation}\label{E: mtg}
                M(t) := f(L_t) - f(L_0) - \int_0^t \mathcal{Q}f(L_s) ds
            \end{equation}
            is a martingale with quadratic variation
            \begin{equation}\label{E: M_QV}
                \qv{M}_t = \int_0^t \mathcal{Q}(f^2)(L_s) - 2f(L_s)\mathcal{Q}f(L_s)ds;
            \end{equation}
            see \cite[Lemma 5.1]{kipnis1998scaling} or \cite[Proposition 4.1.7]{ethier2009markov}. A truncation argument will enable us to take $f$ to be the identity function $Id$ (i.e. when $Id(n) = n$ for all $n \in \Z_+$) to obtain
            \begin{equation}\label{E: L_characterization}
                L_t = L_0 + \int_0^t \mathcal{Q}(Id^2)(L_s) - 2L_s\mathcal{Q}(Id)(L_s) ds + M^{Id}(t) \text{ \quad for all } t\in \R_+,
            \end{equation}
            where $M^{Id}$ corresponds to the suitable martingale defined in \eqref{E: mtg} with $f = Id$.
            We calculate directly that
            \begin{align}\label{E: generator_calculation}
                \mathcal{Q}(Id)(n) &= \lambda n + \sum_{i=1}^{n-1} \sum_{|\eta| = i} \tilde{q}_{\xi_0^n \eta}^n (i-n) \leq \lambda n,\\
                \mathcal{Q}(Id^2)(n) &= \lambda n (2n+1) +
                \sum_{i=1}^{n-1} \sum_{|\eta| = i} \tilde{q}_{\xi_0^n \eta}^n(i^2 - n^2), \text{ and }\\
                \mathcal{Q}(Id^2)(n)-2n\mathcal{Q}(Id)(n) &=
                n\lp{\lambda - \la{\tilde{q}_{\xi_0^n \xi_0^n}^n}} + \sum_{i=1}^{n-1} \sum_{|\eta| = i} i \tilde{q}_{\xi_0^n \eta}^n \leq n\lambda.
            \end{align}

            Combining \eqref{E: L_characterization}, \eqref{E: generator_calculation}, and \eqref{E: M_QV} yields
            \begin{align}\label{E: inequalities_martingale}
                L_t &\leq L_0 + \lambda \int_0^t L_s ds + M^{Id}(t),  \text{ and }\\
                \left\langle M^{Id} \right\rangle_t &\leq \lambda \int_0^t L_s ds.
            \end{align}
            
            This allows us to compare $L$ with the solution $X$ to
            \begin{eqnarray*}
                dX_t = \lambda X_{t-} + dM^{Id}_t,\quad   X_0 = L_0,
            \end{eqnarray*}
            using the observation that
            \begin{equation}\label{E: martingale_comparison}
                \PP\lp{\sup_{t \in [0,T]} L_t \geq K} \leq \PP\lp{\sup_{t \in [0,T]} X_t \geq K}.
            \end{equation}
            Observe that $X_t = e^{\lambda t}\lp{L_0+ N_t}$ where $N_t =\int_0^t e^{-\lambda s}dM^{Id}(s)$.
            Therefore
            \begin{equation}\label{E: decomposition_martingale}
                \sup_{t \in [0,T]} X_t \leq e^{\lambda T}\lp{L_0 + \sup_{t \in [0,T]} \int_0^t e^{-\lambda s} dM^{Id}(s)}.
            \end{equation}
            By \eqref{E: M_QV}, \eqref{E: martingale_comparison}  and Doob's inequality applied to \eqref{E: decomposition_martingale} we get that there exists $M(\varepsilon, T)$ such that
            \begin{equation*}
                \PP\lp{\sup_{0 \leq t \leq T} l(t) \leq M(\varepsilon, T)} \leq \varepsilon.
            \end{equation*}

            Finally, \eqref{E: CCC_bis} can be obtained by the same argument above, when $L_t$ and $\int_0^t\mathcal{Q}f(L_s)ds$ are replaced, respectively, by $l^N(k)$ and $\sum_{i=0}^{k-1} \mathcal{Q}^N f(l^N(i))$ for $\mathcal
            Q^N$ the generator of $l^N$. Indeed, the uniform control in $l$ given by Lemma~\ref{L: ancestral_graph_transition_rates}, which depends on Assumptions~\ref{A:c_N} and \ref{A: timescale}, is enough for the argument above to follow as Equation~\ref{E: generator_calculation} holds with the limiting transition rates $q^n_{\xi\eta}$ replaced by the finite $N$ transition rates $q^{N,n}_{\xi\eta}$, which follows by Assumption~\ref{A:Q_Nn} (here $\xi,\eta$ are simply placeholders for the true elements used in the above equation). \eqref{E: CCC_bis} implies the convergence of $\bar{G}^N$ in $\mathcal{D}\lp{\R_+, \Omega}$ to $G_{\tilde{Q}, \lambda}^n$ by \cite[Corollary 8.9, Chapter 4]{ethier2009markov} and \eqref{E: semigroup_difference}.        
            
        \end{proof}

        The discrete-time ancestral graph $G^{N,n}=(G^{N,n}(k))_{k\in\Z_+}$ can be viewed graphically as having $n$ initial nodes (the nodes at the bottom of Figure \ref{F: ancestral_graph_fig}). These $n$ nodes  correspond to the $n$ individuals $\{\hat{X}_i^N(0)\}_{1 \leq i \leq n}$ from whom we sampled the $n$ lineages at time-step $0$. Viewed as a $\Omega$-valued process, this graph has the following key advantages:
        it get rid of the labels of all individuals, while  keeping track of the set of all possible ancestral individuals (of the sample) and their parental relationships.

        As mentioned, the ancestral lines $\{X_i^N\}$ defined in \eqref{E:Alines} is a system of coalescing random walks on the pedigree. We shall construct a system of coalescing random walks, view as $\Z_+$-valued processes, on the reduced object $G^{N,n}$ as follows.
                
        Define $x_i^N(0) := i$ for all $1 \leq  i \leq n$. Suppose $x_i^N(j)$ is defined for all $j \leq k$ for induction. If $\hat{X}_i^N(k)$ experiences an outcrossing event in the $k$th time-step, then $x_i^N(k+1)$ is equal to $l^N(k+1)$ with probability $\frac{1}{2}$ or else remains equal to $x_i^N(k)$ with probability $\frac{1}{2}$. (In the event that there is more than one outcrossing event in a single time-step, it may be that each outcrossing individual coalesces in the position $l+1$; in the limited outcrossing regime, this possibility is negligible, so we do not need to consider this case in any detail.) We therefore have $x_i^N = \lp{x_i^N(k)}_{k \in \Z_+}$ defined for all time. If any two walks $x_i^N$ and $x_j^N$ coalesce at any point in time, they are then required to remain together for all time thereafter. This gives a family of $n$ simple random walks on $G^{N,n}$ satisfying
        \begin{enumerate}
            \item $x_i^N(k) \leq l^N(k) $ for all $k \in \Z_+$,
            \item at any outcrossing event, each random walk follows each of the two paths available with equal probability,
            \item and once two random walks coalesce they remain together for all time thereafter.
        \end{enumerate}

\medskip

 We may also define an associated coalescent process based on these $x_i^N$. 
        \begin{definition}\label{D: discrete_ancestral_graph_coalescent}
         Let  $\Pi_{G}^{N,n} = \lp{\Pi_{G}^{N,n}(t)}_{ \in \R_+}$ be  the sped-up coalescent process defined by
            \begin{equation*}
                i \sim j \text{ with respect to } \Pi_{G}^{N,n}(t) \text{ if } x_i^N\lp{\lfloor t c_N^{-1} \rfloor} = x_j^N\lp{\lfloor t c_N^{-1} \rfloor},
            \end{equation*}
        and  define the quenched law
            \begin{equation}\label{Def:LNnG}
                \mathcal{L}^{N,n}_{G} := \PP\lp{\Pi^{N,n}_{G} \in \cdot \,\mid\, G^{N,n}}.
            \end{equation}
        \end{definition}
        
        \begin{remark}\rm
            Note that $\mathcal{L}^{N,n}_G$ defined in \eqref{Def:LNnG} is equal to $\PP\lp{\chi^{N,n}(\lfloor t c_N^{-1} \rfloor) \in \cdot \,\mid\, \mathcal{A}_N}$ which admits no random time change, unlike that of $\mathcal{L}^{N,n}$. This is because, unlike the random walks on the pedigree, we assume being in the same node in the ancestral graph $G^{N,n}$ means instantaneous coalescence of the particles. 
                We will show in Lemma \ref{L: G_and_bar_same_limits} that $\mathcal{L}^{N,n}$ and $\mathcal{L}^{N,n}_G$ have the same weak limit as $N\to\infty$. 
        \end{remark}

A suitable notion of continuity for random walks on $\Omega$-valued processes, established in Section~\ref{S: theta_walks}, will then give the claimed result Theorem~\ref{T: quenched_limited_outcrossing}.

    \subsection{Coalescing random walks on \texorpdfstring{$\Omega$-valued}{} processes}\label{S: theta_walks}

        To each element $g = \lp{g(t)}_{t \in\R_+} = \lp{l(t), m(t), \xi(t)}_{t \in \R_+}$ in the Skorokhod space $\mathcal{D}\lp{\R_+, \Omega}$, we will define an associated coalescent process $\chi(g)$. Let $g(t-) = \lim_{s \uparrow t} g(s) = (l(t-), m(t-),\xi(t-))$ denote the left limit of $g$ at $t$. We suppose that $g(0) = (l_0, m_0, \xi_0)$ and define, for all $1 \leq i \leq l_0$ that $x_i(0) = i$. The $x_i$ may be extended to all time as follows:
        \begin{itemize}
            \item On any interval $[s,t)$ on which $l$ is constant, each $x_i$ stays constant.
            \item At each jump time $t$ of $g$, if $l(t) < l(t-)$, then we set
            \begin{equation*}
                x_i(t) = \inf \{[x_i(t-)]_{\xi(t)}\},
            \end{equation*}
            where $[a]_{\xi}$ denotes the block of $\xi$ containing $a$. If the infimum is taken over an empty set, then $x_i(t) = x_i(t-)$.
            \item At each jump time $t$ of $g$, if $l(t) > l(t-)$, then we let $C(t)$ be an independent random variable taking value in $\{m(t), l(t)\}$ where $\PP\lp{C(t) = m(t)} = \PP\lp{C(t) = l(t)} = \frac{1}{2}$. Then for every $i$ in the set $I := \{i : x_i(t-) = m(t)\}$, we let $x_i(t) := C(t)$. If $i$ is not in $I$, then $x_i(t) = x_i(t-)$.
        \end{itemize}

        The $\{x_i\}_{i=1}^n$ as $\Z_+$-valued processes take the form a family of coalescing random walks on $g$. We define their associated coalescent $\Pi^n(g)$ by
        \begin{equation*}
            i \sim j \text{ with respect to } \Pi^n(g)(t) \text{ if } x_i(t) = x_j(t).
        \end{equation*}

        \begin{remark}\rm
            Note that  $\Pi^{N,n}_{G} $ in Definition \ref{D: discrete_ancestral_graph_coalescent} is equal to $\Pi^n(\bar{G}^N)$, where $\bar{G}^N$ is the sped-up ancestral graph of Lemma~\ref{L: graph_convergence} and $\Pi^n$ is the map defined above.
        \end{remark}

        \begin{lemma}\label{L: phi_continuity}
        Let $\mathcal{E}_{\infty} := \bigsqcup_{n \in \N} \mathcal{E}_n$ and recall the space $\Omega$ in \eqref{Def:StateSpaceOmega}. The measure-valued map $\Phi: \mathcal{D}\lp{\R_+, \Omega} \to \mathcal{M}_1\lp{\mathcal{D}\lp{\R_+, \mathcal{E}_{\infty}}}$ defined by
        \begin{equation*}
            \Phi(g) := \PP\lp{\Pi(g) \in \cdot \,\mid\, g}.
        \end{equation*}
        is  continuous.
        \end{lemma}
        \begin{proof}
            Suppose that $g_N$ is a sequence of elements in $\mathcal{D}\lp{\R_+, \Omega}$ that converge therein to an element $g$. We need to show that $\Phi(g_N) \towd \Phi(g)$. That is, we need to show that for any dense subset $\mathcal{F}$ of functions $f$ in $C_b(\mathcal{D}\lp{\R_+, \mathcal{E}_{\infty}})$ that $\int f d\Phi(g_N)$ converges to $\int f d\Phi(g)$. As $\Omega$ is a locally compact Polish space, by Theorem~\ref{T: random_measure_integral_functionals} it suffices to check that
            \begin{equation*}
                \lp{\Phi(g_N)\lp{\int_0^\infty \varphi_i(s, x(s)) ds}}_{1 \leq i \leq k} \toL \lp{\Phi(g)\lp{\int_0^\infty \varphi_i(s, x(s)) ds}}_{1 \leq i \leq k}
            \end{equation*}
            for any collection $\lp{\varphi_i}_{1 \leq i \leq k}$ of elements in $C_c\lp{\R_+ \times \Omega}$.

            As $g_N$ converges to $g$ and $\Omega$ is a discrete space we have that, for large enough $N$, that $g_N(0)$ is eventually constant. In particular, we may restrict to the case that $g_N(0) = (n, m_0, \xi_0)$ for all $N$. This allows us to consider continuity of the restriction of $\Phi$ as a map into $\mathcal{M}_1\lp{\mathcal{D}\lp{\R_+, \mathcal{E}_n}}$ without loss of generality.

            As each $\varphi_i$ is compactly supported on $\R_+ \times \Omega$, there is a $T > 0$ such that $\cup_i supp(\varphi_i) \subset [0,T] \times \Omega$. Notice that, for any $N\geq 2$,
            \begin{equation}\label{E: integral_evaluation}
                \lp{\Phi(g_N)\lp{\int_0^\infty \varphi_i(s, x(s)) ds}}_{1 \leq i \leq k} = \lp{\mathbb{E}\lb{\int_0^T \varphi_i(s,  \Pi^n(g_N)(s)) ds \in \cdot \,\mid\, g_N}}_{1 \leq i \leq k}.
            \end{equation}
            The quantity on the right-hand-side of \eqref{E: integral_evaluation} depends only on discrete graph structure of $g_N$, i.e. the finite sequence of values that $g_N$ takes, and the edge lengths of the graph $g_N$, i.e. the jump times. As $\Omega$ is discrete and $g_N \to g$ in the Skorokhod topology, for large enough $N$ the discrete graph structure of $g_N$ is eventually constant. Further, the edge lengths of $g_N$ converge uniformly on $[0,T]$ to those of $g$. This yields the desired convergence.

            As such, the coalescing random walks on $g_N$ and $g$ may be coupled in such a way that, on $[0,T]$, they make the same jumps at each fragmentation event as each other. In particular, we can take it so that $\Pi^n(g_N)$ converges pointwise almost surely on $[0,T]$ to $\Pi^n(g)$. Because the integrands are bounded and convergence is pointwise almost surely on $[0,T]$, the dominated convergence theorem yields the claim.
            
        \end{proof}

        As a corollary, we have that if $G_N$ is a sequence of $\mathcal{D}\lp{\R_+, \Omega}$-valued random variables converging weakly to $G$, then $\Phi(G_N)$ converges weakly in law to $\Phi(G)$. This is a key observation in the proof of Theorem~\ref{T: quenched_limited_outcrossing}.

    \subsection{Proof of Theorem~\ref{T: quenched_limited_outcrossing}}\label{S: quenched_critical_proof}

        To prove Theorem~\ref{T: quenched_limited_outcrossing}, we establish a lemma demonstrating the equivalence of the quenched coalescent for the $Q$-$\lambda$ graph, as in Definition~\ref{D: coalescent_on_Q_lambda}, and the quenched coalescent for an EFC, as in Definition~\ref{D: EFC_walks_coalescent}. Before the formal statement of the lemma, we give an intuitive explanation for why the two constructions agree. Fix any finite sample of at most $n$ particles and follow their coalescing random-walk trajectories. Under both the $Q$-$\lambda$ graph and the EFC, coalescences of the walks occur with the same structures and rates, and fragmentations of each walk occur with the same rates. Because these rates uniquely determine the law of the induced coalescent on partitions and because both the $Q$-$\lambda$ graph and EFC give the same amount of information on the particle trajectories, the quenched coalescent laws coincide. The lemma that follows makes this argument rigorous via a joint coupling of the $Q$-$\lambda$ graph, the EFC, and random walks thereon.

        \begin{lemma}\label{L: equivalence_EFC_Q_lambda}
            Suppose that $Q = \lp{Q_n}_{n \in \N}$ is a consistent family of generators for a $\Xi$-coalescent for some measure $\Xi$ satisfying \eqref{E:TotalTimescale} (i.e. $2=\Xi(\Delta)$). Let $\Xi_0$ be the restriction of $\Xi$ on $\Delta\setminus \{\bf{0}\}$. Let $\Pi$ be an EFC process with characteristics $c_k = \Xi(\mathbf{0})$, 
            $\nu_{Coag} = \nu_{\Xi_0}$, $c_e = \lambda$, and $\nu_{Disl} = 0$. Then the random measures $\mathcal{L}_{\Pi}^n$ and $\mathcal{L}_{Q,\lambda}^n$ are equal in distribution.
        \end{lemma}

        \begin{proof}
            For a given $Q$-$\lambda$ graph $G = \lp{(l(t), m(t), \xi(t))}_{t \geq 0}$ and an EFC process $\Pi = \lp{\Pi(t)}_{t \geq 0}$, we let $x^G = \lp{x_i^G}_{1 \leq i \leq n}$ and $x^\Pi = \lp{x_i^\Pi}_{1 \leq i \leq n}$ denote the family of coalescing random walks on $G$ and on $\Pi$, respectively. Let $\chi^n_G$ and $\chi^n_\Pi$ denote the associated coalescents for $x^G$ and $x^\Pi$, respectively. We describe a coupling of $\Pi$, $G$, $x^\Pi$, and $x^G$ together so that the claim holds. Specifically, it will be such that $\chi^n_G = \chi^n_\Pi$ pointwise, and such that $\PP\lp{\chi^n_\Pi \in \cdot \,\mid\, \Pi} = \PP\lp{\chi^n_\Pi \in \cdot \,\mid\, G}$, which together prove the claim. The construction follows that for the EFC \cite[Section 3.2]{berestycki04}.

            We take $\Pi(0) = \xi_0$, $G(0) = \lp{n,0,\xi_0^n}$, $x^G(0) = x^\Pi(0) = \lp{1,2,\ldots, n}$, and $\xi(0) = \xi_0^n$. Let $P_C = \lp{(t, \xi^{(C)}(t))}_{t \geq 0}$ and $P_F = \lp{\lp{t, \xi^{(F)}(t), k(t)}}_{t \geq 0}$ be two independent Poisson point processes (PPPs) on the same filtration. The atoms of $P_C$ are points in $\R_+ \times \mathcal{E}_\infty$ and $P_C$ has intensity measure $dt \otimes C$. The atoms of $P_F$ are points in $\R_+ \times \mathcal{E}_\infty \times \N$ and has intensity measure $dt \otimes F \otimes \#$, where $\#$ is the counting measure on $\N$. We utilize $P_F$ and $P_C$ to construct $(\Pi, G, x^G, x^\Pi)$ jointly.

            To this end, we establish notation for the restriction of partitions of $\N$ to finite subsets. For $\xi \in \mathcal{E}_\infty$ and a finite subset $I = \left\{i_j\right\}_{1 \leq j \leq l}$ of $\N$, we define $\xi(I)$ to be the partition of $[l]$ defined by
            \begin{equation*}
                j \sim_{\xi(I)} k \text{ if } i_j \sim_{\xi} i_k.
            \end{equation*}
            With this notation, the coupling may be described as follows:
            \begin{enumerate}
                \item If $t$ is neither an atom time of $P_F$ nor $P_C$, then $\Pi(t) = \Pi(t-)$, $G(t) = G(t-)$, $x^\Pi(t) = x^\Pi(t-)$, and $x^G(t) = x^G(t-)$.
                \item If $t$ is an atom time of $P_C$, then
                \begin{enumerate}
                    \item $\Pi(t) = Coag(\Pi(t-), \xi^{(C)}(t))$,
                    \item $G(t) = \lp{\la{\xi^{(C)}(t)\lp{\{x_i^\Pi(t-)\}_{1 \leq i \leq n}}}}, m(t-), \xi^{(C)}(t)\lp{\{x_i^\Pi(t-)\}_{1 \leq i \leq n}}$,
                    \item for all $1 \leq i \leq n$, $x_i^\Pi(t) = \inf\lb{x_i^\Pi(t-)}_{\xi^{(C)}(t)}$, where $[a]_\xi$ denotes the block in $\xi$ to which $a$ belongs,
                    \item and $x_i^G(t)$ is equal to the index of the block of $\xi^{(C)}(t)\lp{\{x_i^\Pi(t-)\}_{1 \leq i \leq n}}$ to which $x_i^\Pi(t-)$ belongs, when ordered by their least element.
                \end{enumerate}
                \item If $t$ is an atom time of $P_F$, then
                \begin{enumerate}
                    \item $\Pi(t) = Frag(\Pi(t-), \xi^{(F)}(t), k(t))$,
                    \item if $k(t) \notin \{x_i^\Pi(t)\}_{1 \leq i \leq n}$, then $G(t) = G(t-)$, $x^\Pi(t) = x^\Pi(t-)$, and $x^G(t) = x^G(t-)$,
                    \item and if $k(t) \in \{x_i^\Pi(t)\}_{1 \leq i \leq n}$, then $G(t) = \lp{l(t-), m(t), \xi(t-)}$, where $m(t)$ is equal to $x_i^G(t-)$ (for any $i$ such that $x_i^\Pi(t-) = k(t)$; check that this is indeed well-defined); for $H$ an independent Bernoulli random random variable with $\PP\lp{H = 0} = \PP\lp{H = 1} = \frac{1}{2}$ and if the only non-singleton block of $\xi^{(F)}(t)$ is $\{r,s\}$, then
                    \begin{equation*}
                        x_i^\Pi(t) = 
                        \begin{cases}
                            x_i^\Pi(t-) &, \text{ if } k(t) \neq x_i^\Pi(t-) \text{ or } H=0\\
                            s &, \text{ if } k(t) = x_i^\Pi(t-), H = 1
                        \end{cases}
                    \end{equation*}
                    and
                    \begin{equation*}
                        x_i^G(t) = 
                        \begin{cases}
                            x_i^G(t-) &, \text{ if } x_i^\Pi(t) = x_i^\Pi(t-)\\
                            l(t) &, \text{ if } x_i^\Pi(t) = s
                        \end{cases}.
                    \end{equation*}
                \end{enumerate}
            \end{enumerate}

            By construction of $P_F$ and $P_C$ it is manifest that $\Pi$ is an EFC as described and $G$ is a $Q$-$\lambda$ graph. Further, it is indeed the case that $\chi^n_G = \chi^n_\Pi$ pointwise. As the structure of coalescent events and fragmentation events for the $x^\Pi$ are fully determined by $G$ we indeed have that $\PP\lp{\chi^n_\Pi \,\mid\, \Pi} = \PP\lp{\chi^n_\Pi \,\mid\, G}$. This yields the claim.

        \end{proof}

        We now show that if $\mathcal{L}^{N,n}_G$ converges in distribution as a random measure to a random measure $\mathcal{L}$, then so does $\mathcal{L}^{N,n}$. In particular, these two sequences have the same distributional limit. The spirit of the argument is really that the path of elements in $\En$ for $\bar{\chi}^{N,n}$ and $\Pi^{N,n}_G$ are the same with high probability, and that the two differ on a set of negligible measure as $N \to \infty$. As we compare these measures via integral functionals of compactly supported test functions, in the sense of Theorem~\ref{T: random_measure_integral_functionals}, this is sufficient to show that the differences in their evaluations of these integral functions vanishes as $N$ tends to infinity.

        \begin{lemma}\label{L: G_and_bar_same_limits}
            Suppose that, as $N \to \infty$, that $c_N \to 0$, $d_N c_N^{-1} \to \lambda < \infty$, and $\mathcal{L}^{N,n}_G\towd\mathcal{L}$. Then $\mathcal{L}^{N,n}\towd\mathcal{L}$.
        \end{lemma}

        \begin{proof}
            We apply Theorem~\ref{T: random_measure_integral_functionals} to demonstrate the claim. To this end, fix $\{\varphi_i\}_{i=1}^k \subset C_c(\R_+ \times \En)$ a finite collection of test functions. By Theorem~\ref{T: random_measure_integral_functionals} and the assumption that $\mathcal{L}^{N,n}_G \towd \mathcal{L}$ we have that
            \begin{align}\label{E: integral_convergence_assumption}
                \lp{\mathcal{L}^{N,n}_G\lp{\int_{\R_+} \varphi_i(s,x(s))ds}}_{1 \leq i \leq k}
                &= \lp{\mathbb{E}\lb{\int_{\R_+} \varphi_i\lp{s, \Pi^{N,n}_G(s)}ds \, \mid \, G^N}}_{1 \leq i \leq k}\\
                &\toL \lp{\mathcal{L}\lp{\int_{\R_+} \varphi_i(s,x(s))ds}}_{1 \leq i \leq k}.
            \end{align}
            For the claim to hold, by Theorem~\ref{T: random_measure_integral_functionals} and Equation~\eqref{E: integral_convergence_assumption}, it suffices to show that, for any one of the $\varphi_i$, that
            \begin{equation}\label{E: difference}
                \mathcal{L}^{N,n}\lp{\int_{\R_+} \varphi_i(s,x(s))ds} - \mathcal{L}^{N,n}_G \lp{\int_{\R_+} \varphi_i(s,x(s))ds}
            \end{equation}
            converges in probability to $0$. To this end, we fix $\varphi$ in $C_c(\R_+ \times \En)$ and show precisely this convergence. We proceed by analyzing how $\bar{\chi}^{N,n}$ and $\Pi^{N,n}_G$ differ.

            By construction of $\Pi_G^{N,n}$ we have that, if $x_i^N(t)=x_j^N(t)$, then $i$ and $j$ belong to the same block in $\Pi^{N,n}_G(t)$. However, for the corresponding random walks on the pedigree $\hat{X}_i^N(\lfloor tc_N^{-1} \rfloor)=\hat{X}_j^N(\lfloor t c_N^{-1} \rfloor)$ does not imply that $i$ and $j$ belong to the same blocks in $\bar{\chi}^{N,n}(t)$. Instead, it means that $\chi^{N,n}(\lfloor t c_N^{-1} \rfloor) \notin \En$, and so we are experiencing the random time-change where $S(\lfloor tc_N^{-1} \rfloor) \neq \lfloor t c_N^{-1} \rfloor$. The two blocks containing $i$ and $j$ may either coalesce before $\chi^{N,n}$ reenters $\En$ or else disperse. We claim that dispersal occurs with probability tending to $1$ as $N \to \infty$, and that the time $T$ until either coalescence or dispersal satisfies $c_N T$ converges in probability to $0$. This will demonstrate further, at all times where $S(\lfloor t c_N^{-1} \rfloor) = \lfloor tc_N^{-1} \rfloor $, that $\Pi_G^{N,n}(t) = \bar{\chi}^{N,n}(t)$ with probability tending to $1$ as $N \to \infty$, for each $t\in(0,\infty)$.
            
            Indeed, by finiteness of the partition sizes it suffices to show this with the partition $\{(\{1\}, \{2\})\}$. Observe that the number $U$ of selfing events in the ancestral line of the sample before the ancestral lines the sample experience an outcrossing event satisfies
            \begin{equation*}
                \PP\lp{U = r} = \alpha_N^r(1-\alpha_N).
            \end{equation*}
            At each selfing event, the two sample lineages have probability $\frac{1}{2}$ of coalescing. Therefore, the probability of not coalescing before an outcrossing event is
            \begin{equation}\label{E: coalescence_probability}
                \sum_{r = 0}^{\infty} \alpha_N^r(1-\alpha_N) 2^{-r} = \frac{1-\alpha_N}{1-\frac{\alpha_N}{2}}.
            \end{equation}
            By Assumption~\ref{A: timescale} we know that $\alpha_N \to 1$, which shows that the non-coalescence probability of Equation~\eqref{E: coalescence_probability} converges to $0$ as $N \to \infty$. That $c_N T$ converges to $0$ as $N \to \infty$ follows from Assumption~\ref{A:c_N} and the fact that the number of time-steps until a coalesce event or a dispersal is geometric with success probability $\frac{K_N}{N}\lp{\frac{\alpha_N}{2} + 1-\alpha_N}$. By a direct computation this is $O(c_N)$ in expectation. The $\frac{K_N}{N}$ term comes from the probability that the individual containing the two sample lineages is an offspring, $\frac{\alpha_N}{2}$ is the probability that a selfing event occurs and yields a coalescence, and $1-\alpha_N$ is the probability of seeing an outcrossing event.

            We have established thus far that $\Pi_G^{N,n}(t) = \bar{\chi}^{N,n}(t)$ for all $t\in(0,\infty)$ in which $S(\lfloor t c_N^{-1} \rfloor) \neq \lfloor t c_N^{-1} \rfloor$ with probability $1 - o(1)$. Further, we have shown that the measure of time $t$ on which $S(\lfloor t c_N^{-1} \rfloor) \neq \lfloor t c_N^{-1} \rfloor$ converges to $0$ in distribution as $N\to\infty$. Therefore we have convergence in probability to $0$ of Equation~\eqref{E: difference}, as needed.        
        \end{proof}

        We now proceed with the proof of Theorem~\ref{T: quenched_limited_outcrossing}.
        \begin{proof}
        Notice by Lemma~\ref{L: equivalence_EFC_Q_lambda}, it suffices to show weak convergence in law of $\mathcal{L}^{N,n}$ to $\mathcal{L}_{Q,\lambda}^n$. By Lemma~\ref{L: graph_convergence} we have that $G^{N,n}$ converges weakly in $\mathcal{D}\lp{\R_+, \Omega}$ to a $Q$-$\lambda$ graph $G$. By Lemma~\ref{L: phi_continuity} we know that the map $\Phi$ is continuous. Continuity preserves weak convergence 
        and so $\mathcal{L}_G^{N,n} = \Phi(G^N)$ converges weakly in law to $\mathcal{L}_{Q,\lambda}^n = \Phi(G)$. By Lemma~\ref{L: G_and_bar_same_limits}  $\mathcal{L}^{N,n}$ converges weakly in law to $\mathcal{L}_{Q,\lambda}^n$.
        \end{proof}

\section*{Acknowledgements}
    Research supported by National Science Foundation grants DMS-2534011, DMS-2532574, DMS-2348164 and DMS-2152103. We appreciate Adam Jakubowski for pointing out  references \cite{bk10, kouritzin2016} and for discussion about convergence determining functions for measures on the Skorokhod space.

\bibliographystyle{alpha}
\bibliography{main}


\appendix
\renewcommand{\thesection}{Appendix~\Alph{section}}

\section{Proof of the Theorem~\ref{T: annealed_limit}}\label{S: annealed_proof}

    We produce a lemma here to show that, for our notion of weak convergence in distribution that the intensity measure map sending a random measure $\mu$ to its intensity measure $\mathbb{E}\lb{\mu}$ is continuous.

    \begin{lemma}\label{L: continuity_intensity_measure}
        Let  $\lp{\mu_N}_{N \in \N}$ be a sequence of random variables taking value in $\mathcal{M}_1\lp{\mathcal{D}\lp{\R_+, E}}$, where $E$ is a locally compact Polish space. Suppose that $\mu_N \towd \mu$ in $\mathcal{M}_1\lp{\mathcal{D}\lp{\R_+, E}}$. Then the sequence of intensity measures $\mathbb{E}\lb{\mu_N}\toL \mathbb{E}\lb{\mu}$ in $\R$.
    \end{lemma}
    \begin{proof}
        Recall that for a random measure $\mu$ the intensity measure $\mathbb{E}\lb{\mu}$ is defined as the unique measure for which
        \begin{equation*}
           \mathbb{E}\lb{\mu}(f) :=  \mathbb{E}\lb{\mu(f)}
        \end{equation*}
        for any $f$ in $C_b(\mathcal{D}\lp{\R_+, E})$.

        We need to show that
        \begin{equation*}
            \mathbb{E}\lb{\mu_N(f)}=\mathbb{E}\lb{\mu_N}(f) \to \mathbb{E}\lb{\mu}(f) = \mathbb{E}\lb{\mu(f)}
        \end{equation*}
        for any $f$ in $C_b\lp{\mathcal{D}\lp{\R_+, E}}$. Fix any such $f$. The topology on $\mathcal{M}_1(\mathcal{D}\lp{\R_+, E})$ is such that the evaluation maps $T_f$ sending $\mu$ to $\mu(f)$ is continuous and bounded. As $\mu_N \towd \mu$ we have that 
        \begin{equation*}
            T_f(\mu_N) = \mu_N(f) \toL T_f(\mu) = \mu(f).
        \end{equation*}
        As $T_f$ is continuous and bounded, convergence in distribution implies convergence in expectation. This yields the claim.
    
    \end{proof}

    With Theorem~\ref{T: quenched_limited_outcrossing} in hand, we are ready to prove Theorem~\ref{T: annealed_limit} as a corollary.

    \begin{proof}[Proof of Theorem~\ref{T: annealed_limit}]
        Note that the assumptions of Theorem~\ref{T: annealed_limit} are the same as those of Theorem~\ref{T: quenched_limited_outcrossing}. Consequently, by Theorem~\ref{T: quenched_limited_outcrossing} we have that the sequence of random measures $\mathcal{L}^{N,n}$ converges weakly in law to $\mathcal{L}_\Pi^n$ where $\Pi$ is an EFC with coalescence measure $\Xi$ on $\Delta$, no non-binary fragmentation, and binary fragmentation with rate $\lambda$. By Lemma~\ref{L: continuity_intensity_measure}, the intensity measure mapping is continuous, and so $\mathbb{E}\lb{\mathcal{L}^{N,n}} = \PP_{\xi_0^n}\lp{\bar{\chi}^{N,n} \in \cdot}$ converges in distribution 
 to $\mathbb{E}\lb{\mathcal{L}_{\Pi}^n}$. It suffices, therefore, to show that $\mathbb{E}\lb{\mathcal{L}_{\Pi}^n}$ is the law of an $n$-$\Xi$-coalescent.

        Consider a fixed sample $m \leq n$ of particles located in the EFC at time $t$. Then, regardless of the position of the $m$ particles, when we anneal over the EFC the times at which these $m$ particles coalesce is governed by $Q_m$, the projection of the generator $Q$ on $\Enfty$ associated to $\Xi$ by Remark~\ref{R: xi_q_connection}. But this is precisely the definition of the generator of an $m$-$\Xi$-coalescent in its initial state. This gives the claim.
    \end{proof}

\section{Convergence criteria for the general model}

    To provide convergence criteria for our diploid exchangeable model with selfing and overlapping generations, we proceed along the lines of \cite{MohleSagitov2003} and \cite{birkner2018coalescent, abfw25}. We define the genetic contribution $\tilde{V}_i$ for the $i$th individual to count the number of genetic descendents among the $2N$ genes in the next time-step that can trace their lineage back to the $i$th individual. Formally, we have
    \begin{equation*}
        \tilde{V}_i = \sum_j V_{i,j} +  V_{i,i}.
    \end{equation*}
    Selfing counts twice as each individual contains two sample lineages. Denote by
    \begin{equation*}
        \tilde{V}_{(1)} \geq \tilde{V}_{(2)} \geq \ldots \geq \tilde{V}_{(N)}
    \end{equation*}
    the ranked version of the total offspring numbers $\lp{\tilde{V}_i}_{1 \leq i \leq N}$ and by
    \begin{equation*}
        \Phi_N := \mathfrak{L}\lp{\mathfrak{V}_N}
    \end{equation*}
    the law on $\Delta$ of the ranked offspring frequencies $\mathfrak{V}_N$ defined by
    \begin{equation}\label{Def:rankoffspring}
        \mathfrak{V}_N := \lp{\frac{\tilde{V}_{(1)}}{2N}, \frac{\tilde{V}_{(2)}}{2N}, \ldots, \frac{\tilde{V}_{(N)}}{2N}, 0, 0, \ldots}.
    \end{equation}

    \begin{remark}
        We recall here the transition rate $q^n(\xi,\eta) = \lambda_{b;k_1,\ldots,k_r;s }$ from a state $\xi$ to a state $\eta$ for an $n$-$\Xi$-coalescent, where $\xi$ consists of $b$ blocks and $\eta$ is of $(k_1,\ldots,k_r;s)$-type with respect to $\xi$. That is, $\eta$ is obtained from $\xi_0$ by keeping $s$ of the blocks the same and coalescing the remaining $b-s$ blocks into $r$ blocks made up of $k_1,\ldots,k_r$ blocks from $\xi$. By \cite[Equation 1.9]{birkner2018coalescent}
        \begin{align}\label{E: Xi_generator}
            q^n(\xi,\eta) = \lambda_{b;k_1,\ldots,k_r;s } &= \mathbbm{1}_{r=1;k_1=2}\Xi(\mathbf{0})\\
            &+ \int_{\Delta \setminus \{\mathbf
            0\}} \sum_{l=0}^s \sum_{i_1,\ldots,i_r\ \mathrm{distinct}} \binom{s}{l}
            \Big(\prod_{j=1}^r x_{i_j}^{k_j}\Big)\lp{\prod_{j = r+1}^l x_{i_j}}\,
            (1-|x|)^{s-l} \frac{\Xi(dx)}{\left\langle x,x \right\rangle}.
        \end{align}
    \end{remark}

    \begin{lemma}\label{L: ordered_offspring_distribution}
        Suppose that Assumption~\ref{A:c_N} holds and that, as $N \to \infty$, 
        \begin{equation}\label{E: offspring_dist_convergence}
            \frac{1}{2c_N}\Phi_N(dx) \to \frac{1}{\left\langle x,x \right\rangle} \Xi'(dx)
        \end{equation}
        vaguely on $\Delta \setminus\{\mathbf{0}\}$, where $\Xi'$ is a probability measure on $\Delta$ (i.e. $\Xi'(\mathbf{0})=1-\Xi'(\Delta)$). Then Assumption~\ref{A:Q_Nn} holds where, for each $n\geq 2$, $Q_n$ is the generator of an $n$-$\Xi$-coalescent with $\Xi = 2\Xi'$.
    \end{lemma}

        \begin{proof}
        Let $\mathit{H}_{N,n} = \lp{h_{\xi\eta}^{N,n}}_{\xi,\eta \in \En}$ denote the one-step transition matrix of the ancestral process after applying the haploid map $F_{\rm hap}$ of Equation~\eqref{E: hap_map} so that
        \begin{equation*}
            h_{\xi\eta}^{N,n} 
            = \PP\lp{F_{\rm hap}(\chi^{N,n}(1)) = \eta \,\mid\, \chi^{N,n}(0) = \xi}.
        \end{equation*}
        To prove the claim we need to show that $\frac{1}{c_N}\lp{\mathit{H}_{N,n}-I}$ converges to a matrix $Q_n$ that is the generator of an $n$-$\Xi$-coalescent with $\Xi = 2\Xi'$. 
        
        Without loss of generality, consider $\xi = \xi_0^b$ for some $1 \leq b \leq n$, and let $\eta$ be of $(k_1,\ldots,k_r;s)$-type. That is, $\eta$ is obtained from $\xi_0^b$ by keeping $s$ of the blocks as singletons and coalescing the remaining $b-s$ blocks into $r$ blocks of sizes $k_1,\ldots,k_r$. Fix $\varepsilon > 0$ and define $A_\varepsilon := \{x \in \Delta : \langle x,x\rangle > \varepsilon\}$, with complement $A_\varepsilon^{\mathrm c} := \Delta \setminus A_\varepsilon$. We decompose
        \begin{equation}\label{E: QNn_decomposition}
            Q_{N,n} := \frac{1}{c_N}\lp{\mathit{H}_{N,n}-I}
            = \Phi_N(A_\varepsilon)\,Q_{N,n}^\varepsilon + \Phi_N(A_\varepsilon^{\mathrm c})\,Q_{N,n}^{\varepsilon, \mathrm c},
        \end{equation}
        where $Q_{N,n}^{\varepsilon}$ and $Q_{N,n}^{\varepsilon,\mathrm c}$ are the conditional transition rate matrices given the events         
        $\{\mathfrak{V}_N \in A_\varepsilon\}$ and $\{\mathfrak{V}_N \in A_\varepsilon^{\mathrm c}\}$ respectively, where     $\mathfrak{V}_N$ is   the ranked offspring frequencies \eqref{Def:rankoffspring}. That is,
        \begin{align*}
            Q_{N,n}^{\varepsilon} &= \frac{1}{c_N}\lp{\PP\lp{F_{\rm hap}(\chi^{N,n}(1)) = \eta \,\mid\, \chi^{N,n}(0) = \xi, \mathfrak{V}_N \in A_\varepsilon} - \delta_{\xi\eta}} \text{ and }\\
            Q_{N,n}^{\varepsilon,c} &= \frac{1}{c_N}\lp{\PP\lp{F_{\rm hap}(\chi^{N,n}(1)) = \eta \,\mid\, \chi^{N,n}(0) = \xi, \mathfrak{V}_N \in A_\varepsilon^c} - \delta_{\xi\eta}}.
        \end{align*}
        
        By Assumption~\eqref{E: offspring_dist_convergence},
        \begin{equation*}
            \frac{\Phi_N(A_\varepsilon)}{2c_N}
            \to \int_{A_\varepsilon}\frac{1}{\langle x,x\rangle}\,\Xi'(dx),
        \end{equation*}
        and hence $\Phi_N(A_\varepsilon) \in O_{\varepsilon}(c_N)$ and $\Phi_N(A_\varepsilon^{\mathrm c}) \to 1$ as $N \to \infty$ by Assumption~\ref{A:c_N}. Here $O_{\varepsilon}(1)$ simply means that, the term is $O(1)$ for any fixed $\varepsilon > 0$.

        For $x \in A_\varepsilon$, denote by
        \begin{equation*}
            \Upsilon_{b;(k_1,\ldots,k_r;s)}(x)
            := \sum_{l=0}^s \sum_{i_1,\ldots,i_r\ \mathrm{distinct}} \binom{s}{l}
            \Big(\prod_{j=1}^r x_{i_j}^{k_j}\Big)\lp{\prod_{j = r+1}^l x_{i_j}}\,
            (1-|x|)^{s-l} \frac{\Xi(dx)}{\left\langle x,x \right\rangle}
        \end{equation*}
        the probability that $b$ blocks thrown uniformly on $[0,1]$ with a paintbox $x$ form a partition of type $k_1,\ldots,k_r;s$ (see \cite[Equation 1.9]{birkner2018coalescent}). 
        For any measure $\Theta$ on $\Delta$, the probability that $b$ blocks thrown uniformly at random and independently on $[0,1]$ have $r$ groups of size $k_1, \ldots, k_r$ falling into the same intervals and $s$ blocks falling into their own unique intervals (or else in excess of $\sum_i x_i$) for $x$ sampled from $\frac{1}{\left\langle x,x \right\rangle}\Theta$ is exactly
        \begin{equation}\label{E: paintbox_transition}
            \int_{\Delta \setminus \{\mathbf{0}\}} \frac{\Upsilon_{b;(k_1,\ldots,k_r;s)}(x)}{\left \langle x,x \right\rangle} d\Theta(x).
        \end{equation}
        That is, the $x$-paintbox transition probability for $x$ sampled from $\frac{1}{\left\langle x,x \right\rangle}\Theta$ is exactly \eqref{E: paintbox_transition}.

        We now show that
        \begin{equation}\label{E: epsilon_generator}
            \lim_{\varepsilon \to 0} \lim_{N \to \infty} \Phi_N(A_\varepsilon)Q_{N,n}^{\varepsilon}(\xi_0^b, \eta) = 2\int_{\Delta \setminus \{\mathbf{0}\}} \frac{\Upsilon_{b;(k_1,\ldots,k_r;s)}(x)}{\left \langle x,x \right\rangle} d\Xi'(x),
        \end{equation}
        which implies that $\lim_{\varepsilon \to 0} \lim_{N \to \infty}\Phi_N(A_\varepsilon)Q_{N,n}^{\varepsilon}$ exists and is the generator of an $n$-$\Xi$-coalescent with $\Xi = 2\Xi'-2\delta_{\mathbf{0}}\Xi'(\mathbf{0})$.
        Conditional on $x \neq \mathbf{0}$, we have
        \begin{equation*}
            \PP\lp{F_{\rm hap}(\chi^{N,n}(1)) = \eta \mid x, \chi^{N,n}(0) = \xi_0^b}
            = 2\Upsilon_{b;(k_1,\ldots,k_r;s)}(x)
            + \Delta_N(x),
            \qquad \text{with } |\Delta_N(x)| \le \frac{K_b}{N}.
        \end{equation*}
        That $|\Delta_N(x)|$ is $O(\frac{1}{N})$ follows simply from the observation that the probability of the $n \ll N\sum_i x_i$ sample lineages falling into each of the intervals of $x$ are asymptotically independent, with corrections at worst of order $\frac{1}{N}$, which can be uniformly controlled by compactness of $\Delta$ and finiteness of $b$.
        Therefore,
        \begin{equation*}
            \frac{1}{c_N}\mathbb{E}\lb{\mathbf 1_{A_\varepsilon}\,
            \PP\lp{F_{\rm hap}(\chi^{N,n}(1)) = \eta \mid x, \chi^{N,n}(0) = \xi_0^b}}
            = \frac{1}{c_N}\,
            \mathbb{E}\lb{\mathbf 1_{A_\varepsilon}\,2\Upsilon_{b;(k_1,\ldots,k_r;s)}(x)}
            + o(1),
        \end{equation*}
        since the $\frac{1}{N}$ coupling error is multiplied by $\Phi_N(A_\varepsilon)/c_N = O_\varepsilon(1)$. By Assumption~\eqref{E: offspring_dist_convergence}, vague convergence on $\Delta\setminus\{\mathbf{0}\}$ yields
        \begin{equation*}
            \lim_{N\to\infty}\frac{1}{c_N}\,
            \mathbb{E}\lb{\mathbf 1_{A_\varepsilon}\,\Upsilon_{b;(k_1,\ldots,k_r;s)}(x)}
            = 2\int_{A_\varepsilon}\frac{\Upsilon_{b;(k_1,\ldots,k_r;s)}(x)}{\langle x,x\rangle}\,\Xi'(dx),
        \end{equation*}
        and letting $\varepsilon \downarrow 0$ by monotone convergence gives
        \begin{equation*}
            2\int_{\Delta\setminus\{\mathbf{0}\}}\frac{\Upsilon_{b;(k_1,\ldots,k_r;s)}(x)}{\langle x,x\rangle}\,\Xi'(dx),
        \end{equation*}
        which coincides with the $x$-paintbox transition probability for $\Xi\!\restriction_{\Delta\setminus\{\mathbf{0}\}} = 2\Xi'\!\restriction_{\Delta\setminus\{\mathbf{0}\}}$. 
    
        We proceed to show now that
        \begin{equation}\label{E: epsilon_c_generator}
            \lim_{\varepsilon \to 0} \lim_{N \to \infty} Q_{N,n}^{\epsilon,c}(\xi_0^b, \eta) =
            \begin{cases}
                0 &, \xi \not\prec \eta\\
                2 &, \xi \prec \eta,
            \end{cases},
        \end{equation}
        where $\xi \prec \eta$ if $\eta$ is the result of a binary merger of blocks in $\xi$. That is, we will show $\lim_{\varepsilon \to 0} \lim_{N \to \infty} Q_{N,n}^{\epsilon,c}$ exists and is equal to the generator of an $n$-$\Xi$-coalescent with $\Xi = \Xi'(\mathbf{0})\delta_{\mathbf{0}}$. It suffices, therefore, to show that non-binary mergers occur with negligible probability and that binary mergers occur with rate $2$.
        For $x\in A_\varepsilon^{\mathrm c}$ we have $\langle x,x\rangle\le\varepsilon$, hence $\max_i x_i\le\|x\|_2\le\varepsilon^{1/2}$. Then
        \begin{equation*}
            \sum_i x_i^3 \;\le\; (\max_i x_i)\sum_i x_i^2 \;\le\; \varepsilon^{1/2}\sum_i x_i^2,
            \qquad
            \Big(\sum_i x_i^2\Big)^2 \;\le\; \varepsilon\,\sum_i x_i^2.
        \end{equation*}
        Let $M_b$ denote the event that, among the $b$ blocks thrown on $[0,1]$ that there is at least one non-binary merger from the $x$-paintbox. Then
        \begin{equation*}
            \PP(M_b\mid x)
            \;\le\; C_b\!\left(\sum_i x_i^3+\Big(\sum_i x_i^2\Big)^2\right)
            \;\le\; C_b\big(\varepsilon^{1/2}+\varepsilon\big)\sum_i x_i^2
            \;\le\; 2C_b\,\varepsilon^{1/2}\sum_i x_i^2
        \end{equation*}
        for some constant $C_b$ depending only on $b$. Therefore, using \(\displaystyle \lim_{N\to\infty}\frac{1}{2c_N}\mathbb{E}\!\big[\sum_i x_i^2\big]=\Xi'(\Delta\setminus{\{\mathbf{0}}\})\),
        \begin{equation*}
            \limsup_{N\to\infty}\frac{1}{2c_N}\,\mathbb{E}\!\big[\mathbf 1_{A_\varepsilon^{\mathrm c}}\PP(M_b\mid x)\big]
            \;\le\; 2C_b\,\varepsilon^{1/2}\,\limsup_{N\to\infty}\frac{1}{2c_N}\mathbb{E}\!\Big[\sum_i x_i^2\Big]
            \;=\; 2C_b \Xi'(\Delta\setminus \{\mathbf{0}\}) \,\varepsilon^{1/2}.
        \end{equation*}
        In particular,
        \begin{equation*}
            \lim_{\varepsilon\downarrow 0}\;\limsup_{N\to\infty}\frac{1}{c_N}\,\mathbb{E}\!\big[\mathbf 1_{A_\varepsilon^{\mathrm c}}\PP(M_b\mid x)\big]=0.
        \end{equation*}
        
        For any unordered pair $\{a,b\}$,
        \begin{equation*}
            \PP(\{a,b\}\text{ coalesce}\mid x)=\sum_i x_i^2+O\!\Big(\sum_i x_i^3\Big),
        \end{equation*}
        and on $A_\varepsilon^{\mathrm c}$ the $O(\sum_i x_i^3)$ error satisfies
        \begin{equation*}
            \frac{1}{c_N}\,\mathbb{E}\!\big[\mathbf 1_{A_\varepsilon^{\mathrm c}}\sum_i x_i^3\big]
            \;\le\; \varepsilon^{1/2}\,\frac{1}{c_N}\,\mathbb{E}\!\Big[\sum_i x_i^2\Big]
            \;\xrightarrow[N\to\infty]{}\; 2\Xi'(\Delta\setminus\{\mathbf{0}\}) \,\varepsilon^{1/2}.
        \end{equation*}
        Thus
        \begin{equation*}
            \frac{1}{c_N}\,\mathbb{E}\!\big[\mathbf 1_{A_\varepsilon^{\mathrm c}}\PP(\{a,b\}\text{ coalesce}\mid x)\big]
            =\frac{1}{c_N}\,\mathbb{E}\!\big[\mathbf 1_{A_\varepsilon^{\mathrm c}}\sum_i x_i^2\big] + o_\varepsilon(1)
            \end{equation*}
        with $o_\varepsilon(1)\to 0$ as $\varepsilon\downarrow 0$ uniformly in $N$. Decomposing,
        \begin{equation*}
            \frac{1}{c_N}\,\mathbb{E}\!\big[\mathbf 1_{A_\varepsilon^{\mathrm c}}\sum_i x_i^2\big]
            =2
            -\frac{1}{c_N}\,\mathbb{E}\!\big[\mathbf 1_{A_\varepsilon}\sum_i x_i^2\big].
        \end{equation*}
        By the normalization and the vague convergence assumption \eqref{E: offspring_dist_convergence}, as $N\to\infty$,
        \begin{equation*}
            \frac{1}{c_N}\,\mathbb{E}\!\Big[\sum_i x_i^2\Big]\;\to\;2\Xi'(\Delta\setminus\{\mathbf{0}\}),
            \qquad
            \frac{1}{c_N}\,\mathbb{E}\!\big[\mathbf 1_{A_\varepsilon}\sum_i x_i^2\big]
            \;\to\; 2\Xi'(A_\varepsilon),
        \end{equation*}
        hence
        \begin{equation*}
            \frac{1}{c_N}\,\mathbb{E}\!\big[\mathbf 1_{A_\varepsilon^{\mathrm c}}\sum_i x_i^2\big]
            \;\to\; 2-2\Xi'(A_\varepsilon)
            \xrightarrow[\varepsilon\downarrow 0]{}\;2\Xi'(\mathbf{0}).
        \end{equation*}

        Combining the contributions from $A_\varepsilon$ from \eqref{E: epsilon_generator} and $A_\varepsilon^{\mathrm c}$ from \eqref{E: epsilon_c_generator} with the decomposition \eqref{E: QNn_decomposition} yields
        \begin{align*}
            \lim_{N \to \infty} Q_{N,n}(\xi,\eta) &= \lim_{\varepsilon \to 0} \lim_{N \to \infty } \Phi_N(A_\varepsilon)\,Q_{N,n}^\varepsilon + \Phi_N(A_\varepsilon^{\mathrm c})\,Q_{N,n}^{\varepsilon, \mathrm c}\\
            &= \begin{cases}
                2\int_{\Delta\setminus\{\mathbf{0}\}}\frac{\Upsilon_{b;(2;b-2)}(x)}{\langle x,x\rangle}\,\Xi'(dx) + 2\Xi'(\mathbf{0}) &, \text{ if } \xi \prec \eta\\
                2\int_{\Delta\setminus\{\mathbf{0}\}}\frac{\Upsilon_{b;(k_1,\ldots,k_r;s)}(x)}{\langle x,x\rangle}\,\Xi'(dx) &, \text{ if } \xi \not \prec \eta
            \end{cases}.
        \end{align*}
        This is precisely the generator of an $n$-$\Xi$-coalescent with $\Xi = 2 \Xi'$ by comparing with \eqref{E: Xi_generator}.
    
    \end{proof}

    \begin{remark}\rm
\cite[Theorem 1.1]{birkner2018coalescent} establishes that,
        if  Assumption~\ref{A:c_N} and 
       \eqref{E: offspring_dist_convergence} hold, then the  rescaled ancestral process converges to the $\Xi$-coalescent governed by 
       $\Xi'\circ \varphi^{-1}$, where $\varphi$ denotes the halving map $\varphi: \Delta \to \Delta$ defined by
       \begin{equation*}
           \lp{x_1,x_2,\ldots} \mapsto \lp{\frac{1}{2}x_1, \frac{1}{2}x_1, \frac{1}{2}x_2, \frac{1}{2}x_2, \ldots}.
       \end{equation*}
Note that this measure  $\Xi'\circ \varphi^{-1}$ (which is under no selfing) is different from the measure $2\Xi'$ in the annealed convergence in  Theorem \ref{T: annealed_limit}. This difference is expected because when $\alpha_N \not\to 1$,  blocks of the ancestral process that enter the same individual may disperse before coalescing, and the annealed limit of the rescaled ancestral process 
will \textit{not} have $Q_n$ as its generator. The reason is that the haploid map $F_{\rm hap}$ in \eqref{E: hap_map}, which appears in the definition of $Q_n$, makes $Q_n$ the generator of the limiting process where two blocks that enter the same individual always coalesce instantaneously in the limit.

    \end{remark}
    
    We refer the  interested reader to \cite[Appendix A]{birkner2018coalescent} for several equivalent formulations to Equation~\eqref{E: offspring_dist_convergence}.

\section{Combinatorics for the general model}

Let $N$ be fixed and consider the general diploid overlapping-generations model described at the beginning of Section \ref{S:GeneralModel}. At time step $0$, let $K:=K_N^{(0)}$, $S:=\sum_{i=1}^N V^{(0)}_{i,i}$,
$v_i:=V^{(0)}_{i,i}$ and $u_i:=\sum_{j\neq i} V^{(0)}_{i,j}$ for $i\in[N]$.

\begin{lemma}[One-step triple coalescence]\label{lem:triple-one-step}
The probability that three distinct sampled lineages from three distinct individuals at time $0$
coalesce into a single ancestor in one step (to time $1$ in the past) is
\begin{align}\label{eq:triple-uncond}
\mathfrak{c}_3 
=
\frac{1}{\binom{N}{3}}\,&
\mathbb{E}\!\bigg[ 
\frac{N-K}{N}\sum_{i=1}^N\!\left(
\frac{1}{4}\binom{v_i}{2}
+\frac{1}{8}\,v_i u_i
+\frac{1}{16}\binom{u_i}{2}
\right) \notag\\
&\;+\;
\sum_{i=1}^N\!\left(
\frac{1}{8}\binom{v_i}{3}
+\frac{1}{16}\binom{v_i}{2}\,u_i
+\frac{1}{32}\,v_i\binom{u_i}{2}
+\frac{1}{64}\binom{u_i}{3}
\right)
\bigg].
\end{align}
Here the outer expectation is over the joint law of $(K_N^{(0)},V^{(0)})$;
$\alpha_N$ influences $\mathfrak{c}_3$ only through this law and does not appear explicitly in
\eqref{eq:triple-uncond}.
\end{lemma}

\begin{proof}
Sample uniformly without replacement three distinct individuals at time $0$; there are $\binom{N}{3}$
unordered choices. Condition first on the realized newborn set $B$ (of size $K$), the carried-over set
$~C$ (of size $N-K$), and the parentage matrix $V^{(0)}$. There are two ways to obtain a one-step triple merger:

\smallskip
\noindent
\emph{(i) Two newborns share a parent $i$ and the third sampled individual is that parent $i$ carried over.}
Fix $i\in C$. Among the $n_i:=v_i+u_i$ newborns that have $i$ as a parent, there are
$\binom{v_i}{2}$ pairs where both are selfed by $i$, $v_i u_i$ mixed pairs, and $\binom{u_i}{2}$ pairs
with both outcrossing via $i$. Tracing one step back, to coalesce with the gene copy traced from the
carried-over $i$, each newborn must choose the same copy in $i$:
probabilities $1/2$ (selfed) and $1/4$ (outcrossed), independently of the carried-over copy choice.
Hence the Mendelian factors are $1/4$, $1/8$, and $1/16$, respectively. Summing over $i\in C$ gives the
conditional contribution
\[
\frac{1}{\binom{N}{3}}\sum_{i\in C}\!\left(
\frac{1}{4}\binom{v_i}{2}
+\frac{1}{8}\,v_i u_i
+\frac{1}{16}\binom{u_i}{2}
\right).
\]

\smallskip
\noindent
\emph{(ii) Three newborns share a common parent $i$ (not necessarily carried over).}
For a fixed $i\in[N]$, the three newborns can be all selfed ($\binom{v_i}{3}$ triples),
two selfed and one outcrossed ($\binom{v_i}{2}u_i$ triples), one selfed and two outcrossed
($v_i\binom{u_i}{2}$ triples), or all outcrossed ($\binom{u_i}{3}$ triples). To coalesce one step back,
all three lineages must choose the same copy in $i$, yielding Mendelian factors $1/8$, $1/16$, $1/32$,
and $1/64$, respectively. Summing over $i$ gives the conditional contribution
\[
\frac{1}{\binom{N}{3}}\sum_{i=1}^N\!\left(
\frac{1}{8}\binom{v_i}{3}
+\frac{1}{16}\binom{v_i}{2}\,u_i
+\frac{1}{32}\,v_i\binom{u_i}{2}
+\frac{1}{64}\binom{u_i}{3}
\right).
\]

\smallskip
Adding the above two displays yields the conditional probability of triple coalescent given $V^{(0)},B,C$ as
\[
\frac{1}{\binom{N}{3}}
\left[
\sum_{i\in C}\!\left(
\frac{1}{4}\binom{v_i}{2}
+\frac{1}{8}\,v_i u_i
+\frac{1}{16}\binom{u_i}{2}
\right)
+
\sum_{i=1}^N\!\left(
\frac{1}{8}\binom{v_i}{3}
+\frac{1}{16}\binom{v_i}{2}\,u_i
+\frac{1}{32}\,v_i\binom{u_i}{2}
+\frac{1}{64}\binom{u_i}{3}
\right)
\right].
\]
Now average over the uniformly random carried-over set $C$ of size $N-K$, and 
finally, take expectation with respect to $(K_N^{(0)},V^{(0)})$ to obtain \eqref{eq:triple-uncond}.
\end{proof}

        By $\xi \prec \eta$ we denote that $\eta$ is obtainable from $\xi$ via a single binary merger of blocks of $\xi$. In the following lemma we calculate the one-step transition probabilities for non-binary mergers, and the one-step exit probability for a state $\xi$.

        \begin{lemma}\label{L: one_step_combinatorics_no_triple}
            Suppose that $\mathfrak{c}_3 \in o(c_N)$ as $N \to \infty$. Then for any $\xi\in\En$, it holds that $\mathit{h}_{\xi\xi}^{N,n} = 1 - 2\binom{|\xi|}{2}c_N + o(c_N)$ and that $\mathit{h}_{\xi\eta}^{N,n} \in o(c_N)$ for any $\eta$ that cannot be obtained from $\xi$ via a single binary merger, i.e. where $\xi \not \prec \eta$ 
        \end{lemma}

        \begin{proof}
            Note that there are $\binom{|\xi|}{2}$ possibly binary mergers for any given $\xi$, and each is equally likely by exchangeability. Therefore, it suffices to show that the exit probability after factoring through the haploid map $F_{\rm hap}$ for any $\xi$ is $2c_N\binom{|\xi|}{2} + o(c_N)$. This will follow by finiteness of the state space if for any $\eta$ not obtainable by a single binary merger from $\xi$ that the one-step transition probability from $\xi$ to $\eta$ is $o(c_N)$, which we show below. We proceed by a monotonicity argument.

            Let $r$ denote the largest number of blocks of $\xi$ that are coalesced together into a single block of $\eta$. Suppose to start that $r \geq 3$ after factoring through the haploid map. For any subset of size $3$ of these $r$ blocks, of which there are $\binom{r}{3}$, the probability of that given subset coalescing in a single time-step is $\mathfrak{c}_3$. Therefore the one-step transition probability $h_{\xi\eta}^{N,n}$ that $\eta$ is obtained from $\xi$ is at most $\binom{r}{3}\mathfrak{c}_3$, which is $o(c_N)$ by assumption. This shows that any $\eta$ obtained by coalescing $r \geq 3$ blocks of $\xi$ into a single block happens with $o(c_N)$ probability.

            Suppose then that $r = 2$ and that $\xi \not\prec \eta$. Then there are at least four blocks $C_1, C_2, C_3, C_4$ of $\xi$ coalesced as pairs $C_1 \cup C_2, C_3 \cup C_4$ in $\eta$. We denote this transition by $\PP\lp{2,2\to1,1}$. Note that any pairwise coalescence event, conditional on $V$, is bounded above by
            \begin{equation}
                \frac{1}{4N^2}\sum_{i=1}^N (V_i)_2.
            \end{equation}
            Therefore, where we condition on there being no triple merger, we have that
            \begin{equation*}
                \PP\lp{2,2\to1,1 \,\mid\, V, \text{ no triple}} \leq \mathbb{E}\lb{\lp{\frac{1}{4N^2}\sum_{i=1}^N(V_i)_2}^2}.
            \end{equation*}
            By averaging over $V$ we then have
            \begin{equation*}
                \PP\lp{2,2\to1,1} \leq \binom{4}{3}\mathfrak{c}_3 + \mathbb{E}\lb{\frac{1}{16N^4}\lp{\sum_{i=1}^N (V_i)_2}^2}.
            \end{equation*}
            The second summand is $O(c_N^2) \subset o(c_N)$ by a direct comparison, and by assumption $\mathfrak{c}_3$ is $o(c_N)$. This shows that when $r = 2$ with $\xi \not\prec \eta$ that $h_{\xi\eta}^{N,n} \in o(c_N)$.

            Summing over the single-pair coalescence probabilities among the $\binom{|\xi|}{2}$ unordered pairs yields the exit rate $2\binom{|\xi|}{2}c_N + o(c_N)$, which completes the proof.

        \end{proof}

\end{document}